\documentclass[10pt]{amsart}
\usepackage{amsmath, a4wide}
\usepackage{amssymb}
\usepackage{verbatim}
\usepackage{xcolor}
\usepackage{subfigure}
\usepackage{epsfig}

\newtheorem{lemma}{Lemma}
\newtheorem{prop}[lemma]{Proposition}

\newtheorem{theorem}[lemma]{Theorem}

\newtheorem{example}{Example}

\newtheorem{definition}[lemma]{Definition}

\newcommand{\lfD}{\tiny{\Box}} 
\newcommand{\mcs}{\tau}

\newcommand{\ssupp}{\operatorname{sing\, supp}}

\newcommand{\cF}{\mathcal{F}}

\newcommand{\DCO}{\mathcal{D}(\Omega)}

\newcommand{\DCOpr}{\mathcal{D}'(\Omega)}

\newcommand{\N}{\mathbb{N}}    
\newcommand{\R}{\mathbb{R}}    
\newcommand{\Z}{\mathbb{Z}}    
\newcommand{\cC}{\mathcal{C}}

\newcommand{\wf}{\operatorname{WF}}
\newcommand{\bs}{\backslash}
\newcommand{\wh}{\widehat}

\newcommand{\supp}{\operatorname{supp}}
\newcommand{\setsp}{\;:\;}     
\newcommand{\la}{\langle}
\newcommand{\ra}{\rangle}
\newcommand{\dirac}{\delta}   

\newcommand{\bhcut}[1]{} 

\numberwithin{equation}{section}
\numberwithin{figure}{section}

\begin{document}

\title[Sobolev wavefront set \& local holder smoothness by shearlets]{Microlocal analysis and characterization of Sobolev wavefront sets using shearlets}
\author{Bin Han}

\address{Department of Mathematical and Statistical Sciences,
	University of Alberta, Edmonton,\quad Alberta, Canada T6G 2G1.
	\quad {\tt bhan@ualberta.ca}\quad
	{\tt http://www.ualberta.ca/$\sim$bhan}
}

\author{Swaraj Paul}

\address{Discipline of Mathematics, Indian Institute of Technology Indore, Simrol, Indore 453 552, India. \quad {\tt swaraj.lie@gmail.com}\quad {\tt nirajshukla@iiti.ac.in}
}

\author{Niraj K. Shukla}

\thanks{Research of B. Han was partially supported by the Natural Sciences and Engineering Research Council of Canada (NSERC). Research of S. Paul and N. K. Shukla was supported by a research grant from CSIR, New Delhi [25(0280)/18/EMR-II]}

\makeatletter \@addtoreset{equation}{section} \makeatother

\begin{abstract}
Sobolev wavefront sets and $2$-microlocal spaces play a key role in describing and analyzing the singularities of distributions in microlocal analysis and solutions of partial differential equations.
Employing the continuous shearlet transform to Sobolev spaces,
in this paper we characterize the microlocal Sobolev wavefront sets,
the $2$-microlocal spaces, and local H\"older spaces of distributions/functions.
We then establish the connections among Sobolev wavefront sets, $2$-microlocal spaces, and local H\"older spaces through the continuous shearlet transform.
\end{abstract}

\keywords{Continuous shearlet transform, Sobolev wavefront sets, H\"{o}lder smoothness, pseudo-differential operators, $2$-microlocal spaces, continuous wavelet transform}

\subjclass[2010]{42C40, 42C15} \maketitle

\pagenumbering{arabic}


\section{Introduction and main results}

The continuous wavelet transform (CWT) has many applications such as signal and image processing for providing time-scale analysis of various types of data and functions.
Recall that the \emph{continuous wavelet transform} of a function $f\in L^2(\R)$ is defined to be
\begin{equation} \label{cwt}
\tilde{f}(a,b):=
\int_{\mathbb{R}}|a|^{-\frac{1}{2}}\overline{
	 \psi\left(\tfrac{x-b}{a}\right)}f(x)dx, \qquad a\in \R\bs\{0\}, b\in \R,
\end{equation}
where $\psi\in L^2(\R)$ is \emph{an admissible wavelet}
satisfying $C_\psi:=\int_\R \frac{|\wh{\psi}(\xi)|^2}{|\xi|} d\xi<\infty$. Here
$\wh{f}(\xi):=\int_{\R} f(x) e^{-2\pi i x\xi} dx$ for $\xi\in \R$ is the Fourier transform of $f\in L^1(\R)$ and is extended to tempered distributions.
The function $f$ can be recovered from its CWT through
$$
f(x)=\frac{1}{C_{\psi}}\int_{\mathbb{R}}\int_{\mathbb{R}}\tilde{f}(a,b) \psi\left(\tfrac{x-b}{a}\right)\frac{dadb}{a^2},
$$
with convergence of the integral in the weak sense.
Discretizing the continuous wavelet transform  by restricting $(a,b)$ in \eqref{cwt} to a discrete subset of $\R^2$, the (discretized) wavelets and their associated discrete wavelet expansions are often used to effectively and sparsely represent functions having point-like discontinuity/singularity, thanks to the good time-frequency localization and high vanishing moments of the underlying wavelet function $\psi$.
In multiple dimensions, singularities of discontinuity across a curve such as edges in an image often hold the key information of a multidimensional function/signal or the solution of a partial differential equation.
Though discrete wavelet expansions using classical wavelets are known to be optimal for capturing point-like singularities,
using isotropic dilations of wavelet functions, they are not that effective for capturing curve-like singularities.
Motivated by the ineffectiveness of wavelets for curve-like singularities,
\emph{curvelets} using rotation in Cand\'es and Donoho \cite{candes2004new} and \emph{shearlets} using shear transform in \cite{guo2006sparse}, whose underlying systems are discrete affine systems, provide an optimally sparse approximation of a cartoon-type image with discontinuity/singularity across a $C^2$ smooth curve.

For a nonempty open subset $\Omega\subseteq \R^d$, the test function space $\DCO$ consists of all compactly supported $C^\infty$ functions $\varphi: \Omega \rightarrow \mathbb{C}$ such that $\supp(\varphi)\subseteq \Omega$. Then its dual space $\DCOpr$ consists of all distributions in $\Omega$.
Our main goal of this paper is to characterize the Sobolev wavefront set of a distribution $f\in \DCOpr$
and the $2$-microlocal space of a function
using the continuous shearlet transform for Sobolev spaces, and then to explore their connections with local H\"{o}lder regularity. Before presenting the main results of this paper,
we first briefly review the
continuous shearlet transform and Sobolev wavefront sets of a distribution.

\subsection{Continuous shearlet transform}\label{shearlet}

Shearlets are related to wavelets with composite dilations, which were introduced in \cite{guo2006wavelets}.
Using shear transform to preserve the rectangular grid/lattice structure, shearlets have been introduced in \cite{guo2006sparse,labate2005sparse}.
Let the (horizontal) shearlet $\psi\in L^2(\R^2)$ be given by
\begin{equation}\label{psidef}
	 \wh{\psi}(\xi)=\wh{\psi}(\xi_1,\xi_2):=\widehat{\psi_1}(\xi_1)\widehat{\psi_2}\left({\xi_2}/{\xi_1}\right),\qquad \mbox{for}\; \xi=(\xi_1,\xi_2)\in\mathbb{R}^2,\ \xi_1\neq 0,
\end{equation}
where the one-dimensional functions $\psi_1$ and $\psi_2$ satisfy the following conditions:
\begin{itemize}
	\item[(C1)] The function $\psi_1\in L^2(\mathbb{R})$ satisfies the Calder\'{o}n condition
	$\int_0^{\infty} |\widehat{\psi_1}(a)|^2\frac{da}{a}=
	\int_0^{\infty} |\widehat{\psi_1}(-a)|^2 \frac{da}{a}=
	1$, which is equivalent to
	 \begin{equation}\label{psi1:admissible}
		\int_0^{\infty} |\wh{\psi_1}(a\xi)|^2 \frac{da}{a}=1,\qquad \mbox{for all}\ \xi\in \R\bs\{0\}
	\end{equation}
	and $\widehat{\psi_1}\in \mathcal{D}(\mathbb{R})$ with $\supp(\widehat{\psi_1}) \subseteq [-2,-\frac{1}{2}]\cup [\frac{1}{2},2].$
	
	\item[(C2)] $\|\wh{\psi_2}\|_{L^2}=1$ and $\widehat{\psi_2}\in \mathcal{D}(\mathbb{R})$ with $\supp(\widehat{\psi_2}) \subseteq [-1,1]$.
\end{itemize}
If a shearlet $\psi$ is defined through \eqref{psidef}, then throughout the paper we always assume that $\psi_1$ and $\psi_2$ satisfy the conditions in (C1) and (C2).
Define the shear matrix $S_s$, the parabolic scaling matrix $D_a$ and the matrix $M_{as}$ by
\begin{equation}\label{mas}
	 S_s:=\begin{bmatrix}1&-s\\0&1\end{bmatrix},\qquad D_a:=\begin{bmatrix}a&0\\0&\sqrt{a}\end{bmatrix},\qquad
	M_{as}:=S_s D_a=
	\begin{bmatrix} a&-\sqrt{a} s\\0&\sqrt{a}\end{bmatrix}.
\end{equation}
%
%
We shall use the following notation for shearlet elements:
\begin{equation}\label{psiast}
	\psi_{ast}:=a^{-{3}/{4}}\psi
	\left(M_{as}^{-1}(\cdot-t)\right)
	=a^{-{3}/{4}}\psi
	 \left(D_{a}^{-1}S_s^{-1}(\cdot-t)\right).
\end{equation}
Then $\{\psi_{ast} \setsp a\in (0,1],\ s\in [-2,2], t\in\mathbb{R}^2\}$
is called a (horizontal) \emph{continuous shearlet system} for
$L^2(\mathcal{C})^{\vee}:=\{ f\in L^2(\R) \setsp \supp(\wh{f}) \subseteq \mathcal{C}\}$,
where $\mathcal{C}$ is the horizontal cone given by
\begin{equation}\label{hcone}
	\mathcal{C}:=\left\{(\xi_1,\xi_2)\in \mathbb{R}^2 \setsp |\xi_1|\geq 2\quad \mbox{and}\quad |{\xi_2}/{\xi_1}|\leq 1\right\}.
\end{equation}
The vertical continuous shearlet system is obtained by switching the role of the horizontal and vertical axes. More precisely, the (vertical) shearlet $\psi^{(v)}:=\psi(J\cdot)$ and
$$
\psi^{(v)}_{ast}:=a^{-3/4} \psi^{(v)}(JM_{as}^{-1} J(\cdot-t))
=a^{-3/4}\psi(M_{as}^{-1} J(\cdot-t))
=\psi_{as(Jt)}(J\cdot),
$$
where $J:\R^2\rightarrow \R^2$ is the matrix mapping $(x,y)$ to $(y,x)$.
The (horizontal) \emph{continuous shearlet transform}  (CST) and
the (vertical) CST of a function $f$ on $\R^2$ is defined by
\begin{equation*}
	 \mathcal{SH}_{\psi}f(a,s,t):=\left\langle f,\psi_{ast}\right\rangle,\qquad
	 \mathcal{SH}^{(v)}_{\psi}f(a,s,t):=\langle f,\psi^{(v)}_{ast}\rangle,
\end{equation*}
for $a\in (0,1]$, $s\in [-2,2]$, and $t\in\mathbb{R}^2$.
Note that $\mathcal{SH}^{(v)}_\psi f(a,s,t)=\mathcal{SH}_\psi (f(J\cdot))(a,s,Jt)$.
Due to the conditions in (C1) and (C2) on $\psi_1$ and $\psi_2$, the horizontal and vertical shearlets are bandlimited.
See Figure~\ref{fig:testfig} for the frequency support of the shearlets for different values of $a$ and $s$.

\begin{figure}[tbhp]
	\centering \subfigure[]{\label{fig:a1}\includegraphics[width=6cm,height=4cm]{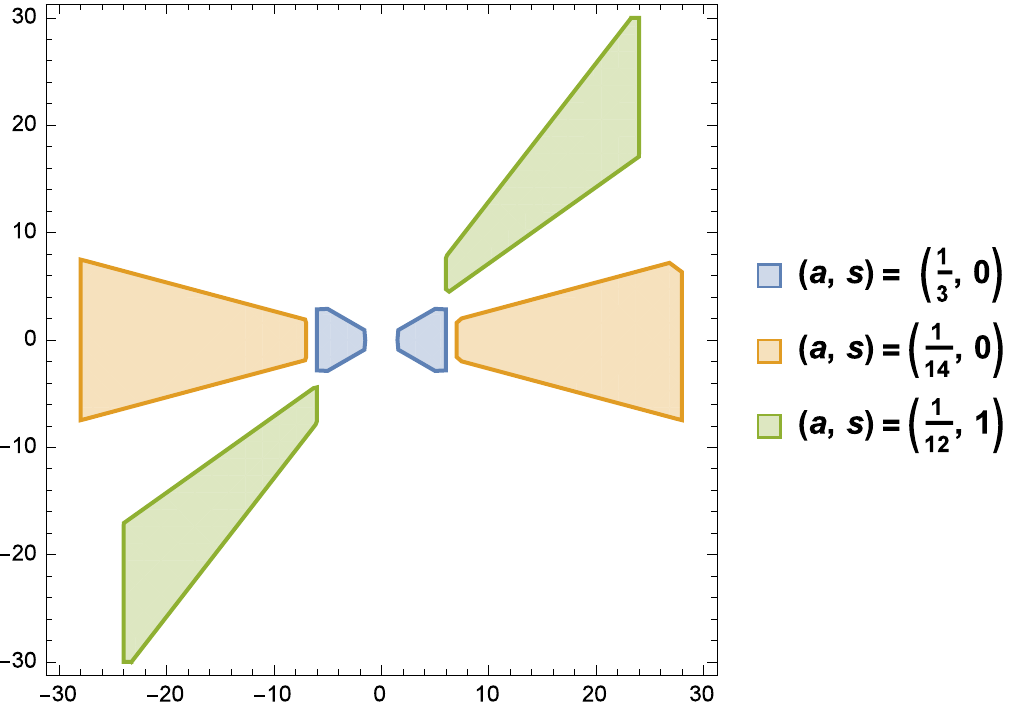}} \quad \subfigure[]{\label{fig:b1}\includegraphics[width=6.5cm,height=4cm]{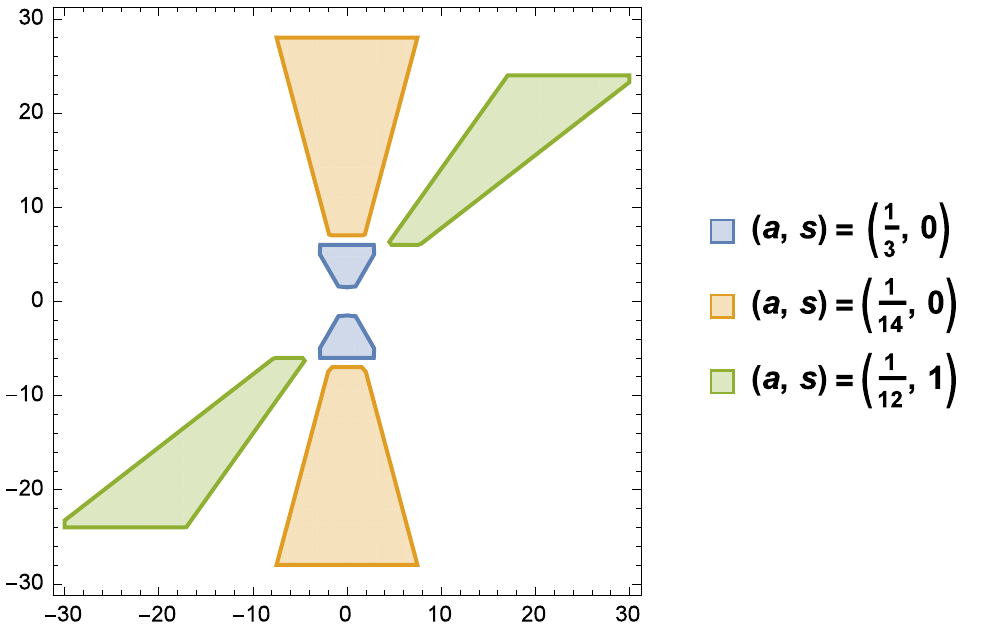}}
	\caption{The left-hand side is the support of
		$\wh{\psi_{ast}}$ for
		horizontal shearlet elements.
		Note that $|\wh{\psi_{ast}}|=|\wh{\psi_{as0}}|$.
		The right-hand side is the support of $\wh{\psi^{(v)}_{ast}}$
		for vertical shearlet elements with different values of $a$ and $s$.} \label{fig:testfig}
\end{figure}

Shearlets are particularly useful in representing anisotropic functions due to their properties of an affine-like system of well-localized waveforms at various scales, locations, and orientations.
Candes and Donoho \cite{candes2005continuous} introduced the curvelet transform to resolve the wavefront set of a distribution (see
Subsection~\ref{subsec:wfs} below). Kutyniok and Labate \cite{kutyniok2009resolution} showed that the CST also characterizes the wavefront set of a distribution.

\subsection{Sobolev wavefront sets}\label{subsec:wfs}

In microlocal analysis, the singularities of a distribution $f\in \DCOpr$ can be described through the wavefront set of $f$. Nowadays wavefront sets play an important role to represent time-machine space-times, quantum energy inequalities and cosmological models \cite{brouder2014smooth}.

(Sobolev) wavefront sets provide more refined description of the singular support of a distribution through localizing the distribution in the spatial domain and then using the properties of the Fourier transform.
They are frequently used in the study of regularities/singularities of solutions of PDEs \cite{hormander2003analysis}.
We now recall the singular support, the $C^{\infty}$-wavefront set, and the Sobolev  wavefront set of a distribution.

Recall that $\DCOpr$, as the dual space of the test function space $\DCO$, is the space of all distributions on a nonempty open subset $\Omega \subseteq \R^2$.
Let $f\in \DCOpr$ be a distribution.
A point $x_0\in \Omega$ is not in the \emph{singular support} of $f$, denoted by $\ssupp(f)$, if  there exists $\varphi\in \DCO$ such that $\varphi=1$ in a neighborhood of $x_0$ and $\varphi f\in C^\infty(\Omega)$. Obviously, $\ssupp(f)\subseteq \supp(f)$.
Roughly speaking, $\ssupp(f)$ is the complement of the set of points in $\Omega$ where $f$ is microlocally $C^\infty$.
In particular, if $x_0\not \in \ssupp(f)$, then $\varphi f\in \DCO$ and hence
$\widehat{\varphi f}$ must decay rapidly at infinity:
\begin{equation}\label{Fourierdecay}
	\sup_{\xi\in \R^2} (1+|\xi|)^N |\widehat{\varphi f}(\xi)|<\infty,\qquad \forall\ N\in \N_0:=\N\cup\{0\}.
\end{equation}
In other words, if $x_0\in \ssupp(f)$, then the above fast decay property in \eqref{Fourierdecay} fails.
However, it may happen that the fast decay property in \eqref{Fourierdecay} may fail only in certain directions.
We shall use $\xi_0\in \R^2\bs\{0\}$ for a possible direction where \eqref{Fourierdecay} may hold in an open cone with direction $\xi_0$.
For $\xi_0\in \R^2\bs\{0\}$ and $0<\epsilon<\pi/2$, an open cone $V_{\xi_0,\epsilon}$ with its vertex at the origin and pointing to the direction $\xi_0$ is defined to be
\begin{equation}\label{V}
	V_{\xi_0,\epsilon}:=\{ (r\cos \theta, r\sin \theta)  \setsp r>0, |\theta-\theta_0|<\epsilon\} \quad \mbox{with}\quad
	\theta_0:=\mbox{Arg}(\xi_0),
\end{equation}
where $\mbox{Arg}(\xi_0)$ is the principal angle of the direction $\xi_0$ (by identifying $\xi_0$ as a complex number).
The wavefront set of a distribution $f\in \DCOpr$ describes $\ssupp(f)$ by telling us in which direction the decay property in \eqref{Fourierdecay} fails.

\begin{definition}\label{cinfinitywavefront}
	($C^{\infty}$ wavefront set)  Let $f\in \DCOpr$ and $x_0 \in \ssupp(f)$. For $\xi_0\in \R^2\bs\{0\}$,
	we say that $(x_0,\xi_0)$ is not in the ($C^\infty$) wavefront set of $f$, if there exist $\varphi\in \DCO$ and $0<\epsilon<\pi/2$ such that $\varphi=1$ in a neighborhood of $x_0$ and
	 \begin{equation}\label{cinfinitydefn}
		\sup_{\xi \in V_{\xi_0,\epsilon}} (1+|\xi|)^N |\widehat{\varphi f}(\xi)|<\infty, \qquad \mbox{for all}\quad N\in \N_0.
	\end{equation}
	The ($C^{\infty}$) wavefront set of a distribution $f\in \DCOpr$ is denoted by $\wf(f)$.
\end{definition}

To check whether $(x_0,\xi_0)$ belongs to $\wf(f)$, we have to check infinitely many conditions in \eqref{cinfinitydefn}.
Using a different notion of regularity other than $C^{\infty}$, the Sobolev wavefront set is effective to describe the singularity  behavior of solutions of PDEs along different directions \cite{hormander1997lectures,petersen1983introduction}.  Recall that
$f\in H^m(\mathbb{R}^2)$ with $m\in\mathbb{R}$ if
$\|f\|_{H^m(\R^2)}^2:=\int_{\R^2} (1+|\xi|^2)^m |\wh{f}(\xi)|^2 d\xi<\infty$.
We now recall the definition of the Sobolev wavefront set of a distribution in \cite{hormander1997lectures}.

\begin{definition}\label{sobolevwf} (Sobolev wavefront set)
	Let $f\in \DCOpr$ and $x_0 \in \ssupp(f)$.
	For $\xi_0\in \R^2 \bs\{0\}$ and $m \in \R$, we say that $(x_0,\xi_0)$ is not in the  Sobolev wavefront set of order $m$ of the distribution $f$
	if there exist $\varphi\in \DCO$ and $0<\epsilon<\pi/2$
	such that $\varphi=1$ in a neighborhood of $x_0$ and
	 \begin{equation}\label{sobolevwavefront}
		 \int_{V_{\xi_0,\epsilon}}\left(1+|\xi|^2\right)^m\left|\widehat{\varphi f}(\xi)\right|^2 d\xi<\infty.
	\end{equation}
	The Sobolev wavefront set of order $m$ of a distribution $f\in \DCOpr$ is
	denoted by $\wf_m(f)$.
	In other words, $(x_0,\xi_0)\not \in \wf_m(f)$ if and only if $f$ is microlocally in the Sobolev space $H^{m}$ at $(x_0,\xi_0)$, written as $f\in H^m(x_0,\xi_0)$.
\end{definition}

Since the cone $V_{\xi_0,\epsilon}$ in \eqref{V} is independent of $|\xi|$, $(x_0,\xi_0)\in \wf_m(f)$ if and only if $(x_0, \xi_0/|\xi_0|)\in \wf_m(f)$.
If $f$ is a real-valued tempered distribution, then $\wh{f}(-\xi)=\overline{\wh{f}(\xi)}$ and consequently, $(x_0,\xi_0)\in \wf_m(f)$ if
and only if $(x_0,-\xi_0)\in \wf_m(f)$. Note that the shearlet $\psi$ in \eqref{psidef} is often real-valued and has support symmetric about the origin.

\subsection{Main results on Sobolev wavefront sets}

We now state our main results
Theorem~\ref{thm:swf:integral} and Theorem~\ref{thm:swf:decay}
characterizing Sobolev wavefront sets using  the continuous shearlet transform.
To state our main results, we define
$B_{r_0}(x_0):=\{t\in \R^2 \setsp |t-x_0|<r_0\}$, which is the open disk with center $x_0\in \R^2$ and radius $r_0>0$.

The following result characterizes the Sobolev wavefront set of a distribution through the square integrability of the continuous shearlet transform coefficients.

\begin{theorem}\label{thm:swf:integral}
	Let $m\in \R$ and $\Omega\subseteq \R^2$ be a nonempty open subset of $\R^2$.
	Consider the horizontal and vertical shearlets defined in the Subsection~\ref{shearlet}.
	Let  $\xi_0\in \R^2\bs\{0\}$ and define $s_0:=\tan(\mbox{Arg}(\xi_0))\in \R\cup\{\infty\}$, i.e., $s_0$ is the slope of the direction $\xi_0$. For a real-valued distribution $f\in \DCOpr$ and $x_0\in \ssupp(f)$,
	the two points $(x_0, \pm \xi_0)$
	are not in the Sobolev wavefront set of the distribution $f$, i.e., $(x_0,\pm \xi_0)\not\in\wf_m(f)$ if and only if there exists $r_0>0$ such that
	\begin{align}
		&\int_{a=0}^{a=1}
		 \int_{s=s_0-r_0}^{s=s_0+r_0}\int_{ t\in B_{r_0}(x_0)}
		 \left|\mathcal{SH}_{\psi}f(a,s,t)\right|^2
		a^{-2m-3}dt ds da<\infty\quad \mbox{if}\quad |s_0|<2, \label{squareintegrabilityofshearlet}\\
		&\int_{a=0}^{a=1}
		 \int_{s=\frac{1}{s_0}-r_0}^{s=\frac{1}{s_0}+r_0}
		\int_{t\in B_{r_0}(x_0)}
		 \left|\mathcal{SH}_{\psi}^{(v)}f(a,s,t)\right|^2a^{-2m-3}dt dsda<\infty\quad \mbox{if}\quad |s_0|>1/2.\label{1squareintegrabilityofshearlet}
	\end{align}
\end{theorem}

Due to the relation
$\mathcal{SH}^{(v)}_\psi f(a,s,t)=\mathcal{SH}_\psi (f(J\cdot))(a,s,Jt)$, for every $0<|s_0|<\infty$ and sufficiently small $r_0>0$,
\eqref{squareintegrabilityofshearlet} is finite if and only if \eqref{1squareintegrabilityofshearlet} is finite. In Theorem~\ref{thm:swf:integral}, if $f\in \mathcal{D}'(\Omega)$ cannot be regarded as a tempered distribution on $\R^2$, then it is understood that $f$ in both \eqref{squareintegrabilityofshearlet} and \eqref{1squareintegrabilityofshearlet} should be replaced by $\varphi f$, where $\varphi\in \mathcal{D}(\Omega)$ takes value $1$ in a neighborhood of $x_0$.

For $m\in \R$ and $f\in \DCOpr$, we say that $\mathcal{SH}_\psi f(a,s,t)=\mathcal{O}(a^k)$ as $a\to 0^+$ uniformly when $(s,t)$ is near $(s_0,x_0)$
if there exist $r_0>0$ and $C>0$ such that
\begin{equation}\label{cst:decay}
	|\mathcal{SH}_\psi f(a,s,t)|\le C a^{k}\qquad \forall\, a\in (0,r_0), |s-s_0|<r_0, |t-x_0|<r_0.
\end{equation}

Using Theorem~\ref{thm:swf:integral},
the following result characterizes the Sobolev wavefront set of a distribution through the decay property of CST coefficients as in \eqref{cst:decay}.

\begin{theorem}\label{thm:swf:decay}
	Let $m\in \R$ and $\Omega\subseteq \R^2$ be a nonempty open subset. For a real-valued distribution $f\in \DCOpr$,
	 \begin{equation}\label{wfm:cst:decay}
		[\Lambda_1(m+1+\epsilon)\cup  \Lambda_2(m+1+\epsilon)] \subseteq \wf_m(f)^c,\qquad \forall\, m\in \R, \epsilon>0
	\end{equation}
	and
	\begin{equation} \label{wfm:cst:decay:2}
		\wf(f)^c=\Lambda_1(\infty) \cup \Lambda_2(\infty),
	\end{equation}
	where the sets $\Lambda_1(m)$ and $\Lambda_2(m)$ are defined as follows:
	\begin{align*}
		\Lambda_1(m):=\Big\{ (x_0, (r, rs_0)) \setsp & x_0\in \Omega, s_0 \in (-2,2), r\in \R\bs\{0\},\;
		 |\mathcal{SH}_{\psi}f(a,s,t)|=\mathcal{O}(a^m)\\ &\mbox{as $a \to 0^+$ uniformly when $(s,t)$ is near $(s_0, x_0)$}\Big\},\\
		\Lambda_2(m):=\Big\{ (x_0, (rs_0, r)) \setsp & x_0\in \Omega, s_0 \in (-2,2), r\in \R\bs\{0\},\;
		 |\mathcal{SH}^{(v)}_{\psi}f(a,s,t)|=\mathcal{O}(a^m)\\ &\mbox{as $a \to 0^+$ uniformly when $(s,t)$ is near $(s_0, x_0)$}\Big\},
	\end{align*}
	and $\Lambda_1(\infty):=(\cap_{k\in \N} \Lambda_1(k))^{\mathrm{o}}$
	and $\Lambda_2(\infty):=(\cap_{k\in \N} \Lambda_2(k))^{\mathrm{o}}$, where $E^{\mathrm{o}}$ stands for the interior of a set $E$.
\end{theorem}

The proofs of Theorems~\ref{thm:swf:integral} and~\ref{thm:swf:decay} will be given in
Section~\ref{sec:swf:proof}.
The identity in \eqref{wfm:cst:decay:2} characterizing $C^{\infty}$ wavefront sets through $\Lambda_1(\infty) \cup \Lambda_2(\infty)$ was already given by Kutyniok and Labate \cite[Theorem~5.1]{kutyniok2009resolution}.
Our approach of characterizing $C^{\infty}$ wavefront sets in \eqref{wfm:cst:decay:2} is different by taking advantage of the identity in \eqref{wfm:cst:decay} characterizing  Sobolev wavefront sets.
Theorem~\ref{thm:swf:decay} shows
that the continuous shearlet transform resolves the Sobolev wavefront set precisely as $a\to 0^+$ asymptotically.
The consistency of the result in Theorem~\ref{thm:swf:decay} will be verified through Proposition~\ref{sobolevineq}.

\subsection{Main results on $2$-microlocal spaces}

The local pointwise H\"older exponent is a key concept in multifractal analysis and signal processing. Recall that

\begin{definition}
	A function $f$ on $\R^d$ belongs to a local H\"older space $C^{\mcs}_{x_0}(\R^d)$ with $\mcs>0$ and $x_0\in \R^d$ if there exist a positive constant $C$ and a polynomial $P_{x_0}$ of (total) degree less than $\mcs$ such that for all $x$ in a neighborhood of $x_0$,
	\begin{equation}\label{pointholder}
		|f(x)-P_{x_0}(x)|\leq C \|x-x_0\|^{\mcs}.
	\end{equation}
\end{definition}

Using wavelet transform,
Jaffard \cite{jaffard1991pointwise}
characterized pointwise and uniform H\"older regularities, which can be also studied by
curvelets and Hart Smith transforms in \cite{nualtong2005analysis,sampo2009estimations},
as well as by continuous and discrete shearlet transforms in \cite{lakhonchai2010shearlet}.
But the local H\"older exponent fails to fully characterize regularity of a function/distribution, e.g., the stability theory through the action of pseudo-differential operators to a function/distribution. To overcome such difficulty,
generalizing local H\"older spaces,
Jaffard \cite{jaffard1991pointwise} (also see \cite{jaffardmeyer1996}) introduced
$2$-microlocal spaces below.


Let $\varphi$ be a Schwartz function on $\R^d$ such that $\wh{\varphi}(\xi)=1$ for $\xi\in B_{1/2}(0)$ and $\supp(\wh{\varphi})\subseteq B_1(0)$.
Define $\varphi_j=2^{j+1}\varphi (2^{j+1}\cdot)-2^{j}\varphi(2^j\cdot)$ for $j \in \N_0$. Then the Littlewood-Paley decomposition of a tempered distribution $f$ is a set of tempered distributions $\{ S_0f,\Delta_j f\}_{j\in \N_0}$, where $S_0f:=\varphi*f$ and $\Delta_j f:=\varphi_j*f$. Note that $f=S_0 f+\sum_{j=0}^\infty \Delta_j f$, due to $\wh{\varphi}(\xi)+\sum_{j=0}^\infty \wh{\varphi_j}(\xi)=1$ for all $\xi\in \R^d$.

\begin{definition}\label{def:mcs}
	Let $x_0\in \mathbb{R}^d$ and $\mcs,\mcs'\in \R$. A tempered distribution $f$ on $\R^d$ belongs to the $2$-microlocal space $C_{x_0}^{\mcs,\mcs'}(\R^d)$,
	if there exists a positive constant $C$ such that for all $x\in \R^d$ and $j\in \N_0$,
	\[
	|S_0f(x)| \leq C (1+\|x-x_0\|)^{-\mcs'}
	\quad \mbox{and}\quad  |\Delta_j f(x)|\leq C 2^{-j\mcs} (1+2^j\|x-x_0\|)^{-\mcs'}.
	\]
\end{definition}

A 2-microlocal space is an efficient tool to study the regularity/singularity in areas such as theory of semi-linear hyperbolic partial differential equations and image processing, is stable under the action of pseudo-differential operators, and combines the local and pointwise H\"older regularities in a single condition.
Jaffard \cite{jaffard1991pointwise} first described the connection between pointwise H\"older smoothness at a point and a $2$-microlocal space. He also gave an equivalent characterization of a $2$-microlocal space by means of the behavior of the (discrete or continuous) wavelet coefficients.
We now state and prove some necessary and sufficient conditions on a $2$-microlocal space in terms of the behavior of the continuous shearlet transform. This helps us to derive connections between $2$-microlocal spaces and Sobolev wavefront sets of a distribution.

\begin{theorem}\label{thm:2micro}
	Let $x_0\in \mathbb{R}^2$ and  $\mcs,\mcs'\in \R$ with $\mcs+\mcs'>0$ and $\mcs'<0$ such that $\mcs+\mcs'\not\in \N$.
	Then the following statements hold:
	\begin{enumerate}
		\item[(i)] If $f\in C_{x_0}^{\mcs,\mcs'}(\R^2)$, then
		there exists a positive constant $C$ such that the horizontal and vertical continuous shearlet transforms satisfy
		{\small
			 \begin{equation}\label{necessary2micro}
				 \max\left(|\mathcal{SH}_{\psi}f(a,s,t)|, |\mathcal{SH}^{(v)}_{\psi}f(a,s,t)|\right)\leq C a^{\frac{3}{4}+\frac{\mcs+\lfloor \mcs\rfloor}{2}}\left(1+\left\|\tfrac{t-x_0}{\sqrt{a}}\right\|^{-\mcs'}\right),
			\end{equation}
		}
		for all $a\in(0,1)$, $s\in[-2,2]$ and $t\in \mathbb R^2$, where $\lfloor \mcs\rfloor$ is the largest integer $\le \mcs$.
		
		\item[(ii)] For a tempered distribution $f$ on $\R^2$, if there exists $C>0$ such that
		{\small
			 \begin{equation}\label{sufficient2micro1}
				 \max\left(|\mathcal{SH}_{\psi}f(a,s,t)|, |\mathcal{SH}^{(v)}_{\psi}f(a,s,t)|\right)\leq
				C a^{\frac{5}{4}+\mcs}\left(1+\left\|\tfrac{t-x_0}{\sqrt{a}}\right\|^{-2\mcs'}\right),
			\end{equation}
		}
		for all $a\in(0,1)$, $s\in[-2,2]$ and $t\in \mathbb R^2$, then $f\in C_{x_0}^{\mcs,\mcs'}(\R^2)$ and $f\in C_{x_0}^\mcs(\R^2)$.
	\end{enumerate}
\end{theorem}

As an application of Theorem~\ref{thm:2micro},
we connect $2$-microlocal spaces and local H\"older spaces with Sobolev wavefront sets.

\begin{theorem}\label{thm:2microsobolevwavefront}
	Let $x_0\in \mathbb{R}^2$ and  $\mcs,\mcs'\in \R$ with $\mcs+\mcs'>0$ and $\mcs'<0$ such that $\mcs+\mcs'\not\in \N$.
	If $f\in C^{\mcs,\mcs'}_{x_0}(\R^2)$, then $(x_0,\xi_0)\not\in \wf_m(f)$ for all $\xi_0\in \R^2 \bs\{0\}$ and for all real numbers $m<\frac{\mcs+\mcs'+\lfloor \mcs\rfloor}{2}-\frac{1}{4}$.
\end{theorem}

We shall prove Theorems~\ref{thm:2micro} and \ref{thm:2microsobolevwavefront} on $2$-microlocal spaces in
Section~\ref{sec:holdersobolev}.
The continuous shearlet transform employed in
Theorems~\ref{thm:swf:integral} and~\ref{thm:2micro} allows us to  connect $2$-microlocal spaces and local H\"older spaces with Sobolev wavefront sets in Theorem~\ref{thm:2microsobolevwavefront}.

\subsection{Organization}
The structure of the paper is as follows. In Section~\ref{sec:wavefrontset}, we review some basic properties of $C^\infty$ and Sobolev  wavefront sets.
To prove the main results Theorems~\ref{thm:swf:integral} and~\ref{thm:swf:decay},
we shall develop in Section~\ref{CSTHm} the continuous shearlet transform in the Sobolev space  $H^m(\mathbb{R}^2)$ with $m\in\mathbb{R}$,
including tempered distributions.  In particular, some microlocal properties of the continuous shearlet transform are also developed in this section in order to prove our main results in Theorems~\ref{thm:swf:integral} and~\ref{thm:swf:decay}, whose proofs are given in Section~\ref{sec:swf:proof}.
In Section~\ref{sec:holdersobolev}, we shall further study $2$-microlocal spaces and local H\"older spaces. Then we shall prove Theorems~\ref{thm:2micro} and~\ref{thm:2microsobolevwavefront} in Section~\ref{sec:holdersobolev}.
We complete this paper by providing an example of a distribution with discontinuity along a curve such that its continuous shearlet transform has no rapid decay along some directions.

\section{Some properties of Sobolev wavefront sets}
\label{sec:wavefrontset}

In this section, we shall discuss some basic properties of the  wavefront set of a distribution. From the definitions of $C^\infty$ and Sobolev wavefront sets in Definitions~\ref{cinfinitywavefront} and~\ref{sobolevwf}, it is trivial that if $(x_0,\xi_0)\in \wf(f)$ (or in $\wf_m(f)$), then $(x_0,r\xi_0)\in \wf(f)$ (or in $\wf_m(f)$) for all $r>0$. The known consistency theory on wavefront sets shows that the definitions for $(x_0,\xi_0)$ outside $\wf(f)$ or $\wf_m(f)$ are independent of the choice of $\varphi\in \DCO$ and $\epsilon>0$ as long as
$\epsilon$ and $\supp(\varphi)$ are sufficiently small with $\varphi=1$ near $x_0$.
As shown in the following known result,
the condition~\eqref{sobolevwavefront} given in Definition~\ref{sobolevwf} for Sobolev wavefront sets can be replaced by a simpler but equivalent condition.

\begin{lemma}\label{lem:swf:def}
	The condition in \eqref{sobolevwavefront} of Definition~\ref{sobolevwf} holds if and only if
	\begin{equation}\label{sobolevwf:2}
		\int_{U_{\xi_0,\epsilon}} |\xi|^{2m} |\wh{\varphi f}(\xi)|^2 d\xi<\infty,
	\end{equation}
	where
	\begin{equation}\label{Wxie}
		U_{\xi_0,\epsilon}:=\{ \xi\in V_{\xi_0,\epsilon} \setsp |\xi|>1\}=
		\{(r\cos\theta,r\sin \theta) \setsp r>1, |\theta-\mbox{Arg}(\xi_0)|<\epsilon\}.
	\end{equation}
\end{lemma}

\begin{lemma}\label{lem:sobolev:wfm}
	Let $f\in H^m(\R^2)$. Then
	\begin{equation}\label{I}
		 \int_{a=0}^{a=1}\int_{s=-2}^{s=2}
		\int_{t\in \R^2} |\mathcal{SH}_\psi f(a,s,t)|^2 a^{-2m-3} dtdsda<\infty
	\end{equation}
	and
	\[
	\int_{a=0}^{a=1}\int_{s=-2}^{s=2}
	\int_{t\in \R^2} |\mathcal{SH}^{(v)}_\psi f(a,s,t)|^2 a^{-2m-3} dtdsda<\infty.
	\]
\end{lemma}

\begin{proof}
	By the definition of $\psi_{ast}$ in \eqref{psiast}, we observe
	\begin{equation*}\label{sh:Fourier}
		\mathcal{SH}_\psi f(a,s,t)
		=\la f, \psi_{ast}\ra
		=\la \wh{f}, \wh{\psi_{ast}}\ra
		 =a^{3/4}(\wh{f}\overline{\wh{\psi}(M_{as}^T\cdot)})^\vee(t).
	\end{equation*}
	Let $I$ stand for the integral in \eqref{I}. By Plancherel's identity, we have
	\[
	I=\int_{\xi\in \R^2} |\wh{f}(\xi)|^2
	g(\xi) d\xi
	\quad \mbox{with}\quad
	 g(\xi):=\int_{a=0}^{a=1}\int_{s=-2}^{s=2}
	|\wh{\psi}(M_{as}^T\xi)|^2 a^{-2m-3/2} ds da.
	\]
	We now estimate $g(\xi)$. Note that $M_{as}^T\xi=(a\xi_1, \sqrt{a} \xi_2-s\sqrt{a}\xi_1)$. By the definition of $\wh{\psi}$ in \eqref{psidef},
	\[
	 \wh{\psi}(M_{as}^T\xi)=\wh{\psi_1}(a\xi_1)
	 \wh{\psi_2}(a^{-1/2}(\tfrac{\xi_2}{\xi_1}-s)).
	\]
	For simplicity of presentation, we assume $\xi_1>0$ and $\xi_2>0$.
	Using change of variables $a'=a\xi_1$ and $s'=a^{-1/2}(\tfrac{\xi_2}{\xi_1}-s)$, we deduce that $a=\frac{a'}{\xi_1}$, $s=\frac{\xi_2}{\xi_1}-\sqrt{\frac{a'}{\xi_1}}s'$, and
	\begin{align*}
		g(\xi)&=\int_{a'=0}^{a'=\xi_1}
		\int_{s'=\sqrt{\frac{\xi_1}{a'}}
			(\frac{\xi_2}{\xi_1}-2)}
		^{s'=\sqrt{\frac{\xi_1}{a'}}
			(\frac{\xi_2}{\xi_1}+2)}
		 |\wh{\psi_1}(a')\wh{\psi_2}(s')|^2
		 \left(\frac{\xi_1}{a'}\right)^{2m+3/2}
		\frac{(a')^{1/2}}{\xi_1^{3/2}}
		da' ds'\\
		&=|\xi_1|^{2m}
		\int_{a'=0}^{a'=\xi_1}
		\int_{s'=\sqrt{\frac{\xi_1}{a'}}
			(\frac{\xi_2}{\xi_1}-2)}
		^{s'=\sqrt{\frac{\xi_1}{a'}}
			(\frac{\xi_2}{\xi_1}+2)}
		 |\wh{\psi_1}(a')\wh{\psi_2}(s')|^2
		(a')^{-2m-1}da' ds'.
	\end{align*}
	If $m\ge 0$, since $\supp(\wh{\psi_1})\subseteq [-2,-\frac{1}{2}]\cup [\frac{1}{2},2]$, we have
	\[
	g(\xi)\le
	|\xi_1|^{2m}
	\int_{a'=1/2}^{a'=2}
	\int_{s'\in \R}
	|\wh{\psi_1}(a')\wh{\psi_2}(s')|^2
	(a')^{-2m-1}da' ds'
	\le C_m(1+\|\xi\|^2)^{m},
	\]
	where $C_m:=\|\wh{\psi_2}\|_{L^2}^2 \int_{a'=1/2}^{a'=2} \frac{|\wh{\psi_1}(a')|^2}{(a')^{2m+1}} da'<\infty$.
	
	If $m<0$, since $\supp(\wh{\psi_1})\subseteq [-2,-\frac{1}{2}]\cup [\frac{1}{2},2]$, we also have
	\[
	g(\xi)
	\le \|\wh{\psi_2}\|_{L^2}^2
	|\xi_1|^{2m}
	\int_{a'=1/2}^{a'=\xi_1} |\wh{\psi_1}(a')|^2 (a')^{-2m-1} da'\le C |\xi_1|^{2m} \chi_{[1/2,\infty)}(|\xi_1|),
	\]
	where $C:=\|\wh{\psi_2}\|_{L^2}^2
	\int_{a'=1/2}^{a'=2}
	|\wh{\psi_1}(a')|^2 (a')^{-2m-1} da'<\infty$.
	Note that $\supp(\wh{\psi}(M_{as}^T\cdot))\subseteq \Xi(a,s)$, where
	\begin{equation}\label{xias}
		\Xi(a,s):=\left\{(\xi_1, \xi_2) \setsp |\xi_1| \in \left[\tfrac{1}{2a},\tfrac{2}{a}\right], \left|\tfrac{\xi_2}{\xi_1}-s\right|\leq \sqrt{a}\right\}.
	\end{equation}
	Hence, for $|\xi_2/\xi_1|>3$, we trivially have $g(\xi)=0$. Hence, for $|\xi_2/\xi_1|\le 3$ and $m<0$, we conclude that
	\[
	g(\xi)\le C |\xi_1|^{2m} \chi_E(\xi)\le C_m (1+\|\xi\|^2)^{m},
	\]
	where $E:=\{(x_1,x_2) \setsp |\xi_1|\ge 1/2, |\xi_2/\xi_1|\le 3\}$ and
{$C_m:=C \sup_{\xi\in E} (\frac{1+\|\xi\|^2}{|\xi_1|^2})^{m}\le 14^{m}C$}.
	Now we conclude that
	\[
	I\le \int_{\xi\in \R^2} |\wh{f}(\xi)|^2 g(\xi) d\xi
	\le C_m \int_{\R^2} |\wh{f}(\xi)|^2(1+\|\xi\|^2)^m d\xi<\infty.
	\]
	This completes the proof of \eqref{I}. The second inequality can be proved similarly.
\end{proof}

We provide two examples to illustrate $C^\infty$ and Sobolev wavefront sets, and their characterization in Theorems~\ref{thm:swf:integral} and~\ref{thm:swf:decay} using the continuous shearlet transform.

\begin{example}\label{example 1}
	{\rm
		Let $\dirac$ be the Dirac delta distribution:
		$\dirac(\varphi)=\varphi(0)$ for $\varphi\in \mathcal{D}(\R^2)$.
		For $\varphi \in \mathcal{D}(\R^2)$, we have
		$\varphi \dirac=\varphi(0) \dirac$ and hence
		$\widehat{\varphi \dirac}(\xi)=\varphi(0)$ for all $\xi \in \R^2$. Consequently, it is trivial to conclude that $\supp(\dirac)=\ssupp(\dirac)=\{0\}$ and $\wf(\dirac)=\{0\}\times (\R^2 \bs \{0\})$.
		Consider the direction $\xi_0=(\cos\theta_0,\sin \theta_0)$ with $\theta_0\in \R$ and arbitrarily small $\epsilon>0$.
		For $\varphi\in \mathcal{D}(\R^2)$ such that $\varphi=1$ in a neighborhood of $0$,
		noting $\widehat{\varphi \dirac}=\varphi(0)=1$ and using the polar coordinates, we have
		\[
		\int_{U_{\xi_0,\epsilon}} (1+|\xi|^2)^m |\widehat{\varphi \dirac}(\xi)|^2 d\xi=
		 \int_{\theta_0-\epsilon}^{\theta_0+\epsilon}
		\int_1^\infty (1+r^2)^m r drd\theta=
		2\epsilon \int_1^\infty (1+r^2)^m r dr,
		\]
		which is finite if and only if $m<-1$.
		Thus, $\wf_m(\dirac)=\emptyset$ for $m<-1$, and  $\wf_m(\dirac)=\{0\}\times (\R^2\bs\{0\})$ for $m\ge -1$. This is not surprising since $\dirac \in H^m(\R^2)$ for all $m<-1$.
		
		We now apply Theorem~\ref{thm:swf:integral} to calculate $\wf_m(\delta)$ instead.
		By definition, we have
		\[
		 \mathcal{SH}_{\psi}\delta(a,s,t)=\langle \delta, \psi_{ast}\rangle=\overline{\psi_{ast}(0)}
		 =a^{-3/4}\overline{\psi(-M_{as}^{-1}t)}.
		\]
		Let $\xi_0\in \R^2 \backslash\{0\}$ and $s_0:=\tan(\mbox{Arg}(\xi_0))$.
		For $r_0>0$ and 
$|s_0|<2$, we have
		\begin{align*}
			\int_{0}^{1} \int_{s_0-r_0}^{s_0+r_0} &\int_{B_{r_0}(0)}
			 |\mathcal{SH}_{\psi}\delta(a,s,t)|^2 a^{-2m-3} dt ds da\\
			&=\int_{0}^{1} \int_{s_0-r_0}^{s_0+r_0} \int_{B_{r_0}(0)} |\psi(-M_{as}^{-1}t)|^2 a^{-2m-9/2} dtdsda
			=\int_{0}^{1}
			a^{-2m-3} \eta(a) da
		\end{align*}
		with
		 $\eta(a):=\int_{s_0-r_0}^{s_0+r_0} \int_{-M_{as}^{-1} B_{r_0}(0)}
		|\psi(t)|^2 dt ds$. It is noted that $ |\psi(-M_{as}^{-1}t)|\sim O(a^N)$ for all $N\in\N$ and $t\notin B_{r_0}(0)$ since $\psi(t)$ decays rapidly as $|t|\rightarrow \infty$. Hence the above integral is always finite for any $m$. Since $\psi$ is a Schwartz function, we have $\eta(a) \le 2r_0 \|\psi\|_{L^2}^2$ for all $a\in (0,1)$. For $m<-1$, we have
		\[
		\int_{0}^{1}
		a^{-2m-3} \eta(a) da \le 2r_0 \|\psi\|_{L^2}^2  \int_0^1 a^{-2m-3} da=\frac{r_0 \|\psi\|_{L^2}^2}{2(-m-1)}<\infty.
		\]
		The case  $|s_0|>1/2$ can be handled similarly for the vertical shearlets. By Theorem~\ref{thm:swf:integral}, $(0, \pm \xi_0)\not\in \wf_m(\delta)$ for all $m<-1$.
		Hence, $\wf_m(\delta)=\emptyset$ for $m<-1$.
		For $m\ge -1$, we observe that $\lim_{a\to 0^+} \eta(a)
		=\int_{s_0-r_0}^{s_0+r_0} \int_{\R^2}
		|\psi(t)|^2 dtds=2r_0\|\psi\|_{L^2}^2
		>0$. Since
		$\int_0^{\epsilon} a^{-2m-3} da=\infty$ for all $m\ge -1$ and $0<\epsilon\le 1$, we conclude that for $m\ge -1$,
		\[
		\int_{0}^{1} \int_{s_0-r_0}^{s_0+r_0} \int_{B_{r_0}(0)}
		 |\mathcal{SH}_{\psi}\delta(a,s,t)|^2 a^{-2m-3} dt ds da=\infty.
		\]
		By Theorem~\ref{thm:swf:integral}, $(0, \pm \xi_0) \in \wf_m(\delta)$ for all $m\ge -1$. Therefore,
		$\wf_m(\dirac)=\{0\}\times (\R^2\bs\{0\})$ for $m\ge -1$. Hence, we obtained the same conclusion using both the definition of Sobolev wavefront sets in Definition~\ref{sobolevwf} and Theorem~\ref{thm:swf:integral}.
		By Theorem~\ref{thm:swf:decay}, we also conclude that $\wf(\dirac)=\{0\}\times (\R^2\bs\{0\})$.
}\end{example}

To present the next example, we need an auxiliary result.

\begin{lemma}\label{lem:g}
	Let $g\in L^2(\R)$ with $\|g\|_{L^2}>0$ and $\theta_0\in \R$ such that $\cos \theta_0=0$. Then there exists a positive constant $C$ such that
	\begin{equation}\label{est:g}
		\frac{C^{-1}}{r}\le \int_{\theta_0-\epsilon}^{\theta_0+\epsilon}
		|\wh{g}(r\cos\theta)|^2 d\theta
		\le \frac{C}{r}
	\end{equation}
	for all sufficiently large $r$ and for all $0<\epsilon\le \pi/4$.
\end{lemma}

\begin{proof}
	Using the substitution $u=r\cos\theta$, we have $du=-r\sin \theta d\theta$ and $\theta(u)=\arccos (u/r)$.
	Note that the interval $[\theta_0-\epsilon,\theta_0+\epsilon]$ is mapped into the interval $I_{r,\theta_0}$ which is the line segment connecting the points $r\cos(\theta_0-\epsilon)$ and $r\cos(\theta_0+\epsilon)$.
	Hence,
	\[
	 \int_{\theta_0-\epsilon}^{\theta_0+\epsilon}
	|\widehat{g}(r\cos\theta)|^2  d\theta
	=\int_{I_{r,\theta_0}}
	\frac{|\wh{g}(u)|^2}{r|\sin \theta(u)|}du.
	\]
	Noting that $\inf_{|\theta-\theta_0|\le \epsilon}|\sin \theta|\ge 2^{-1/2}$ due to $\cos\theta_0=0$ and $0<\epsilon\le \pi/4$, we have
	\begin{equation*}\label{est:extra}
		\frac{1}{r} \int_{I_{r,\theta_0}} |\wh{g}(u)|^2 du \le
		 \int_{\theta_0-\epsilon}^{\theta_0+\epsilon}
		|\widehat{g}(r\cos\theta)|^2  d\theta
		\le \frac{\sqrt{2}}{r} \int_{I_{r,\theta_0}} |\wh{g}(u)|^2 du.
	\end{equation*}
	Noting that $\lim_{r\to \infty} \int_{I_{r,\theta_0}} |\wh{g}(u)|^2 du=
	\int_{\R} |\wh{g}(u)|^2 du>0$, we conclude that there exists a positive constant $C$ such that \eqref{est:g} holds for all sufficiently large $r$.
\end{proof}

\begin{example}\label{example 2}
	{\rm Let $f:=\chi_{[0,1]^2}$ be the characteristic function of the unit square $[0,1]^2$ in $\R^2$.
		Then $\supp(f)=[0,1]^2$ and $\ssupp(f)=\partial [0,1]^2:=(\{0,1\}\times [0,1])\cup ([0,1]\times \{0,1\})$, the boundary of $[0,1]^2$.
		Due to symmetry, we only consider $(\tau,0)$ with $0\le \tau<1$. Define $c:=\min(\tau,1-\tau)/4>0$ for $0<\tau<1$, and $c:=1/8$ for $\tau=0$.
		Take
		$$
		\varphi_1\in \mathcal{D}(\R),\quad
		\varphi_1=1\quad \mbox{on}\quad (-c,c),\quad \mbox{and}
		\quad \supp(\varphi_1)\subseteq (-2c,2c).
		$$
		Now we choose
		 $\varphi(x,y)=\varphi_1(x-\tau)\varphi_1(y)$.
		We first consider $0<\tau<1$.
		Then $(\varphi f)(x,y)=\varphi_1(x-\tau) \varphi_2(y)$ with $\varphi_2:=\varphi_1\chi_{[0,\infty)}$.
		Since $\varphi_1\in \mathcal{D}(\R)$, $\widehat{\varphi_1}$ has fast decay. On the other hand, noting $\varphi_2=1$ on $[0,c)$ and using integration by parts,
		we observe
		 $\widehat{\varphi_2}(\xi)=\frac{1+\widehat{\varphi_2'}(\xi)}{i\xi}$.
		Since $\varphi_2'\in \mathcal{D}(\R)$,  $\widehat{\varphi_2'}$ must have fast decay
		with $\widehat{\varphi_2'}(0)=\int_{\R} \varphi_2'(x) dx=-1$. Hence, there exists a positive constant $C>0$ such that $|\widehat{\varphi_2}(\xi)|^2\le \frac{2C}{1+|\xi|^2}$ for all $\xi\in \R$ and
		$|\widehat{\varphi_2}(\xi)|^2\ge \frac{C}{1+|\xi|^2}$ for sufficiently large $|\xi|$. Consider $\xi_0=(\cos \theta_0, \sin \theta_0)$ with $\theta_0\in \R$.
		Using the polar coordinates and the behavior of $\widehat{\varphi_2}(\xi)$ at infinity,
		we see that the integral
		$\int_{U_{\xi_0,\epsilon}} (1+|\xi|^2)^m |\widehat{\varphi_1}(\xi_1)|^2 |\widehat{\varphi_2}(\xi_2)|^2 d\xi_1d\xi_2$
		is finite if and only if
		 \begin{equation}\label{wf:square}
			\int_1^\infty
			 \int_{\theta_0-\epsilon}^{\theta_0+\epsilon}
			(1+r^2)^m r \frac{|\widehat{\varphi_1}(r\cos\theta)|^2}
			{1+r^2 \sin^2 \theta} d\theta dr
		\end{equation}
		is finite. We first consider $\cos \theta_0\ne 0$. We can choose $\epsilon>0$ small enough so that $\inf_{|\theta-\theta_0|\le \epsilon} |\cos \theta|>0$.
		Since $\varphi_1$ is a Schwartz function, for every $k\in \N$, there exists a positive constant $C_k$ such that $|\wh{\varphi_1}(r\cos\theta)|\le C_k (1+r^2)^{-k}$ for all $r\ge 1$ and $|\theta-\theta_0|\le \epsilon$.
		The above integral in \eqref{wf:square} must be finite for all $m\in \R$ by choosing $k \ge m+1$.
		Now we consider $\cos \theta_0=0$. Choose $0<\epsilon<\pi/4$ and we have
		\begin{equation}\label{est:sin}
			1+\tfrac{1}{2}r^2\le 1+r^2\sin^2\theta\le 1+r^2,\qquad \forall\; r>0,\, \theta\in [\theta_0-\epsilon, \theta_0+\epsilon].
		\end{equation}
		Hence, the above integral in \eqref{wf:square} is finite if and only if
		 \begin{equation}\label{wf:square1}
			\int_1^\infty
			(1+r^2)^{m-1} r
			 \left(\int_{\theta_0-\epsilon}^{\theta_0+\epsilon}
			 |\widehat{\varphi_1}(r\cos\theta)|^2  d\theta\right) dr<\infty.
		\end{equation}
		Since $\varphi_1\in L^2(\R)$ with $\|\varphi_1\|_{L^2}>0$, by Lemma~\ref{lem:g},
		the integral in \eqref{wf:square1} is finite if and only if $\int_1^\infty
		(1+r^2)^{m-1} dr<\infty$, which holds if and only if $m<1/2$.
		
		At the corner point $(0,0)$ with $\tau=0$,
		we have $(\varphi f)(x,y)=\varphi_2(x) \varphi_2(y)$ and similar analysis can be performed by replacing $\varphi_1$ with $\varphi_2$ in \eqref{wf:square}.
		Hence, \eqref{wf:square} becomes
		\[
		\int_1^\infty \int_{\theta_0-\epsilon}^{\theta_0+\epsilon} \frac{(1+r^2)^m r}
		{(1+r^2 \cos^2\theta)(1+r^2 \sin^2 \theta)} d\theta dr.
		\]
		If $\cos\theta_0\ne 0$ and $\sin \theta_0\ne 0$,
		then the above integral is finite if and only if
		$\int_1^r (1+r^2)^{m-2} r dr<\infty$, which holds if and only if $m<1$.
		If either $\cos\theta_0=0$ or $\sin \theta_0=0$,
		then we now claim that the above integral is finite if and only if $m<1/2$. Consider $\cos\theta_0=0$.
		Then \eqref{est:sin} holds for $0<\epsilon<\pi/4$. By Lemma~\ref{lem:g}, \eqref{est:g} holds with $\wh{g}(\xi):=\frac{1}{1+|\xi|^2}$.
		So, for $\cos\theta_0=0$, the above integral is finite if and only if
		$\int_1^\infty (1+r^2)^{m-1} dr<\infty$, which holds if and only if $m<1/2$. The case for $\sin \theta_0=0$ is similar.
		Hence,
		$\wf_m(f)=\emptyset$ for all $m<1/2$.
		For $1/2\le m<1$, $\wf_m(f)=S_1 \cup S_2$ with
		\begin{align*}
			&S_1:=\left\{(x_0, (0,r)) \setsp r\in \R\bs\{0\},
			x_0\in [0,1]\times \{0,1\}\right\},\\
			&S_2:=\left\{(x_0, (r,0)) \setsp r\in \R\bs\{0\},
			x_0\in \{0,1\}\times [0,1]\right\}.
		\end{align*}
		For $m \ge 1$, $\wf_m(f)=S_0\cup S_1 \cup S_2$, where
		{\small \[
			S_0:=\{(0,0), (1,0), (0,1),(1,1)\} \times \{ (r\cos \theta_0, r\sin \theta_0) \setsp
			r>0, \theta_0\in (-\pi,\pi]\bs \{0,\pm \tfrac{\pi}{2},\pi\}\}.
			\]}
		
		We now apply Theorem~\ref{thm:swf:integral} to estimate the wavefront set $\wf_m(f)$.
		By the definition of $\psi_{ast}$ in \eqref{psiast} and change of variables formula,
		we observe
		$$
		\mathcal{SH}_\psi f(a,s,t)
		=\la f, \psi_{ast}\ra
		=\la \wh{f}, \wh{\psi_{ast}}\ra
		 =a^{3/4}(\wh{f}\overline{\wh{\psi}(M_{as}^T\cdot)})^\vee(t).
		$$
		Therefore, we have
		\begin{small}
			$$
			I:=\int_{a=0}^{a=1}
			 \int_{s=s_0-r_0}^{s=s_0+r_0}\int_{ t\in B_{r_0}(x_0)}
			 \left|\mathcal{SH}_{\psi}f(a,s,t)\right|^2
			a^{-2m-3}dt ds da=\int_{a=0}^{a=1} a^{-2m-3} \eta(a) da,
			$$
		\end{small}
		where
		$$
		\eta(a):=
		 \int_{s=s_0-r_0}^{s=s_0+r_0}\int_{ t\in M_{as}^{-T}B_{r_0}(x_0)} |(\wh{f}(M_{as}^{-T}\cdot)\overline{\wh{\psi}})^\vee(t)|^2 dt ds.
		$$
		It is enough to show the convergence of the integral $\int_{a=0}^{a=r_0^2/4} a^{-2m-3} \eta(a) da$ since for all $m\in \R$ we have $\int_{a=r_0^2/4}^{a=1} a^{-2m-3} \eta(a) da<\infty$.
		%
		%
		Note that $\wh{f}(\xi_1,\xi_2)=\frac{\sin(\pi \xi_1)}{\pi \xi_1}
		\frac{\sin(\pi \xi_2)}{\pi \xi_2}$. Due to symmetry, we only consider points $x_0=(0,\tau)\in \ssupp(f)$ with $0\le \tau<1$. Let us first consider $s_0\ne 0$ and $s_0\ne \infty$.
		Using Plancherel's identity, we deduce that
		%
		\begin{small}
			\begin{align*}
				&\eta(a)\le \int_{s=s_0-r_0}^{s=s_0+r_0} \int_{t\in \R^2}
				 |(\wh{f}(M_{as}^{-T}\cdot)\overline{\wh{\psi}})^\vee(t)|^2 dt ds=\int_{s=s_0-r_0}^{s=s_0+r_0} \int_{\xi\in \R^2}
				 |\wh{f}(M_{as}^{-T}\xi)\overline{\wh{\psi}(\xi)}|^2 d\xi ds\\
				&\le \|\wh{\psi}\|_{L^\infty}^2 \int_{s=s_0-r_0}^{s=s_0+r_0}
				\int_{\supp(\wh{\psi})} |\wh{f}(M_{as}^{-T}\xi)|^2 d\xi ds=a^{3/2}\|\wh{\psi}\|_{L^\infty}^2\int_{s=s_0-r_0}^{s=s_0+r_0}
				\int_{\Xi(a,s)} |\wh{f}(\xi)|^2 d\xi ds.
			\end{align*}
		\end{small}
		Since $|s_0|>0$, for $0<r_0<|s_0|/2$, we have the following estimate
		\[
		\eta(a)
		\leq \|\wh{\psi}\|_{L^\infty}^2 \frac{4a^{7/2}}{\pi^2}\int_{s=s_0-r_0}^{s=s_0+r_0}\int_{\tfrac{1}{2a}<|\xi_1|<\tfrac{2}{a}}\int_{\left|\tfrac{\xi_2}{\xi_1}-s\right|<\sqrt{a}}\left|\frac{\sin \pi\xi_2}{\pi\xi_2}\right|^2d\xi_2d\xi_1ds.
		\]
		If $s_0>0$, then $s_0-r_0-\sqrt{a}>\tfrac{r_0}{2}$ and if $s_0<0$ then $s_0+r_0+\sqrt{a}<-\tfrac{r_0}{2}$ provided $\sqrt{a}<\tfrac{r_0}{2}$. So $\xi_2\neq 0$ throughout the domain.
		Note that the set $\{(\xi_1,\xi_2)\in \R^2 \setsp \frac{1}{2a}<|\xi_1|<\frac{2}{a}, |\frac{\xi_2}{\xi_1}-s|<\sqrt{a}\}$ consists of two trapezoids with height $\frac{3}{2}a^{-1}$ and two bases of lengths $a^{-1/2}$ and $4a^{-1/2}$. Hence, this trapezoid has area $\frac{15}{4} a^{-3/2}$.
		Therefore, for $0<r_0<|s_0|/2$,
		$$
		\eta(a)\leq \|\wh{\psi}\|_{L^\infty}^2\tfrac{16a^{11/2}}{\pi^4(s_0-r_0-\sqrt{a})^2}\int_{s=s_0-r_0}^{s=s_0+r_0}\int_{\tfrac{1}{2a}<|\xi_1|<\tfrac{2}{a}}\int_{\left|\tfrac{\xi_2}{\xi_1}-s\right|<\sqrt{a}}d\xi_2d\xi_1ds\leq Ca^4,
		$$
		where
$C:=2^6 \|\wh{\psi}\|_{L^\infty}^2
		\frac{15}{\pi^4 r_0}$
		by noting that the value of the above
		triple integral is $15 r_0 a^{-3/2}$.
		Hence, $I\leq C\int_{0}^{r_0^2/4}
		a^{-2m+1} da <\infty$ for all $m<1$. The case $|s_0|>1/2$ can be handled similarly. By Theorem~\ref{thm:swf:integral},
		we conclude that $\left((0,\tau),(r\cos \theta_0,r\sin \theta_0)\right)\notin \wf_m(f)$
		with  $r>0$ and $\theta_0 \in (-\pi,\pi]\bs\{0, \pm \frac{\pi}{2},\pi\}$ for all $0\le \tau<1$ and  $m<1$. Note that the angles $\theta_0\in \{0, \pm \frac{\pi}{2}, \pi\}$ correspond to $s_0=0$ or $s_0=\infty$.
		
		We now consider $m\ge 1$. Note that
		{\small
			\begin{align*}
				 &\mathcal{SH}_{\psi}f(a,s,t)=\la \wh{f},\wh{\psi_{ast}} \ra\\
				&=a^{\frac{3}{4}}\int_{ \Xi(a,s)}\frac{\textrm{sin}\ \pi \xi_1}{\pi\xi_1}\frac{\textrm{sin}\ \pi \xi_2}{\pi\xi_2}e^{-\pi i (\xi_1+\xi_2)}\wh{\psi_1}(a\xi_1)\wh{\psi_2}\left(a^{-\tfrac{1}{2}}\left(\tfrac{\xi_2}{\xi_1}-s\right)\right)e^{2\pi i (\xi_1,\xi_2)\cdot t}d\xi_1d\xi_2\\
				 &=a^{\frac{1}{4}}\int_{1/2\le |\xi_1|\le 2}\int_{|\xi_2/\xi_1|<1}\frac{\textrm{sin}\ \pi a^{-1}\xi_1}{\pi \xi_1}\frac{\textrm{sin}\ \pi \xi_1(a^{-\tfrac{1}{2}}\xi_2+a^{-1}s)}{\pi(a^{-\tfrac{1}{2}}\xi_2+a^{-1}s)} e^{-\pi i (a^{-1}\xi_1+\xi_1(a^{-\tfrac{1}{2}}\xi_2+a^{-1}s))}\times \\
				& \hspace{1.5 in}  \wh{\psi_1}(\xi_1)\wh{\psi_2}(\xi_2)e^{2\pi i (a^{-1}\xi_1,\xi_1(a^{-\tfrac{1}{2}}\xi_2+a^{-1}s))\cdot t}d\xi_1d\xi_2.
			\end{align*}
		}
		For $0<\tau<1$ and $s_0\ne 0$, we take $0<r_0<\min(\tau, |s_0|)/2$.
		For any $N\in \N$, using integration by parts, we deduce from the above identity that there exists a positive constant $C_N$ such that $|\mathcal{SH}_\psi f(a,s,t)|\le C_N a^{N}$ for all $t\in B_{r_0}(x_0)$ and $s\in (s_0-r_0,s_0+r_0)$. Hence,
		$\eta(a)\le 4\pi r_0^2 C_N a^{N}$ and consequently, $I=\int_{a=0}^{a=1} a^{-2m-3} \eta(a) da<\infty$ by taking $N\ge 2m+3$.
		Hence, $((0,\tau), r\cos\theta_0,r\sin\theta_0))\not \in \wf_m(f)$ for all $m\in\R$ with $r>0$ and $\theta_0\in (-\pi,\pi]\bs\{0, \pi\}$ for $0<\tau<1$.
		
		For $\tau=0, s_0\ne 0$ and $a>0$ sufficiently small,
		\begin{align*}
			\int_{t\in M_{as}^{-1} B_{r_0}(x_0)}
			 |(\wh{f}(M_{as}^{-T}\cdot)\overline{\wh{\psi}})^\vee(t)|^2 dt
			&\asymp
			\int_{t\in \R^2}
			 |(\wh{f}(M_{as}^{-T}\cdot)\overline{\wh{\psi}})^\vee(t)|^2 dt\\
			&=\int_{\xi\in \Xi(a,s)}
			 |\wh{f}(M_{as}^{-T}\xi)\overline{\wh{\psi}(\xi)}|^2 d\xi.
		\end{align*}
		For $c>0$, we define a set
		$E_c:=\{\xi\in \R^2 \setsp |\wh{\psi}(\xi)|\ge c\}$. Take $c>0$ small enough such that $\supp(\wh{\psi})\bs E_c$ is small.
		For $\tau=0$,
		there exists a positive constant $C_1$ such that $\lim_{a\to 0^+} a^{-4}\eta(a)$ $\ge C_1$, since by $|s|\ge |s_0|/2$, for $0<r_0<|s_0|/2$ and sufficiently small $a>0$,
		\begin{small}
			\begin{align*}
				\eta(a) &\geq   c^2 \frac{a^{7/2}}{4\pi^2}\int_{s=s_0-r_0}^{s=s_0+r_0}\int_{\tfrac{1}{2a}<|\xi_1|<\tfrac{2}{a}}\int_{\left|\tfrac{\xi_2}{\xi_1}-s\right|<\sqrt{a}}\left|\frac{\sin \pi\xi_2}{\pi\xi_2}\right|^2d\xi_2d\xi_1ds\\
				&\geq c^2 \tfrac{a^{11/2}}{16\pi^4(s_0+r_0+\sqrt{a})^2}\int_{s=s_0-r_0}^{s=s_0+r_0}\int_{\tfrac{1}{2a}<|\xi_1|<\tfrac{2}{a}}\int_{\left|\tfrac{\xi_2}{\xi_1}-s\right|<\sqrt{a}}d\xi_2d\xi_1ds\\
				&\ge C_1a^4,
			\end{align*}
		\end{small}
		where
$C_1:= c^2\frac{15 r_0}{49 \pi^4 s_0^2}$ by noting that the value of the above triple integral is $15 r_0 a^{-3/2}$.
		Hence the integral $I$ is divergent by $\int_0^1 a^{-2m-3} a^4 da=\infty$ for $m\ge 1$. The case $|s_0|> 1/2$ can be handled similarly. By Theorem~\ref{thm:swf:integral}, $\left((0,0),(r\cos \theta_0,r\sin \theta_0)\right)\in \wf_m(f)$ for all $m\ge 1$ with $r>0$ and $\theta_0\in (-\pi,\pi]\bs \{0,\pm \tfrac{\pi}{2},\pi\}$. Due to symmetry, we conclude that $S_0\subseteq \wf_m(f)$ for all $m\ge 1$.
		
		We now consider the second case: (ii) $s_0=0$.
		For $m<1/2$, we have $f\in H^m(\R^2)$. By Lemma~\ref{lem:sobolev:wfm} and Theorem~\ref{thm:swf:integral},
		we conclude that
		$((0,\tau), (r,0))\not \in \wf_m(f)$ with $r>0$ and $0\le \tau<1$ for all $m<1/2$.
		Using Theorem~\ref{thm:swf:integral} and the same argument for estimating in the case $s_0\neq 0$, for $s_0=0$, we have
		\begin{small}
			\begin{align*}
				\eta(a) &\geq c^2 \frac{a^{7/2}}{4\pi^2}\int_{s=-r_0}^{s=r_0}
				 \int_{\tfrac{1}{2a}<|\xi_1|<\tfrac{2}{a}}
				 \int_{|\frac{\xi_2}{\xi_1}-s|\le \sqrt{a}}\left|\frac{\sin \pi\xi_2}{\pi\xi_2}\right|^2d\xi_2d\xi_1ds\\
				&\geq c^2 \frac{a^{7/2}}{4\pi^2}\int_{s=-r_0}^{s=r_0}\int_{\tfrac{1}{2a}<|\xi_1|<\tfrac{2}{a}}\int_{\frac{s-\sqrt{a}}{2a}<\xi_2<\frac{s+\sqrt{a}}{2a}}\left|\frac{\sin \pi\xi_2}{\pi\xi_2}\right|^2d\xi_2d\xi_1ds\\
				&\geq c^2 \frac{a^{7/2}}{4\pi^2}\int_{s=-r_0/4}^{s=r_0/4}\int_{\tfrac{1}{2a}<|\xi_1|<\tfrac{2}{a}}\int_{-\frac{1}{4\sqrt{a}}<\xi_2<\frac{1}{4\sqrt{a}}}\left|\frac{\sin \pi\xi_2}{\pi\xi_2}\right|^2d\xi_2d\xi_1ds\\
				&\geq c^2 \frac{a^{7/2}}{4\pi^2}\int_{s=-r_0/4}^{s=r_0/4}\int_{\tfrac{1}{2a}<|\xi_1|<\tfrac{2}{a}}\int_{|\xi_2|<1/4}\left|\frac{\sin \pi\xi_2}{\pi\xi_2}\right|^2d\xi_2d\xi_1ds\\
				&\ge C_2a^{3},
			\end{align*}
		\end{small}
		where $C_2=\tfrac{3c^2}{8}\pi^{-2}\left(\tfrac{2\pi \operatorname{Si}(\pi/2)-4}{\pi^2}\right)\approx \tfrac{3c^2}{8}\pi^{-2}\times 0.46737$.
		So, $I\ge C_4 \int_0^1 a^{-2m-3} a^3 da=\infty$ for all $m\ge 1/2$. 
Hence, By Theorem~\ref{thm:swf:integral},
		$((0,\tau), (r,0))\in \wf_m(f)$ with $r>0$ and $0\le \tau<1$
		for all $m\ge 1/2$. Therefore, for points on the vertical edge $x_0\in \{0,1\}\times (0,1)$ and $s_0=0$,  we obtain $S_2 \subseteq \wf_m(f)$ for all $m\geq 1/2$ using similar estimates above.
		By symmetry for the horizontal edge, we have $S_1\subseteq \wf_m(f)$ for all $m\ge 1/2$.
		Hence, we obtain the same conclusion using both the definition of Sobolev wavefront sets in Definition~\ref{sobolevwf} and Theorem~\ref{thm:swf:integral}.
		By Theorem~\ref{thm:swf:decay}, we also have
		$\wf(f)=S_0\cup S_1\cup S_2$. See Figure~\ref{fig:singsupp} for the $C^\infty$ wavefront set of the function $f:=\chi_{[0,1]^2}$.
}\end{example}

\begin{figure}[htbp]
	\center{		 \includegraphics[width=0.4\linewidth]{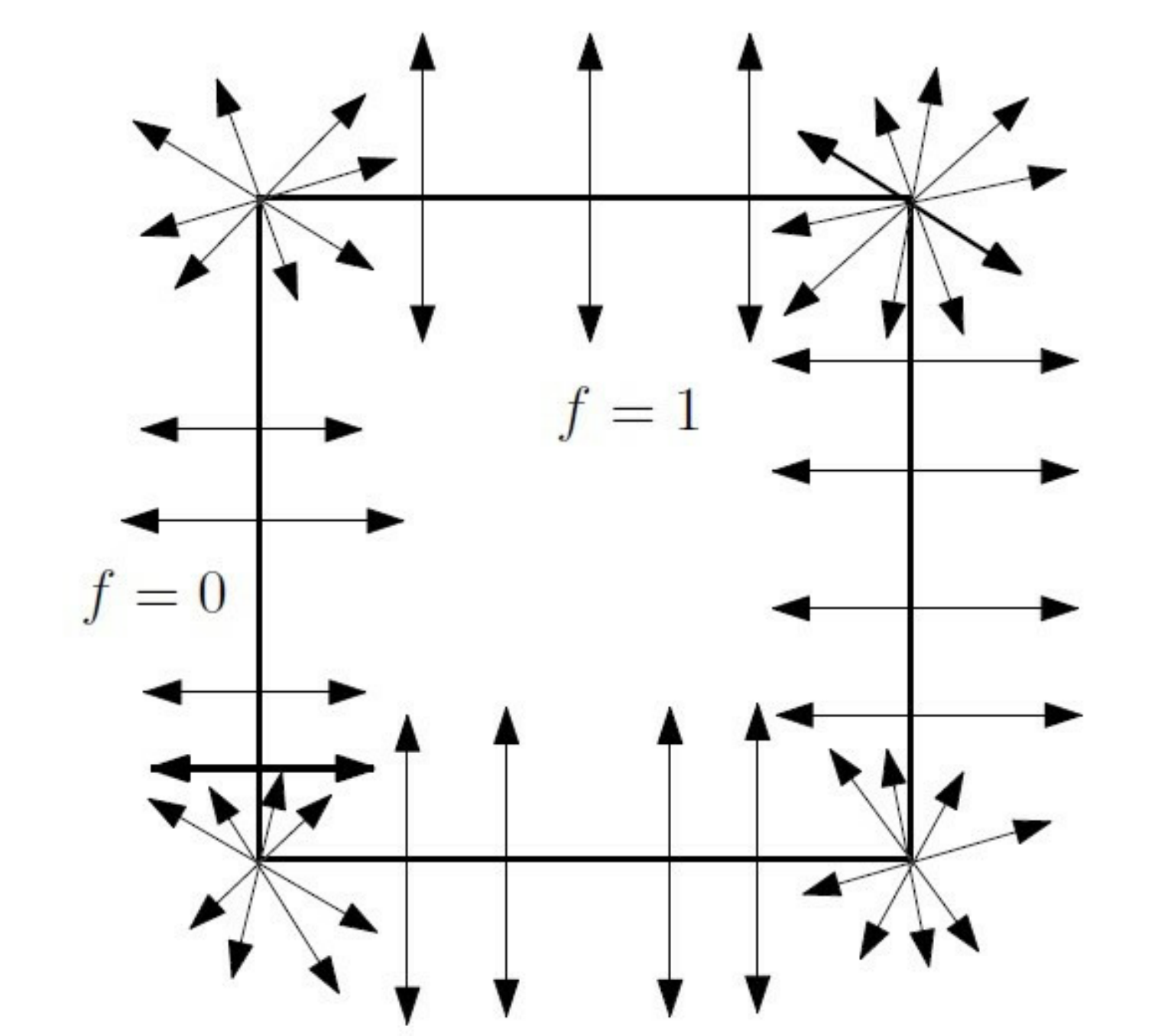}}\label{fig1}
	\caption{The function $f=\chi_{[0,1]^2}$ takes value $1$ on the unit square in $\R^2$ and $f(x)=0$ outside. The singular support, $\ssupp(f)$, is the boundary of the unit square. The directions of the ($C^\infty$) wavefront set of the function $f$ are shown with the arrows.}\label{fig:singsupp}
\end{figure}

Sobolev wavefront sets have been used for studying pseudo-differential operators and non-linear partial differential operators. In particular, this concept helps to prove H\"ormander's propagation theorem, which characterizes the $C^{\infty}$ wavefront set of the distributional solution of a PDE using the principal symbol of the corresponding differential operator.
Some basic properties of $C^\infty$ or Sobolev wavefront sets are given below (see \cite{hormander1997lectures,petersen1983introduction}) and are useful later in this paper.

\begin{prop}\label{sobolevineq}
	For a distribution $f\in \DCOpr$,
	\begin{equation*}\label{wfm:nested}
		\wf_m(f) \subseteq \wf_n(f) \subseteq \wf(f),\qquad \forall\; m\le n
	\end{equation*}
	and $\wf(f)$ is the complement of the interior of $\cap_{m\in \R} (\wf_m(f)^c)$, more precisely,
	\begin{equation*}\label{wf:wfm}
		\wf(f)^c=\left(\bigcap_{m\in \R} \left(\wf_m(f)^c\right)\right)^{o},
	\end{equation*}
	where the superscript $o$ stands for the interior of a set.
\end{prop}

\section{Continuous shearlet transform in Sobolev spaces} \label{CSTHm}

To prove Theorems~\ref{thm:swf:integral} and~\ref{thm:swf:decay} for characterizing Sobolev wavefront sets through the continuous shearlet transform, we need to study the continuous shearlet transform $\mathcal{SH}_{\psi}f(a,s,t)$ in the Sobolev space $H^m(\R^2)$ with $m\in \R$.  We now extend the continuous shearlet transform from $L^2(\R^2)$ to the Sobolev space $H^m(\R^2)$ for tempered distributions.

Let $\R_+:=(0,\infty)$.
The shearlet group $\left(\mathbb{R}_+\times \R\right)\times \mathbb{R}^2$
is equipped with a group multiplication: $(a,s,t)\cdot (a',s',t')=(aa',s+s'\sqrt{a},t+M_{as} t)$.
It has been studied by many researchers (e.g., \cite{kutyniok2009resolution}) that for every $f\in L^2(\R^2)$,
\begin{equation}\label{reprodformula}
	f(\cdot)=\int_{\R_+} \int_{\R}\int_{\R^2}\la f, \psi_{ast}\ra \psi_{ast}(\cdot) a^{-3} dt ds da
\end{equation}
holds in the weak sense under the admissibility condition of the function $\psi$:
\begin{equation}\label{admissible}
	 \int_{\R_+}\int_{\R}\left|\wh{\psi}\left(M_{as}^T \xi\right)\right|^2 
	a^{-3/2}
	ds da=1,
\end{equation}
where $M_{as}$ and $\psi_{ast}$  are defined in \eqref{mas} and \eqref{psiast}.
Let $\psi$ be defined in \eqref{psidef}, i.e., $\wh{\psi}(\xi_1,\xi_2):=\wh{\psi_1}(\xi_1)\wh{\psi_2}(\xi_2/\xi_1)$ with $\psi_1, \psi_2$ satisfying the conditions in (C1) and (C2) in Subsection~\ref{shearlet}.
Then it is well known that \eqref{admissible} holds.
Indeed,
by $\det(M_{as})=a^{3/2}$ and
$\wh{\psi}(M_{as}^T\xi)=\wh{\psi_1}(a\xi_1)\wh{\psi_2}(\frac{1}{\sqrt{a}}(\frac{\xi_2}{\xi_1}-s))$,
we have
%
\[
\int_{\R_+}\int_{\R}|\wh{\psi}(M_{as}^T \xi)|^2 a^{-3/2} ds da
=\int_{\R_+} \int_{\R} |\wh{\psi_1}(a\xi_1)|^2 |\wh{\psi_2}(\tfrac{1}{\sqrt{a}}(\tfrac{\xi_2}{\xi_1}-s))|^2
a^{-3/2} ds da
=1,
\]
%
where we used the assumptions $\|\wh{\psi_2}\|_{L^2}=1$ and \eqref{psi1:admissible} in (C1) and (C2).

Since $\psi$ is a Schwartz function,
$\mathcal{SH}_\psi f(a,s,t):=\la f, \psi_{ast}\ra$ is well defined for any
tempered distribution $f$.
By \eqref{psiast}, we have $\psi_{ast}=a^{-3/4} \psi(M_{as}^{-1}(\cdot-t))$ and hence
\begin{equation}\label{hatpsiast}
	\wh{\psi_{ast}}(\xi)=a^{3/4}
	e^{-i2\pi t\cdot \xi} \wh{\psi}(M_{as}^T\xi).
\end{equation}
Therefore,
by $\mathcal{SH}_\psi f(a,s,\cdot)=\la f, \psi_{as\cdot}\ra=
a^{-3/4} f*\overline{\psi(-M_{as}^{-1}\cdot)}$, we obtain
\begin{equation}\label{hatSH}
	 \left(\mathcal{SH}_{\psi}f(a,s,\cdot)\right)^{\wedge}(\xi)
	=a^{{3}/{4}}\wh{f}(\xi)
	\overline{\wh{\psi}\left(M_{as}^T \xi\right)}.
\end{equation}
Since
$\supp(\wh{\psi})\subseteq \{(\xi_1,\xi_2) \setsp 1/2\le |\xi_1|\le 2, |\xi_2/\xi_1|\le 1\}$, using \eqref{hatpsiast}, we obtain
\begin{equation}\label{supp:hatpsiast}
	\supp(\widehat{\psi_{ast}})
	=\supp(\wh{\psi}(M_{as}^T\cdot))=
	M_{as}^{-T} \supp(\wh{\psi})\subseteq
	\Xi(a, s),
\end{equation}
where $\Xi(a,s)$ is defined in \eqref{xias}.
For $(a,s) \not \in (0,1]\times [-2,2]$, we have either $a>1$ or $|s|>2$.
Let $(\xi_1,\xi_2)\in \Xi(a,s)$.
If $a>1$, then $|\xi_1|<2$ and therefore,
$(\xi_1,\xi_2)\not \in \cC$, where the cone $\cC$ is defined in \eqref{hcone}.
If $|s|>2$ and $0<a\le 1$, then $|\xi_2/\xi_1| \ge |s|- \sqrt{a}>1$ and hence, $(\xi_1,\xi_2)\not \in \cC$.
This proves
\begin{equation}\label{hatSH=0}
	 \wh{\psi_{ast}}(\xi)=0=\wh{\psi}\left(M_{as}^T\xi\right), \qquad \forall\; \xi\in \cC \quad \mbox{and}\quad
	(a,s)\not \in (0,1]\times [-2,2].
\end{equation}

Our main goal in this section is to develop some theory and estimate for the continuous shearlet transform in the Sobolev space $H^m(\R^2)$ with any real number $m\in\R$.
Firstly, since $\wh{\psi}$ in Subsection~\ref{shearlet} belongs to $\mathcal{D}(\R^2)$, the shearlet $\psi_{ast}$ defined in \eqref{psiast} and all its derivatives are Schwartz functions with rapid decay at infinity.
Hence, $\la f,\psi_{ast} \ra$ is well defined for every tempered distribution $f$ on $\R^2$.
Now, we  define  a  \textit{continuous shearlet transform}  (CST)  associated to $\psi$  (defined in Subsection~\ref{shearlet}) as
\begin{align*}
	\mathcal{SH}_{\psi}: H^m(\mathbb{R}^2)\rightarrow L^2\left(\left(\R_+\times \R,a^{-3} dsda\right),H^{m}(\mathbb{R}^2)\right),
\end{align*}
by
\begin{align*}
	\mathcal{SH}_{\psi}f(a, s, t):=\left\langle f,\psi_{ast} \right\rangle, \  \mbox{for} \  f \in H^m(\mathbb{R}^2) \  \mbox{and} \  (a, s, t)\in \left(\mathbb{R}_+\times \mathbb{R}\right)\times \mathbb{R}^2.
\end{align*}
Here, the space  $L^2((\R_+\times \R,a^{-3}dsda),H^{m}(\mathbb{R}^2))$ is a fiber space, which is abbreviated by $\mathcal{F}^m$. Note that $\mathcal{F}^m$ is isomorphic to the tensor product $L^2(\R_+\times \R,a^{-3}dsda)\wh{\otimes}$ $H^{m}(\mathbb{R}^2)$. The associated norm of $\mathcal{SH}_{\psi} f$ on  the fiber space $\mathcal{F}^m$ is then defined by
\begin{equation*}\label{fibernorm}
	 \left\|\mathcal{SH}_{\psi}f\right\|_{\mathcal{F}^m}:=
	 \left(\int_{a\in\R_+}\int_{s\in\mathbb{R}}
	 \left\|\mathcal{SH}_{\psi}f(a,s,\cdot)\right\|_{H^m}^2 a^{-3} dsda\right)^{{1}/{2}}.
\end{equation*}

The extended CST from $L^2(\mathbb R^2)$ to   $H^m(\mathbb{R}^2)$ with $m \in \mathbb R$ leads to a unitary operator between the Sobolev space $H^{m}(\mathbb R^2)$ and the fiber space $\mathcal{F}^m$.
Some properties of the extended CST in the Sobolev space $H^m(\R^2)$ are as follows:

\begin{prop}\label{prop:bandcomnorm}
	For  $f \in H^m(\mathbb R^2)$ with $m\in \R$, the following statements hold:	 
	\begin{enumerate}
		\item [(i)]	For every fixed $a\in \R_+$ and $s\in \R$, the mapping $\mathcal{SH}_{\psi}f(a,s,\cdot)$ can be extended uniquely into a continuous mapping from $H^m(\mathbb{R}^2)$ to itself.
		\item[(ii)] $\mathcal{SH}_{\psi}$  is an isometry from $H^m(\mathbb{R}^2)$ to $\mathcal{F}^m$, i.e., $\|\mathcal{SH}_{\psi}f\|_{\mathcal{F}^m}
		=\|f\|_{H^m}$.
	\end{enumerate}
\end{prop}

\begin{proof} (i)
	By \eqref{hatSH} and the definition of $H^m(\R^2)$, we deduce that
	\begin{equation*}\label{SHfHm}
		 \left\|\mathcal{SH}_{\psi}f(a,s,\cdot)\right\|_{H^m}^2
		 =a^{{3}/{2}}\int_{\R^2}(1+|\xi|^2)^m |\wh{f}(\xi)|^2 \left|\wh{\psi}\left(M_{as}^T \xi\right)\right|^2 d\xi,
	\end{equation*}
	from which we straightforwardly have
	 $\left\|\mathcal{SH}_{\psi}f(a,s,\cdot)\right\|_{H^m}^2
	\le a^{{3}/{2}} \|\widehat{\psi}\|_{L^\infty}^2
	\|f\|_{H^m}^2$.
	
	\noindent(ii) Using \eqref{hatSH} and Fubini's theorem, by the admissibility condition in \eqref{admissible}, we have
	\begin{align*} \left\|\mathcal{SH}_{\psi}f\right\|_{\mathcal{F}^m}^2
		&=\int_{\R_+}\int_{\R}    \left\|\mathcal{SH}_{\psi}f(a,s,\cdot)\right\|_{H^m}^2a^{-3} ds da\\ &=\int_{\mathbb{R}^2}\left(\int_{\R_+}\int_{\R}a^{{3}/{2}}\left|\wh{\psi}\left(M_{as}^T\ \xi\right)\right|^2 a^{-3} dsda\right)(1+|\xi|^2)^m |\wh{f}(\xi)|^2  d\xi\\
		&=\int_{\mathbb{R}^2}
		(1+|\xi|^2)^m |\wh{f}(\xi)|^2  d\xi=\|f\|_{H^m}^2.
	\end{align*}
This completes the proof.
\end{proof}

The above result is established where the dilation and direction parameters $a$ and $s$ vary on the non-compact sets $\R_+$ and $\R$, respectively. Now, our aim is to  obtain the  above result for the Sobolev space $H^m(\mathbb{R}^2)$ when both the parameters $a$ and $s$ vary on compact sets.
Let $\mathcal{C}$ be the horizontal cone defined in \eqref{hcone}. By switching $\xi_1$ with $\xi_2$,
we define the vertical cone $\mathcal{C}^v:=\{ (\xi_2,\xi_1) \setsp (\xi_1,\xi_2)\in \mathcal{C}\}$.
Consider the subspace $H^m(\mathcal{C})^{\vee}$ of $H^m(\R^2)$ over the horizontal cone $\mathcal{C}$ as follows:
\begin{equation*}\label{Hm:hcone}
	H^m(\mathcal{C})^{\vee}:=\{ f\in H^m(\mathbb{R}^2) \setsp \supp(\wh{f})\subseteq \mathcal{C}\}.
\end{equation*}
Similarly, we define $H^m(\mathcal{C}^v)^{\vee}:=\{ f\in H^m(\mathbb{R}^2) \setsp \supp(\wh{f})\subseteq \mathcal{C}^v\}$ over $\mathcal{C}^v$.
Let $R_0:=(0,1]\times [-2,2]$.
We define the fiber spaces
$$
\mathcal{F}^m_{\mathcal C}:=L^2((R_0,a^{-3} dsda),H^m(\mathcal{C})^{\vee})\quad  \mbox{and} \quad  \mathcal{F}^m_{\mathcal C^v}:=L^2((R_0,a^{-3} dsda),H^m(\mathcal{C}^v)^{\vee}).
$$

Now we establish the following result on restricting the  continuous shearlet transform $\mathcal{SH}_{\psi}$ on the Sobolev subspaces $H^m(\mathcal{C})^{\vee}$ and $H^m(\mathcal{C}^v)^{\vee}$.

\begin{prop}\label{prop:shfcone}
	The following identities are satisfied
	\begin{equation}\label{SHfHmcone}
		\|\mathcal{SH}_{\psi} f\|_{\cF^m_{\cC}}^2=\|f\|_{H^m}^2=
		 \|\mathcal{SH}_{\psi}f\|_{\cF^m}^2, \qquad \forall\, f\in H^m(\cC)^\vee
	\end{equation}
	and
	\[
	\|\mathcal{SH}_{\psi}^{(v)} f\|_{\cF^m_{\cC^v}}^2=\|f\|_{H^m}^2=
	\|\mathcal{SH}_{\psi}^{(v)} f\|_{\cF^m}^2, \qquad \forall\, f\in H^m(\cC^v)^\vee.
	\]
\end{prop}

\begin{proof} For $\xi \not \in \cC$, by $\wh{f}(\xi)=0$ and \eqref{hatSH}, we have $(\mathcal{SH}_\psi f(a,s,\cdot))^{\wedge}(\xi)=0$.
	For $\xi\in \cC$ but $(a,s)\not\in I_0:= (0,1]\times [-2,2]$, it follows directly from \eqref{hatSH=0} that $(\mathcal{SH}_\psi f(a,s,\cdot))^{\wedge}(\xi)=0$.
	Thus, it is now trivial to conclude that
	 $\|\mathcal{SH}_{\psi}f\|_{\cF^m_{\cC}}^2=
	\|\mathcal{SH}_{\psi}f\|_{\cF^m}^2$ for all $f\in H^m(\cC)^\vee$.
	Since we already proved $\|\mathcal{SH}_{\psi} f\|_{\cF^m}^2=\|f\|_{H^m}^2$ in item (ii) of Proposition~\ref{prop:bandcomnorm},
	this proves \eqref{SHfHmcone}.
\end{proof}

The above  result  shows that a function $f$ in $H^m(\mathcal{C})^{\vee}$(or $ H^m(\mathcal{C}^v)^{\vee}$) can be continuously reproduced by using the continuous horizontal (or vertical) shearlet transform, but neither of them is sufficient to produce all $f \in H^m(\R^2)$. To overcome this problem, noting that $\mathbb R^2=\lfD \cup \mathcal{C} \cup  \mathcal{C}^v$ with $\lfD:=[-2,2]^2$,
we decompose $f$ as follows:
\begin{equation}\label{decomp}
	 f=P_{\lfD}f+P_{\mathcal{C}}f+P_{\mathcal{C}^{v}}f
	\quad \mbox{with}\quad
	 \wh{P_{\lfD}f}=\wh{f}\chi_{\lfD},\quad
	 \wh{P_{\mathcal{C}}f}=\wh{f}\chi_{\mathcal{C}},\quad
	 \wh{P_{\mathcal{C}^v}f}=\wh{f}\chi_{\mathcal{C}^v}.
\end{equation}
As a direct consequence of Proposition~\ref{prop:shfcone}, we have

\begin{theorem}\label{thm:sobolevis}
	For every $f\in H^m(\R^2)$ with $m\in \R$,
	the following identity holds:
	\[
	\|f\|_{H^m}^2=
	\|P_{\lfD} f\|_{H^m}^2+
	 \left\|\mathcal{SH}_{\psi}(P_{\mathcal{C}}f)
	 \right\|_{\mathcal{F}^m}^2+\left\|\mathcal{SH}_{\psi}^v(P_{\mathcal{C}^{v}}f)
	\right\|_{\mathcal{F}^m}^2.
	\]
\end{theorem}

\begin{proof} Since $\R^2$ is an almost disjoint union of
	$\lfD$, $\mathcal{C}$ and  $\mathcal{C}^v$, by definition, we have $\|f\|_{H^m}^2=\|P_{\lfD} f\|_{H^m}^2+\| P_{\mathcal{C}}f\|_{H^m}^2+\|P_{\mathcal{C}^v}f\|_{H^m}^2$. Now the claim follows directly from Proposition~\ref{prop:shfcone}.
\end{proof}

Note that the reproducing formula \cite[(3.7)]{kutyniok2009resolution} holds as a corollary of Theorem~\ref{thm:sobolevis} for $f\in H^m(\mathbb{R}^2)$ with $m=0$.

\section{Proofs of Theorems~\ref{thm:swf:integral} and~\ref{thm:swf:decay}}
\label{sec:swf:proof}

In this section we shall prove Theorems~\ref{thm:swf:integral} and~\ref{thm:swf:decay}, which
address the relationship between the continuous shearlet transform and microlocal properties of a distribution. So we will perform microlocal analysis using the continuous shearlet transform.

According to Lemma~\ref{lem:swf:def}, to prove Theorem~\ref{thm:swf:integral}, we have to find suitable conditions on which the microlocal Sobolev integral $\int_{U_{\xi_0,\epsilon}}|\widehat{\varphi f}(\xi)|^2|\xi|^{2m}d\xi$ is finite.
So, we now decompose the localized function $\varphi f$.
Using the reproducing formula in \eqref{reprodformula} and noting $\mathcal{SH_\psi} (\varphi f)(a,s,t)=\la \varphi f, \psi_{ast}\ra$, we have
\[
\wh{\varphi f}(\xi)=
\int_{\R_+}\int_\R \int_{\R^2} \wh{\psi_{ast}}(\xi) \mathcal{SH}_\psi (\varphi f)(a,s,t) a^{-3} dt ds da.
\]
Now by \eqref{hatSH=0}, if $\xi\in \cC$, then $\wh{\psi_{ast}}(\xi)=0$ for all $(a,s)\not \in (0,1]\times [-2,2]$. Therefore,
\[
\wh{\varphi f}(\xi)\chi_{\cC}(\xi)=
\chi_{\cC}(\xi) \int_{a=0}^{a=1}\int_{s=-2}^{s=2} \int_{t\in \R^2} \wh{\psi_{ast}}(\xi) \mathcal{SH}_\psi (\varphi f)(a,s,t) a^{-3} dt ds da.
\]
Consequently, we have the following decomposition:
\begin{align}\label{decomposewholedomain}
	(\widehat{\varphi f})(\xi)=
	&(\widehat{P_{\lfD}(\varphi f)})(\xi)+\chi_{\cC}(\xi)
	 \int_{a=0}^{a=1}\int_{s=-2}^{s=2}\int_{t\in \R^2}\widehat{\psi_{ast}}(\xi)\mathcal{SH}_{\psi}(\varphi f)(a,s,t)a^{-3} dt ds da\nonumber\\
	 &+\chi_{\cC^v}(\xi)\int_{a=0}^{a=1}\int_{s=-2}^{s=2}
	\int_{t\in \R^2}\widehat{\psi^{(v)}_{ast}}(\xi)\mathcal{SH}_{\psi}^{(v)}(\varphi f)(a,s,t)a^{-3} dt ds da.
\end{align}
Let $x_0\in \R^2$ and $\xi_0\in \R^2\bs\{0\}$. We define $s_0:=\tan(\mbox{Arg}(\xi_0))$, which is just the slope of the direction $\xi_0$. For small $\epsilon>0$, we observe that
\begin{equation}\label{UW}
	U_{\pm \xi_0,\epsilon} \subseteq \left\{ \tfrac{1}{a}(\cos(\arctan s),
	\sin(\arctan s))
	\setsp |a|\in (0,1), |s-s_0|<r_0\right\}=:W_{s_0,r_0},
\end{equation}
where
\begin{equation}\label{r0}
	 r_0:=\max(\tan(\mbox{Arg}(\xi_0)+\epsilon)-s_0, s_0-\tan(\mbox{Arg}(\xi_0)-\epsilon))>0.
\end{equation}
Note that $r_0\to 0^+$ as $\epsilon\to 0^+$.
Define $N_{r_0}(s_0,x_0):=I_{s_0}\times B_{r_0}(x_0)$, where
\begin{equation*}\label{Is0}
	I_{s_0}:=[-2,2]\cap [s_0-2r_0,s_0+2r_0].
\end{equation*}
Then the decomposition in \eqref{decomposewholedomain} can be rewritten as
\begin{equation}\label{decom}
	(\widehat{\varphi f})(\xi)=
	(\widehat{P_{\lfD}(\varphi f)})(\xi)
	+\wh{g_1}(\xi)+\wh{g_2}(\xi)+
	\wh{g_3}(\xi)+\wh{g_4}(\xi),
\end{equation}
where
\begin{equation}\label{g1}
	\wh{g_1}(\xi):=
	\chi_{\cC}(\xi)
	\int_{a=0}^{a=1}\int_{(s,t)\in I_{s_0}\times B_{r_0}(x_0)}
	 \widehat{\psi_{ast}}(\xi)\mathcal{SH}_{\psi}(\varphi f)(a,s,t)a^{-3} ds dt da
\end{equation}
and
\[
\wh{g_2}(\xi):=
\chi_{\cC}(\xi)
\int_{a=0}^{a=1}\int_{
	(s,t)\in ([-2,2]\times\R^2)\bs N_{r_0}(s_0,x_0)}
\widehat{\psi_{ast}}(\xi)\mathcal{SH}_{\psi}(\varphi f)(a,s,t)a^{-3} dsdtda,
\]
and $g_3,g_4$ are similarly defined as $g_1, g_2$ by exchanging $\mathcal{SH}_\psi$ with $\mathcal{SH}_\psi^{(v)}$, respectively.

Since $\wh{\varphi f}$ is continuous and $\wh{P_{\lfD} (\varphi f)}(\xi)=\wh{(\varphi f)}(\xi) \chi_{\lfD}(\xi)$, by the definition of $U_{\xi_0,\epsilon}$ in \eqref{Wxie},
it is trivial to conclude that
\[
\int_{U_{\xi_0,\epsilon}} |\wh{P_{\lfD} (\varphi f)}(\xi)|^2 |\xi|^{2m}d\xi
\le \left(\sup_{\xi\in \lfD} |\wh{\varphi f}(\xi)|^2\right) \int_{U_{\xi_0,\epsilon}\cap \lfD} |\xi|^{2m}d\xi<\infty,
\]
for all $\xi_0\in \R^2 \bs\{0\}$, $\epsilon>0$ and $m\in \R$.

We have the following result estimating $g_2$.

\begin{prop}\label{convg2}
	For all $\xi_0\in \R^2\backslash\{0\}$, $0<\epsilon<\pi/2$ and $m\in \R$,
	\begin{equation*}\label{g2}
		\int_{U_{\pm \xi_0,\epsilon}} |\wh{g_2}(\xi)|^2 |\xi|^{2m} d\xi<\infty.
	\end{equation*}
\end{prop}

\begin{proof}
	We decompose $([-2,2]\times \R^2)\bs N_{r_0}(s_0,x_0)=E_1\cup E_2$ with
	$E_1:=I_{s_0}^c \times \R^2$ and
	$E_2:=I_{s_0} \times B_{r_0}(x_0)^c$,
	where $I_{s_0}^c:=[-2,2]\bs I_{s_0}$ and $B_{r_0}(x_0)^c:=\R^2\bs B_{r_0}(x_0)$.
	We define $\wh{h_j}:=\wh{g_2}\chi_{E_j}$ for $j=1,2$.
	We first deal with $\int_{U_{\pm \xi_0,\epsilon}} |\wh{h_1}(\xi)|^2 |\xi|^{2m} d\xi$. Note that
	\[
	\widehat{h_1}(\xi)=
	\chi_{\cC}(\xi)\int_{a=0}^{a=1} \int_{s\in I_{s_0}^c}\int_{t\in \R^2} \widehat{\psi_{ast}}(\xi)\mathcal{SH}_{\psi}(\varphi f)(a,s,t)a^{-3} dtdsda.
	\]
	By \eqref{hatpsiast}, we have
	\[
	\int_{\R^2} \widehat{\psi_{ast}}(\xi)\mathcal{SH}_{\psi}(\varphi f)(a,s,t) dt
	=a^{3/4} \wh{\psi}(M_{as}^T\xi)
	(\mathcal{SH}_\psi(\varphi f)(a,s,\cdot))^{\wedge}(\xi).
	\]
	Hence, by \eqref{hatSH}, we conclude that
	\[
	|\wh{h_1}(\xi)|\le \chi_{\cC}(\xi)
	\int_{a=0}^{a=1} \int_{s\in I_{s_0}^c}
	|\wh{\psi}(M_{as}^T\xi)|^2 |\wh{\varphi f}(\xi)| a^{-3/2} ds da.
	\]
	By \eqref{UW}, we have $U_{\pm \xi_0,\epsilon}\subseteq W_{s_0,r_0}$.
	For $s\in I_{s_0}^c$ (which implies $|s-s_0|> 2r_0$), then $\Xi(a,s)\cap W_{s_0,r_0}=\emptyset$ for $0<\sqrt{a}<|s-s_0|-r_0$. For $s\in I_{s_0}^c$, by definition we have $|s-s_0| > 2r_0$ which implies $|s-s_0|-r_0> r_0$.
	Then for $\xi\in U_{\pm \xi_0,\epsilon}$, we deduce from \eqref{supp:hatpsiast} that
	\[
	|\wh{h_1}(\xi)|\le \int_{a=r_0^2}^{a=1} \int_{s\in I_{s_0}^c}
	|\wh{\psi}(M_{as}^T\xi)|^2 |\wh{\varphi f}(\xi)| a^{-3/2} ds da.
	\]
	Since $0<r_0<1$, we see that
	\begin{equation}\label{Kr0}
		K_{r_0}:=\cup_{a\in [r_0^2,1], s\in [-2,2]} \Xi(a,s)
		\subseteq \{(\xi_1,\xi_2) \setsp
		\tfrac{1}{2}\le |\xi_1|\le \tfrac{2}{r_0^2}, |\tfrac{\xi_2}{\xi_1}|\le 3\}
	\end{equation}
	is a bounded set. Because $\wh{\varphi f}$ is continuous,
	we conclude that $C:=\sup_{\xi\in K_{r_0}}
	|\wh{\varphi f}(\xi)|<\infty$.
	So, for $\xi\in U_{\pm \xi_0,\epsilon}$, by \eqref{supp:hatpsiast}, we have
	\[
	|\wh{h_1}(\xi)|
	\le C\| \wh{\psi}\|^2_{L^\infty} \int_{a=r_0^2}^{a=1} \int_{s\in I_{s_0}^c}
	\chi_{\Xi(a,s)}(\xi) a^{-3/2} ds da.
	\]
	Note that $\Xi(a,s)\subseteq K_{r_0}$ for all $a\in [r_0^2,1]$ and $s\in I_{s_0}^c$. Therefore, for $\xi\in U_{\pm \xi_0, \epsilon}$, by \eqref{Kr0}, we have
	\[
	|\wh{h_1}(\xi)|\le C\| \wh{\psi}\|^2_{L^\infty}
	\int_{a=r_0^2}^{a=1} \int_{s\in I_{s_0}^c}
	\chi_{K_{r_0}}(\xi) a^{-3/2} dsda
	\le C_1 \chi_{K_{r_0}}(\xi),
	\]
	where $C_1=8C\|\wh{\psi}\|^2_{L^\infty} (r_0^{-1}-1)$.
	Consequently, by $U_{\pm \xi_0,\epsilon}\subseteq W_{s_0,r_0}$ and \eqref{Kr0},
	\begin{align*}
		\int_{U_{\pm \xi_0,\epsilon}} |\wh{h_1}(\xi)|^2|\xi|^{2m} d\xi
		&\le C_1 \int_{U_{\pm \xi_0,\epsilon}\cap K_{r_0}} |\xi|^{2m} d\xi
		\le C_1 \int_{K_{r_0}}  |\xi|^{2m} d\xi\\
		&\le C_1 \int_{\theta=-\arctan 3}^{\theta=\arctan 3}
		 \int_{r=1/2}^{r=\frac{2\sqrt{1+3^2}}{r_0^2}} r^{2m+1} dr d\theta<\infty.
	\end{align*}
	Finally, we show  $\int_{U_{\pm\xi_0,\epsilon}}|\wh{h_2}(\xi)|^2|\xi|^{2m} d\xi<\infty$. First note that
	$$
	\widehat{h_2}(\xi)=
	 \chi_{\mathcal{C}}(\xi)\int_{a=0}^{a=1}\int_{s\in I_{s_0}}\int_{t\in B_{r_0}(x_0)^{c}}\widehat{\psi_{ast}}(\xi)\mathcal{SH}_{\psi}(\varphi f)(a,s,t)a^{-3}dt ds da.
	$$
	By \cite[Lemma~5.2]{kutyniok2009resolution}  for each $k>0$,
	there exists a positive constant $C_2$ such that
	\begin{equation} \label{lemma5.2}
		|\mathcal{SH}_{\psi}(\varphi f)(a,s,t)|=|\left\langle \varphi f , \psi_{ast}\right\rangle| \leq C_2 a^{\frac{1}{4}}\left(1+\frac{d(t,B_{r_0'}(x_0))^2}{a}\right)^{-k},
	\end{equation}
	for all $a\in (0,1]$, $s\in [-2,2]$ and $t\in \R^2$,
	where $\supp(\varphi)\subseteq B_{r_0'}(x_0)\subseteq B_{r_0}(x_0)\subseteq \Omega$ with $0<r_0'<r_0$. Using  \eqref{hatpsiast}, \eqref{supp:hatpsiast} and
	\eqref{lemma5.2}, we get
	\begin{align*}
		|&\widehat{h_2}(\xi)|\leq \chi_{\mathcal{C}}(\xi)\int_{a=0}^{a=1}\int_{s\in I_{s_0}}\int_{t\in B_{r_0}(x_0)^{c}}|\widehat{\psi_{ast}}(\xi)||\mathcal{SH}_{\psi}(\varphi f)(a,s,t)|a^{-3}dtdsda\\
		&\leq C_2 \|\wh{\psi}\|_{L^\infty}
		\int_{0}^{1}\int_{ I_{s_0}}\int_{ B_{r_0}(x_0)^{c}}
		 a^{\tfrac{3}{4}}\chi_{\Xi(a,s)}(\xi) a^{\tfrac{1}{4}}\left(1+\frac{d(t,B_{r_0'}(x_0))^2}{a}\right)^{-k}a^{-3}dtdsda\\
		&\leq C_2 \|\wh{\psi}\|_{L^\infty}
		\int_{0}^{1}\left(\int_{ I_{s_0}}\chi_{\Xi(a,s)}(\xi)ds\right)
		\left(\int_{B_{r_0}(x_0)^{c}}
		 \left(1+\frac{d(t,B_{r_0'}(x_0))^2}{a}\right)^{-k} dt\right) a^{-2} da.
	\end{align*}
	For $\xi\in \Xi(a,s), s\in [-2,2]$ and $a \in (0, 1)$, we observe
	\begin{equation} \label{est:Xi}
		\tfrac{1}{\sqrt{2}}\leq \tfrac{\sqrt{1+(|s|-\sqrt{a})^2}}{2}<a|\xi|<
		2\sqrt{1+(|s|+\sqrt{a})^2}\leq 2\sqrt{10}.
	\end{equation}
	For $a\in (0,1)$ such that $a\not \in [\frac{1}{\sqrt{2}|\xi|},\frac{2\sqrt{10}}{|\xi|}]$, then the above inequality implies
	$\xi\not \in \Xi(a,s)$ for all $s\in [-2,2]$, which leads to $\int_{ I_{s_0}}\chi_{\Xi(a,s)}(\xi)ds=0$.
	Note that $\xi\in \Xi(a,s)$ implies $\frac{\xi_2}{\xi_1}-\sqrt{a}\le s\le \frac{\xi_2}{\xi_1}+\sqrt{a}$.
	Consequently, we conclude that $\int_{ I_{s_0}}\chi_{\Xi(a,s)}(\xi)ds
	\le 2\sqrt{a}\chi_{[\frac{1}{\sqrt{2}|\xi|},\frac{2\sqrt{10}}{|\xi|}]}(a)$. Therefore, for $k>1$,
	\begin{align*}
		|\widehat{h_2}(\xi)|&\leq
		2C_2\|\wh{\psi}\|_{L^\infty}
		 \int_{a=\tfrac{1}{\sqrt{2}|\xi|}}
		 ^{a=\tfrac{2\sqrt{10}}{|\xi|}}\int_{t\in B_{r_0}(x_0)^{c}}\left(1+\frac{d(t,B_{r_0'}(x_0))^2}{a}\right)^{-k}a^{-\tfrac{3}{2}}dtda\\
		&= 4\pi C_2\|\wh{\psi}\|_{L^\infty}
		 \int_{a=\tfrac{1}{\sqrt{2}|\xi|}}^{a=\tfrac{2\sqrt{10}}{|\xi|}}
		\left(\int_{r=r_0-r_0'}^{\infty}
		\left(1+a^{-1}r^2\right)^{-k}r dr\right) a^{-\tfrac{3}{2}}
		da\\
		&=
		\frac{2\pi C_2\|\wh{\psi}\|_{L^\infty}}{k-1} \int_{a=\tfrac{1}{\sqrt{2}|\xi|}}^{a=\tfrac{2\sqrt{10}}{|\xi|}}a^{-\tfrac{1}{2}}\left(1+a^{-1}(r_0-r_0')^2\right)^{-k+1}da\\
		&
\leq C_3 |\xi|^{-\tfrac{1}{2}}\left(1+\frac{|\xi|}{2\sqrt{10}}(r_0-r_0')^2\right)^{-k+1}, \end{align*}
	where $C_3:=\frac{(4\sqrt{5}-1)\pi C_2 \|\wh{\psi}\|_{L^\infty}}{2^{1/4}(k-1)}<\infty$ for $k>1$.
	Consequently, taking $k>\max(m+2,1)$, we conclude that $\int_{U_{\pm \xi_0,\epsilon}} |\wh{h_2}(\xi)|^2|\xi|^{2m} d\xi<\infty$.
\end{proof}

The same technique in the proof of Proposition~\ref{convg2} for estimating $g_2$  can be applied to estimate $g_4$ for the vertical shearlet coefficients as well.
We now estimate $\int_{U_{\pm \xi_0,\epsilon}}
|\widehat{g_1}(\xi)|^2|\xi|^{2m}d\xi$ as follows.
The same technique can be used to estimate
$g_3$.

\begin{prop}\label{microlocalsquareintegrability}
	There exists a positive constant $c$ such that
	\begin{equation}\label{g1xisquqre}	
		\int_{U_{\pm \xi_0,\epsilon}}
		|\widehat{g_1}(\xi)|^2|\xi|^{2m} d\xi
		\leq  c \int_{a=0}^{a=1}
		\int_{s\in I_{s_0}} \int_{t\in B_{r_0}(x_0)}
		\left|\mathcal{SH}_{\psi}(\varphi f)(a,s,t)\right|^2a^{-2m-3}dt dsda,
	\end{equation}
	where $s_0:=\tan(\mbox{Arg}(\xi_0))$ and $r_0$ is defined in \eqref{r0}.
\end{prop}

\begin{proof}
	To prove \eqref{g1xisquqre},
	by the definition of $\wh{g_1}$ in \eqref{g1}, for $\xi \in U_{\pm \xi_0,\epsilon}$, we get
	\begin{small}
		\begin{align*}
			&\wh{g_1}(\xi)|\xi|^m
			=\chi_{\cC}(\xi)
			\int_{a=0}^{a=1}\int_{s\in I_{s_0}} \int_{t\in B_{r_0}(x_0)}
			 (a|\xi|)^m\widehat{\psi_{ast}}(\xi)\mathcal{SH}_{\psi}(\varphi f)(a,s,t)a^{-m-3} ds dt da\\
			&=\chi_{\cC}(\xi)
			\int_{a=0}^{a=1}\int_{s\in I_{s_0}}\int_{t\in B_{r_0}(x_0)}
			\mathcal{SH}_{\psi}(\varphi f)(a,s,t)a^{-m}e^{-2\pi i \xi\cdot t}(a|\xi|)^m\widehat{\psi_{as0}}(\xi)a^{-3} ds dt da\\	 &=\chi_{\mathcal{C}}(\xi)\int_{a=0}^{a=1}\left(\int_{s\in I_{s_0}}\wh{A}(a, s, \xi)(a|\xi|)^m\widehat{\psi_{as0}}(\xi) a^{-2}ds\right)\frac{da}{a},
		\end{align*}
	\end{small}
	where we define
	\begin{equation*}\label{Aasxi}
		\wh{A}(a,s,\xi):=\int_{t\in B_{r_0}(x_0)} \mathcal{SH}_{\psi}(\varphi f)(a,s,t)a^{-m} e^{-2\pi i \xi\cdot t}dt,
	\end{equation*}
	which is just the Fourier transform of $\chi_{B_{r_0}(x_0)}(\cdot)\mathcal{SH}_{\psi}(\varphi f)(a,s,\cdot)a^{-m}$.
	Because the function $\wh{\psi_1}$ is supported inside $[-2,-\frac{1}{2}]\cup [\frac{1}{2},2]$, we observe that
	$$
	 \wh{\psi_1}(a\xi_1)=\wh{\psi_1}(a\xi_1)
	\chi_{J_\xi}(a)
	\quad \mbox{with}\quad
	J_{\xi}:=[\tfrac{1}{2|\xi_1|}, \tfrac{2}{|\xi_1|}].
	$$
	Using  \eqref{psidef} and \eqref{hatpsiast}, we obtain
	\begin{align*}
		\wh{g_1}(\xi)|\xi|^m &=\chi_{\mathcal{C}}(\xi)\int_{a=0}^{a=1}a^{-\frac{3}{4}}(a|\xi|)^m\wh{\psi_1}(a\xi_1)\chi_{J_\xi}(a)\wh{\mathcal{A}}(a,\xi)\frac{da}{a},
	\end{align*}
	where $\wh{\mathcal{A}}(a,\xi):=\int_{s\in I_{s_0}}\wh{A}(a,s,\xi)\wh{\psi_2}\left(a^{-\frac{1}{2}}\left(\frac{\xi_2}{\xi_1}-s\right)\right) a^{-\frac{1}{2}}ds$. Hence,
	for $\xi=(\xi_1,\xi_2)\in U_{\pm \xi_0,\epsilon}$,
	we get
	 \begin{eqnarray*}\label{cachyschwartz}
		|\wh{g_1}(\xi)|^2|\xi|^{2m} & \leq & \chi_{\mathcal{C}}(\xi)
		\left(\int_{a\in (0,1)\cap J_\xi}
		\left| a^{-\frac{3}{4}}
		(a|\xi|)^m\wh{\psi_1}(a\xi_1) \wh{\mathcal{A}}(a,\xi)\right|\frac{da}{a}\right)^2\nonumber \\
		&\leq & \chi_{\mathcal{C}}(\xi)\left(\int_{a\in (0,1)\cap J_\xi}
		\left|a^{-\frac{3}{4}}(a|\xi|)^m
		 \wh{\psi_1}(a\xi_1)\wh{\mathcal{A}}(a,\xi)\right|^2\frac{da}{a}\right)
		 \left(\int_{\frac{1}{2|\xi_1|}}^{\frac{2}{|\xi_1|}} \frac{da}{a}\right)\\
		&=& (\log 4)  \chi_{\mathcal{C}}(\xi)\left(\int_{a\in (0,1)\cap J_\xi}
		\left|a^{-\frac{3}{4}}(a|\xi|)^m
		\wh{\psi_1}(a\xi_1) \wh{\mathcal{A}}(a,\xi)\right|^2\frac{da}{a}\right).
	\end{eqnarray*}
	By $\supp(\wh{\psi_2})\subseteq [-1,1]$, defining
	$J_{\xi,a}:=\left[
	\frac{\xi_2}{\xi_1}-\sqrt{a}, \frac{\xi_2}{\xi_1}+\sqrt{a}\right]$,
	we observe that
	$$
	 \wh{\psi}_2\left(a^{-\frac{1}{2}}\left(\frac{\xi_2}{\xi_1}-s \right)\right)
	 =\wh{\psi}_2\left(a^{-\frac{1}{2}}\left(\frac{\xi_2}{\xi_1}-s \right)\right) \chi_{J_{\xi,a}}(s).
	$$
	Also by $\|\wh{\psi_2}\|_{L^2}=1$, we deduce that
	$$
	\int_{s\in I_{s_0}}
	 \left|\wh{\psi}_2\left(a^{-\frac{1}{2}}\left(\frac{\xi_2}{\xi_1}-s \right)\right) \right|^2 \frac{ds}{a}
	=a^{-\frac{1}{2}}\int_{s\in a^{-\frac{1}{2}}(\frac{\xi_2}{\xi_1}-I_{s_0})}
	|\wh{\psi_2}(s)|^2 ds\le a^{-\frac{1}{2}}.
	$$
	Using the  Cauchy-Schwartz inequality and the definition of $\wh{\mathcal{A}}(a,\xi)$, we have
	\begin{small}
		\begin{align*} \left|\wh{\mathcal{A}}(a,\xi)\right|^2
			&\leq \left(\int_{s\in I_{s_0}}\left|\wh{\psi}_2\left(a^{-\frac{1}{2}}\left(\frac{\xi_2}{\xi_1}-s \right)\right)\right|^2\frac{ds}{a}\right)\left(\int_{s\in I_{s_0}}\left|\wh{A}(a,s,\xi)\right|^2 \chi_{J_{\xi,a}}(s)ds\right) \\
			&\leq a^{-\frac{1}{2}}
			\int_{s\in I_{s_0}\cap J_{\xi,a}}\left|\wh{A}(a,s,\xi)\right|^2 ds.
		\end{align*}
	\end{small}
	Therefore, we proved
	{\small
		$$
		|\wh{g_1}(\xi)|^2|\xi|^{2m}
		\le (\log 4)  \chi_{\mathcal{C}}(\xi)
		\int_{a\in (0,1)\cap J_\xi}
		\int_{s\in I_{s_0}\cap J_{\xi,a}}
		(a|\xi|)^{2m}
		|\wh{\psi_1}(a\xi_1)|^2
		|\wh{\mathcal{A}}(a,s,\xi)|^2 a^{-3} dsda.
		$$}
	Note that $|\wh{\psi_1}(a\xi_1)|^2 \chi_{J_{\xi,a}}(s)
	\le \|\wh{\psi_1}\|_{L^\infty}^2 \chi_{\Xi(a,s)}(\xi)$.
	Using the inequality in \eqref{est:Xi}, we have
	\begin{equation}\label{axi2m}
		(a|\xi|)^{2m}\le \max(2^{-m},40^{m}), \qquad \forall\,
		\xi\in \Xi(a,s), a\in (0,1), s\in [-2,2].
	\end{equation}
	%
	
	Putting all the above estimates together we obtain the following estimate
	\begin{small}
		\begin{align*}
			\int_{U_{\pm \xi_0,\epsilon}}
			|\wh{g_1}(\xi)|^2|\xi|^{2m} d\xi &
			\leq c \int_{\xi \in U_{\pm \xi_0,\epsilon}\cap \mathcal{C}}
			\int_{a\in (0,1)\cap J_\xi}
			\int_{s\in I_{s_0}\cap J_{\xi,a}}
			|\wh{A}(a,s,\xi)|^2  a^{-3} ds da d\xi\\
			&\le c
			\int_{a=0}^{a=1}\int_{s\in I_{s_0}}
			\left( \int_{\xi \in \R^2}
			|\wh{A}(a,s,\xi)|^2 d\xi \right) a^{-3} ds da\\
			&=c \int_{a=0}^{a=1} \int_{s\in I_{s_0}}
			\int_{t \in B_{r_0}(x_0)}
			|\mathcal{SH}_\psi(\varphi f)(a,s,t)(a,s,t)|^2 a^{-2m-3} dt ds da,
		\end{align*}
	\end{small}
	where $c:=(\log 4)\|\wh{\psi_1}\|_{L^\infty}^2 \max(2^{-m}, 40^m)$ and we noticed that $\wh{A}(a,s,\xi)$ is the Fourier transform of $\chi_{B_{r_0}(x_0)}(\cdot) \mathcal{SH}_\psi(\varphi f)(a,s,\cdot) a^{-m}$.
	This proves \eqref{g1xisquqre}.
\end{proof}

To prove Theorem~\ref{thm:swf:integral}, we need the following two auxiliary results.

\begin{lemma}\label{lem:nicepart}
	Let $f\in \mathcal{D}'(\Omega)$ be a distribution and $x_0\in \Omega$.
	For any $r_0>0$ and a bounded interval $I_{s_0}$ of $[-2,2]$ such that $B_{r_0}(x_0)\subseteq \Omega$,
	if  $\varphi\in \DCO$ satisfies $\supp(\varphi)\subseteq B_{r_0}(x_0)$,
	then
	\begin{equation}\label{est:int1}
		\int_{a=0}^{a=1}
		 \int_{s=s_0-r_0}^{s=s_0+r_0}\int_{t\in B_{r_0}(x_0)}\left|\left\langle (1-\varphi) f, \psi_{ast} \right\rangle\right|^2a^{-2m-3}dt dsda<\infty.
	\end{equation}
	Consequently,  for $\varphi\in \DCO$ with $\supp(\varphi)\subseteq B_{r_0}(x_0)$,
	the following two integrals
	$$
	\int_{a=0}^{a=1}\int_{s\in I_{s_0}}\int_{t\in B_{r_0}(x_0)}
	\left|\mathcal{SH}_{\psi}(\varphi f)(a,s,t)\right|^2a^{-2m-3}dt ds da
	$$
	and
	$$
	\int_{a=0}^{a=1}
	\int_{s\in I_{s_0}}\int_{t\in B_{r_0}(x_0)} \left|\mathcal{SH}_{\psi}f(a,s,t)\right|^2a^{-2m-3}dt ds da
	$$
	are either both finite (i.e., convergent) or both infinite (i.e., divergent).
\end{lemma}

\begin{proof}
	By \cite[Lemma~5.2]{kutyniok2009resolution}, for each $k>0$, there exists a positive constant $C$ depending on $k$ and satisfying
	\begin{equation}\label{est:decayk}
		|\langle (1-\varphi) f, \psi_{ast} \rangle| \leq C a^{\frac{1}{4}}\left(1+\tfrac{\eta^2}{a}\right)^{-k}
		=C a^{k+\frac{1}{4}} (a+\eta^2)^{-k}
		\le C \eta^{-2k} a^{k+\frac{1}{4}}
	\end{equation}
	for all $a\in (0,1)$, $s\in [-2,2]$ and $t\in \R^2$, where $\eta=\mbox{dist}(B_{r_0}(x_0),(\supp \varphi)^c)>0$.
	Hence, we get
	$$
	\int_{a=0}^{a=1}
	 \int_{s=s_0-r_0}^{s=s_0+r_0}\int_{t\in B_{r_0}(x_0)}\left|\left\langle (1-\varphi) f, \psi_{ast} \right\rangle\right|^2 a^{-2m-3}dtdsda\nonumber
	\leq  \frac{C^2}{\eta^{4k}} \int_{0}^{1}a^{2k-2m-\frac{5}{2}} da.
	$$
	Note that the  above last integral $\int_{0}^{1}a^{2k-2m-\frac{5}{2}} da<\infty$ if and only if   $k>m+\frac{3}{4}$.
	Taking $k>\max(0, m+\frac{3}{4})$,
	we conclude that the integral in \eqref{est:int1} is finite.
	
	On the other hand, it is trivial that
	\begin{equation} \label{SH:decomp}
		\mathcal{SH}_{\psi}f(a,s,t)=
		\langle f, \psi_{ast} \rangle
		=\langle \varphi f, \psi_{ast} \rangle+\langle (1-\varphi)f, \psi_{ast} \rangle.
	\end{equation}
	Hence, using the inequality $(b+c)^2\le 2b^2+2c^2$ for all $b,c\in \R$, we have
	$$
	|\langle f, \psi_{ast} \rangle|^2\le 2|\langle \varphi f, \psi_{ast} \rangle|^2+
	2 |\langle (1-\varphi)f, \psi_{ast} \rangle|^2
	$$
	and
	$$
	|\langle \varphi f, \psi_{ast} \rangle|^2\le 2|\langle f, \psi_{ast} \rangle|^2+
	2 |\langle (1-\varphi)f, \psi_{ast} \rangle|^2.
	$$
	Consequently, by \eqref{est:int1},
	the two integrals are either both finite or both infinite.
\end{proof}

\begin{lemma}\label{lem:sh:a>r0}
	Let $f$ be a tempered distribution on $\R^2$.
	For $x_0\in \R^2$ and $s_0\in \R$,
	\begin{equation}\label{sh:a>r0}
		\int_{a=a_0}^{a=1}
		\int_{s=s_0-r_0}^{s=s_0+r_0}
		\int_{t\in B_{r_0}(x_0)} \left|\mathcal{SH}_{\psi}f(a,s,t)\right|^2a^{-2m-3}dt ds da<\infty
	\end{equation}
	for all $a_0>0$, $r_0>0$ and $m\in \R$.
\end{lemma}

\begin{proof}
	For all $a\in [a_0,1)$ and $s\in (s_0-r_0,s_0+r_0)$, the support $\Xi(a,s)$ of $\wh{\psi_{ast}}$ is contained inside a finite ball $B_R(0)$ for some $0<R<\infty$.
	Take $\eta\in \mathcal{D}(\R^2)$ such that $\eta=1$ on $B_{R+1}(0)$.
	Therefore,
	$$
	|\langle f, \psi_{ast}\rangle|=
	|\langle \eta f, \psi_{ast}\rangle|=
	|\langle \wh{\eta f}, \wh{\psi_{ast}}\rangle|
	\le C a^{3/4}
	\quad \mbox{with}\quad
	C:=\|\wh{\psi}\|_{L^\infty} \sup_{|\xi|\le R} |\wh{\eta f}(\xi)|<\infty,
	$$
	since $\eta f$ is a compactly supported distribution and hence $\wh{\eta f}$ is continuous.
	Therefore,
	\begin{align*}
		\int_{a=a_0}^{a=1}
		&\int_{s=s_0-r_0}^{s=s_0+r_0}
		\int_{t\in B_{r_0}(x_0)}
		|\langle f, \psi_{ast}\rangle|^2 a^{-2m-3} dtds da\\
		&\le
		C^2 \int_{a=a_0}^{a=1} \int_{s=s_0-r_0}^{s=s_0+r_0}
		\int_{t\in B_{r_0}(x_0)} a^{-2m-3/2} dt ds da
		<\infty,
	\end{align*}
	since $\int_{a=a_0}^{a=1} a^{-2m-3/2}da<\infty$ due to $a_0>0$.
	This proves \eqref{sh:a>r0}.
\end{proof}

We are now ready to prove our main results in Theorems~\ref{thm:swf:integral} and~\ref{thm:swf:decay}.

\begin{proof}[Proof of Theorem~\ref{thm:swf:integral}]
	Sufficiency. Suppose that \eqref{squareintegrabilityofshearlet}  or \eqref{1squareintegrabilityofshearlet} holds at $(x_0, \xi_0)$ for some $r_0>0$.
	Since $s_0:=\tan(\mbox{Arg}(\xi_0))$, without loss of generality, we assume $|s_0|<2$ and \eqref{squareintegrabilityofshearlet} holds.
	By Proposition~\ref{convg2}, Proposition~\ref{microlocalsquareintegrability} and Lemma~\ref{lem:nicepart},
	we deduce from \eqref{decom} that
	\eqref{sobolevwf:2} holds. Hence, by Lemma~\ref{lem:swf:def}, $(x_0,\pm \xi_0)\not\in \wf_m(f)$.
	
	Necessity. Suppose that $(x_0, \pm \xi_0)\not\in \wf_m(f)$.
	Note that $s_0:=\tan(\mbox{Arg}(\xi_0))$.
	Without loss of generality, we assume $|s_0|<2$.
	We have to prove
	 \eqref{squareintegrabilityofshearlet}.
	To do so, we first prove that
	\begin{equation}\label{est:int2}
		\int_{a=0}^{a=1}
		 \int_{s=s_0-r_0}^{s=s_0+r_0}\int_{t\in B_{r_0}(x_0)}\left|\left\langle \varphi f, \psi_{ast}
		 \right\rangle\right|^2a^{-2m-3}dtdsda<\infty.
	\end{equation}
	Since $(x_0,\pm \xi_0)\not \in \wf_m(f)$, by Lemma~\ref{lem:swf:def},
	there exists $\epsilon>0$ such that
	\eqref{sobolevwf:2} holds, that is,
	$C:=\int_{U} |\wh{\varphi f}(\xi)|^2 |\xi|^{2m} d\xi<\infty$, where
	$U:=U_{\xi_0,\epsilon} \cup (-U_{\xi_0,\epsilon})$.
	To prove \eqref{est:int1}, using the relation $\supp(\wh{\psi_{ast}})\subseteq \Xi(a,s)$ in \eqref{supp:hatpsiast}, we observe
	$$
	\langle \varphi f, \psi_{ast}\rangle=
	\langle \wh{\varphi f}, \wh{\psi_{ast}}\rangle
	=\int_{\R^2} \wh{\varphi f}(\xi) \overline{\wh{\psi_{ast}}(\xi)} d\xi=
	\int_{\Xi(a,s)} \wh{\varphi f}(\xi) \overline{\wh{\psi_{ast}}(\xi)} d\xi.
	$$
Using the definition of $\Xi(a,s)$ in \eqref{xias} and $s_0:=\tan(\mbox{Arg}(\xi_0))$,
	we observe that there exist sufficiently small $r_0>0$ and $a_0\in (0,1)$ such that
	the set
	$\Xi(a,s)$ must be contained inside $U$ for all $a\in (0,a_0)$ and $s\in (s_0-r_0,s_0+r_0)$.
Therefore, for all $a\in (0,a_0)$ and $s\in (s_0-r_0,s_0+r_0)$,
	using $\wh{\psi_{ast}}(\xi)=a^{3/4} e^{-i2\pi t\cdot\xi} \wh{\psi}(M_{as}^T\xi)$,
	we have
	\begin{align*}
		|\langle \varphi f, \psi_{ast}\rangle|^2
		&=\left|\int_{\Xi(a,s)} \wh{\varphi f}(\xi) \overline{\wh{\psi_{ast}}(\xi)} d\xi\right|^2
		\le a^{3/2} \left(\int_{U}
		|\wh{\varphi f}(\xi) \wh{\psi}(M_{as}^T\xi)| d\xi\right)^2\\
		&\le a^{3/2} \left(\int_{U}
		|\wh{\varphi f}(\xi)|^2 |\xi|^{2m} d\xi \right)
		\left(\int_U |\wh{\psi}(M_{as}^T\xi)|^2|\xi|^{-2m} d\xi\right).
	\end{align*}
	Since $U\subseteq \Xi(a,s)$ and $\|\wh{\psi}\|_{L^2}=1$, using \eqref{axi2m} with $m$ being replaced by $-m$, we have
	\begin{align*}
		|\langle \varphi f, \psi_{ast}\rangle|^2
		&\le C a^{2m+3/2} \int_{U} |\wh{\psi}(M_{as}^T\xi)|^2(a|\xi|)^{-2m} d\xi\\
		&\le C_1 a^{2m+3/2} \int_U |\wh{\psi}(M_{as}^T\xi)|^2 d\xi
		\le C_1 a^{2m+3} \int_{\R^2} |\wh{\psi}(\xi)|^2 d\xi=
		C_1 a^{2m+3},
	\end{align*}
	where $C_1:=C\max(2^{m},40^{-m})$.
	Therefore, we conclude that
	\begin{align*}
		\int_{a=0}^{a=a_0}
		&\int_{s=s_0-r_0}^{s=s_0+r_0}
		\int_{t\in B_{r_0}(x_0)}
		|\langle \varphi f, \psi_{ast}\rangle|^2 a^{-2m-3} dtds da\\
		&\le
		C_1\int_{a=0}^{a=a_0} \int_{s=s_0-r_0}^{s=s_0+r_0}
		\int_{t\in B_{r_0}(x_0)} dtdsda<\infty.
	\end{align*}
	Since $a_0>0$ and $r_0>0$,  the inequality \eqref{sh:a>r0} in Lemma~\ref{lem:sh:a>r0} must hold with $f$ being replaced by $\varphi f$.
	Consequently, the integral in \eqref{est:int2} is finite.
	Using Lemma~\ref{lem:nicepart} and \eqref{est:int2}, we conclude that
	the integral in \eqref{squareintegrabilityofshearlet}
	is finite.
\end{proof}

\begin{proof}[Proof of Theorem~\ref{thm:swf:decay}]
	Let $(x_0, (r,rs_0))\in \Lambda_1(m+1+\epsilon)$ with $s_0\in (-2,2)$ and $r>0$. By the definition of the set $\Lambda_1(m)$ in Theorem~\ref{thm:swf:decay}, there exist $0<r_0<1$ and $C>0$ such that \eqref{cst:decay} holds with $k:=m+1+\epsilon$.
	Therefore, by \eqref{cst:decay}, we have
	\begin{align*}
		\int_{a=0}^{a=r_0}
		 \int_{s=s_0-r_0}^{s=s_0+r_0}\int_{t\in B_{r_0}(x_0)}&\left|\mathcal{SH}_{\psi}f(a,s,t)\right|^2a^{-2m}\frac{da}{a^3}\ ds\ dt\\
		& \leq  \int_{a=0}^{a=r_0}
		 \int_{s=s_0-r_0}^{s=s_0+r_0}\int_{t\in B_{r_0}(x_0)}
		C^2 a^{2(m+1+\epsilon)} a^{-2m-3}dt dsda\\
		&= C^2 \pi r_0^3  \epsilon^{-1} r_0^{2\epsilon}
		<\infty.
	\end{align*}
	
	Since $r_0>0$,  the inequality \eqref{sh:a>r0} in Lemma~\ref{lem:sh:a>r0} must hold with $a_0=r_0$. This shows that \eqref{squareintegrabilityofshearlet} holds. By Theorem~\ref{thm:swf:integral}, we conclude that $(x_0, (r,rs_0))\not \in \wf_{m}(f)$. This proves $\Lambda_1(m+1+\epsilon)\subseteq \wf_m(f)^c$. Similarly, we have $\Lambda_2(m+1+\epsilon)\subseteq \wf_m(f)^c$. This proves  \eqref{wfm:cst:decay}.
	
	Note that $\Lambda_j(n)\subseteq \Lambda_j(m)$ for all $m<n$.
	To prove \eqref{wfm:cst:decay:2},
	by Proposition~\ref{sobolevineq} and \eqref{wfm:cst:decay} with $\epsilon=1$,
	we have
	\begin{align*}
		 \wf(f)^c=\left(\bigcap\limits_{k\in \N}(\wf_k(f))^c\right)^{\mathrm{o}}   \supseteq  \left(\bigcap\limits_{k\in \N}\Lambda_1(k+2)\right)^{\mathrm{o}}
		= \left(\bigcap\limits_{k\in \N}\Lambda_1(k)\right)^{\mathrm{o}}  =\Lambda_1(\infty).
	\end{align*}
	Similarly, we can prove $\Lambda_2(\infty)\subseteq \wf(f)^c$.
	Therefore, $\Lambda_1(\infty)\cup \Lambda_2(\infty)\subseteq \wf(f)^c$.
	
	Conversely, let $(x_0, (r,rs_0))\in \wf(f)^c$. Without loss of generality, we assume $|s_0|<2$.
	We use the decomposition in \eqref{SH:decomp}. Then we proved in
	the proof of Lemma~\ref{lem:nicepart}
	that  \eqref{est:decayk} holds for all $k>0$.
	Hence, $|\mathcal{SH}_\psi((1-\varphi)f)(a,s,t)|=
	|\langle (1-\varphi) f, \psi_{ast} \rangle|=\mathcal{O}(a^k)$ as $a\to 0^+$ uniformly when $(s,t)$ is near $(x_0,s_0)$ for each $k>0$.
	
	On the other hand, since $(x_0, (r,rs_0))\in \wf(f)^c$, by Definition~\ref{cinfinitywavefront}, \eqref{cinfinitydefn} holds with $\xi_0:=(r, rs_0)$. Then there exists $r_0>0$ such that $(r,rs)\in V_{\xi_0,\epsilon}$ for all $s\in (s_0-2r_0,s_0+2r_0)$.
	Hence, for $a\in (0,1)$ and $|s-s_0|<r_0$, we have $\Xi(a,s)\subseteq U_{\xi_0,\epsilon}=:U$.
	As we argued in the proof of Theorem~\ref{thm:swf:integral},
	we have
	\begin{align*}
		|\langle \varphi f, \psi_{ast}\rangle|^2
		&=\left|\int_{\Xi(a,s)} \wh{\varphi f}(\xi) \overline{\wh{\psi_{ast}}(\xi)} d\xi\right|^2
		\le a^{3/2} \left(\int_{U}
		|\wh{\varphi f}(\xi) \wh{\psi}(M_{as}^T\xi)| d\xi\right)^2\\
		&\le a^{3/2} \left(\int_{U}
		|\wh{\varphi f}(\xi)|^2 |\xi|^{2m} d\xi \right)
		\left(\int_{U} |\wh{\psi}(M_{as}^T\xi)|^2|\xi|^{-2m} d\xi\right).
	\end{align*}
	Consider $m\in \N$.
	Using \eqref{cinfinitydefn} with $N=m+2$ and observing $U\subseteq V_{\xi_0,\epsilon}$, we see
	$$
	C:=\int_{U} |\wh{\varphi f}(\xi)|^2 |\xi|^{2m} d\xi \le
	C_{m+2}^2 \int_{U} \frac{|\xi|^{2m}}{(1+|\xi|)^{2m+4}} d\xi<
	C_{m+2}^2 \int_{\R^2} (1+|\xi|)^{-4} d\xi<\infty,
	$$
	where $C_N:=\sup_{\xi\in V_{\xi_0,\epsilon}} (1+|\xi|)^N |\wh{\varphi f}(\xi)|<\infty$.
	Consequently, for all $a\in (0,1)$ and $s\in (s_0-r_0,s_0+r_0)$,
	\begin{align*}
		|\langle \varphi f, \psi_{ast}\rangle|^2
		&\le C a^{2m+3/2} \int_{U} |\wh{\psi}(M_{as}^T\xi)|^2(a|\xi|)^{-2m} d\xi\\
		&\le C_1 a^{2m+3/2} \int_U |\wh{\psi}(M_{as}^T\xi)|^2 d\xi
		\le C_1 a^{2m+3} \int_{\R^2} |\wh{\psi}(\xi)|^2 d\xi=
		C_1 a^{2m+3}\|\wh{\psi}\|_{L^2}^2,
	\end{align*}
	where $C_1:=C\max(2^{m},40^{-m})$.
	Combining with the established estimate for $\mathcal{SH}((1-\varphi)f)(a,s,t)$,
	we conclude that
	$|\mathcal{SH}_\psi f(a,s,t)|=\mathcal{O}(a^{2m+3})$ as $a\to 0^+$ uniformly when $(s,t)$ is near $(s_0,x_0)$. This proves that $(x_0,(r,rs_0))\in \Lambda_1(2m+3)$ for all $m\in \N$. Consequently, $(x_0, (r,rs_0))\in \cap_{k\in \N} \Lambda_1(k)$.
	Since $\wf(f)^c$ is open,
	we must have $(x_0, (r,rs_0))\in
	(\cap_{k\in \N} \Lambda_1(k))^{\mathrm{o}}=\Lambda_1(\infty)$.
	This proves $\wf(f)^c \subseteq \Lambda_1(\infty)\cup \Lambda_2(\infty)$ and hence \eqref{wfm:cst:decay:2} must hold.
\end{proof}

The wavefront set of a distribution depends on the asymptotic behavior of its Fourier transform after localization. But this task is difficult in general. In a practical situation, one can handle a finite number of samples of a signal and  the main task is to identify the wavefront set through a discrete transform. Moreover, Andrade-Loarca et al. \cite{andrade2019extraction} introduced an algorithmic approach to extract the digital wavefront set of an image.
Due to the importance of theoretical and practical aspects of detection of singularity along curves of the boundary of characteristic function some interesting study have been done by many authors including Guo and Labate in \cite{guo2018detection,guo2017microlocal} for discrete shearlet transform with bandlimited shearlets, Kutyniok and Petersen \cite{kutyniok2017classification} for 2D and 3D continuous transform with compactly supported shearlet, etc.
A microlocal Sobolev space is an efficient tool to extract wavefront sets in a discrete transform using wave packets \cite{de2014exact}. So, characterizing wavefront sets through microlocal Sobolev spaces is appealing in applications.

\section{Proofs of Theorems~\ref{thm:2micro} and~\ref{thm:2microsobolevwavefront} on $2$-microlocal spaces}
\label{sec:holdersobolev}

Definition~\ref{def:mcs} of a $2$-microlocal space $C_{x_0}^{\mcs,\mcs'}$ describes the local behavior of a tempered distribution $f$ near a point $x_0$.
To obtain sufficient information about a distribution, it is necessary to characterize a $2$-microlocal space $C_{x_0}^{\mcs,\mcs'}$ in Definition~\ref{def:mcs} in  the time/spatial domain.
Jaffard in \cite{jaffard1991pointwise} showed  that  $C_{x_0}^{\mcs,\mcs'} \subseteq C_{x_0}^{\mcs} \subseteq C_{x_0}^{\mcs, -\mcs'}$ for $\mcs+\mcs'>0$ and
\cite[Theorem~2]{jaffard1991pointwise}
characterizes $C_{x_0}^{\mcs,\mcs'}(\R^d)$ using a wavelet $\psi$ as follows:
\begin{align}\label{jaffard}
	f \in C_{x_0}^{\mcs,\mcs'} (\mathbb R^d) \quad \mbox{if and only if}\quad 	 |\langle f, \psi_{j, k}\rangle | \leq C 2^{-(d/2+\mcs) j} (1+\|k-2^jx_0\|)^{-\mcs'},
\end{align}
for all $j\in \Z$ and $k\in \Z^d$,
where $\psi_{j,k}:=2^{dj/2} \psi(2^j\cdot-k)$.
V\'ehel et al. \cite{kolwankar2002time,seuret2003time}  obtained a time-domain characterization of local H\"older spaces $C_{x_0}^{\mcs}(\R)$ and $2$-microlocal spaces $C_{x_0}^{\mcs,\mcs'}(\R)$ by introducing  a new function space $K_{x_0}^{\mcs,\mcs'}(\R)$ in \cite[Definition~5]{seuret2003time}.
Define
$|\alpha|:=\alpha_1+\cdots+\alpha_d$ for $\alpha=(\alpha_1,\ldots,\alpha_d)^T \in \N_0^d$. Recall that the floor function $\lfloor\mcs\rfloor$ for $\mcs\in \R$ is the largest integer such that $\lfloor \mcs\rfloor\le \mcs$.
We now recall the definition of the space $K_{x_0}^{\mcs,\mcs'}(\R^d)$ and its relationship with local H\"older spaces and  $2$-microlocal spaces.

\begin{definition}\label{twomicrolocaldef}
	Let $x_0\in \R^d$.
	We say that a function $f$ on $\R^d$ belongs to $K_{x_0}^{\mcs,\mcs'}(\R^d)$ with $\mcs'\leq 0$ and $\mcs+\mcs'\geq 0$, if there exist $0<\delta<1/4$, a positive constant $C>0$ and polynomials $P_\alpha$ of degree less than $\lfloor\mcs\rfloor-m$ with $m:=\lfloor \mcs+\mcs'\rfloor$ such that
	 \begin{equation}\label{twomicrolocaleqn}
		\begin{split}
			&\left|\frac{\partial^\alpha f(x)-P_\alpha(x)}{\|x-x_0\|^{\lfloor\mcs\rfloor-m}}
			-\frac{\partial^\alpha f(y)-P_\alpha(y)}{\|y-x_0\|^{\lfloor\mcs\rfloor-m}}\right|\\
			&\qquad\qquad\qquad
			\leq C \|x-y\|^{\mcs+\mcs'-m}
			 \left(\|x-y\|+\|x-x_0\|\right)^{-\mcs'-\lfloor\mcs\rfloor+m},
		\end{split}
	\end{equation}
	for all  $\alpha\in \N_0^d$ with
	$|\alpha|=m$ and for all
	$x,y \in \R^d$ satisfying $0<\|x-x_0\|<\delta$ and $0<\|y-x_0\|<\delta$, where $\partial^\alpha f$ stands for the $\alpha$th partial derivative of the function $f$.
\end{definition}

Note that $m:=\lfloor \mcs+\mcs'\rfloor \le \lfloor \mcs\rfloor$ due to $\mcs'\le 0$.
If $m=\lfloor\mcs\rfloor$, then \eqref{twomicrolocaleqn} is equivalent to
\begin{equation} \label{kspaceform0}
	|\partial^\alpha f(x)- \partial^\alpha f(y)|\leq C \|x-y\|^{\mcs+\mcs'-m}\left(\|x-y\|+\|x-x_0\|\right)^{-\mcs'}.
\end{equation}
We now state the following time-domain characterization of a $2$-microlocal space $C_{x_0}^{\mcs,\mcs'}(\R^d)$ and a local H\"older space $C_{x_0}^{\mcs}(\R^d) $ through the space $K_{x_0}^{\mcs,\mcs'}(\mathbb{R}^d)$. The following result generalizes
\cite{seuret2003time,seuret2003time1}
from dimension one to higher dimensions. For the sake of completeness, we provide a proof here.

\begin{lemma}\label{lem:timecharcz2microlocal}
	Let $x_0 \in \mathbb{R}^d$ and $\mcs, \mcs'\in \mathbb{R}$ with $\mcs'\leq 0$ and $\mcs+\mcs'\geq 0$.
	\begin{itemize}
		
		\item[(i)]
		 $K_{x_0}^{\mcs,\mcs'}(\R^d)\subseteq C_{x_0}^{\mcs,\mcs'}(\R^d)$, and the equal sign holds if $\mcs+\mcs'\notin \mathbb{N}_0$ and $\mcs'<0$.
		\item[(ii)]  $K_{x_0}^{\mcs,\mcs'}(\R^d)\subseteq  C^{\mcs}_{x_0}(\R^d)$ provided that $\mcs+\mcs'\notin \mathbb{N}_0$ and $\mcs'<0$. Moreover,
		$K_{x_0}^{\mcs,-\mcs}(\R^d)= C^{\mcs}_{x_0}(\R^d)$ for $\mcs>0$ with $\mcs\not \in \N$.
	\end{itemize}
\end{lemma}

\begin{proof}
	(i) Let us first consider the case $m=\lfloor \mcs\rfloor$.
	Let $f\in K_{x_0}^{\mcs,\mcs'}(\R^d)$.
	Then \eqref{kspaceform0} holds. We shall prove $f\in C_{x_0}^{\mcs,\mcs'}(\R^d)$ using the wavelet characterization of a $2$-microlocal space.
	Let $\psi\in C^1(\R^d)$ be a compactly supported wavelet with at least one vanishing moment, i.e., $\int_{\R^d} \psi=0$.
	Since $\psi$ is a compactly supported function in $L^2(\R^d)$, there exists a positive constant $K$ such that
	$\|x-k2^{-j}\|\le K 2^{-j}$ for all $x\in \supp(\psi_{j,k})$ and $\int_{\supp(\psi_{j,k})} |\psi_{j,k}(x)|dx=\| \psi_{j,k}\|_{L^1(\R^d)}\le  K 2^{-dj/2}$.
	Let $\alpha\in \N_0^d$ with $|\alpha|=m$.
	Define $d_{j,k}:=\langle \partial^\alpha f, \psi_{j,k}\rangle$.
	For $j>\log_2(K/\delta)$, we observe $0<\|k2^{-j}-x_0\|<\delta$ and
	using \eqref{kspaceform0} and the vanishing moment of $\psi$,
	we have
	$\langle \partial^\alpha f(k2^{-j}),\psi_{j,k}\rangle=0$ and
	\begin{align*}
		|d_{j,k}|&=
		|\langle \partial^\alpha f, \psi_{j,k}\rangle|=
		\left|\int_{\supp(\psi_{j,k})}
		\left[\partial^\alpha f(x)-\partial^\alpha f(k2^{-j})\right]\psi_{j,k}(x)dx\right|\\
		&\leq  C \int_{\supp(\psi_{j,k})}
		 \|x-k2^{-j}\|^{\mcs+\mcs'-m}\left(\|x-k2^{-j}\|+\|x-x_0\|\right)^{-\mcs'} \left|\psi_{j,k}(x)\right|dx.
	\end{align*}
	Noting $\|x-x_0\|\leq \|x-k2^{-j}\|+\|k2^{-j}-x_0\|$ and using $\mcs+\mcs'-m\ge 0$ and $-\mcs'\ge 0$,
	we have
	$$
	|d_{j,k}|\leq  CK^{\mcs+\mcs'-m}2^{-j(\mcs-m)}\max(2K,1)(1+\|k-2^jx_0\|)^{-\mcs'} K 2^{-dj/2}.
	$$
	Consequently, we conclude that
	 \begin{equation}\label{wavelet2microlocal}
		|d_{j,k}|\leq C_1  2^{-(d/2+\mcs-m)j}(1+\|k-2^jx_0\|)^{-\mcs'},
		\qquad \forall\; j>\log_2(K/\delta), k\in \Z^d,
	\end{equation}
	where $C_1:=C K^{1+\mcs+\mcs'-m}\max(2K,1)$.
	The inequality in \eqref{wavelet2microlocal} for $j\le \log_2(K/\delta)$ and $k\in \Z^d$ is obvious.
	By the wavelet characterization in  \eqref{jaffard}, we have
	$\partial^\alpha f\in C_{x_0}^{\mcs-m,\mcs'}(\R^d)$ for all $\alpha\in \N_0^d$ with $|\alpha|=m$. Thus, we conclude that $f\in C_{x_0}^{\mcs,\mcs'}(\R^d)$.
	
	Conversely, let $f\in C_{x_0}^{\mcs,\mcs'}(\R^d)$ and $x\ne y\in \R^d$ such that $\|x-x_0\|<1/4$ and $\|y-x_0\|<1/4$. For all $\alpha\in \N_0^d$ with $|\alpha|=m$,
	\eqref{wavelet2microlocal} holds with $d_{j,k}:=\la \partial^\alpha f, \psi_{j,k}\ra$.
	Since $\|x-y\|<1$, let $J$ be the unique integer such that
	\begin{equation}\label{con2micro}
		2^{1-J}\leq \|x-y\|< 2^{-J},\quad \mbox{or equivalently}\quad
		\|y-x\|< 2^{-J} \le \|y-x\|/2.
	\end{equation}
	We deduce from the wavelet representation $\partial^\alpha f(x)=\sum_{j\in \Z}\sum_{k\in \Z^d} d_{j,k}\psi_{j,k}(x)$ with $d_{j,k}:=\la \partial^\alpha f, \psi_{j,k}\ra$ that
	\begin{align*}
		|\partial^\alpha f(x)-\partial^\alpha f(y)| \le &\underbrace{\sum_{j< J}\sum_{k\in \Z^d}|d_{j,k}|\left|\psi_{j,k}(x)-\psi_{j,k}(y)\right|}_\text{\textrm{I}}\\
		&+\underbrace{\sum_{j\ge J}\sum_{k\in \Z^d}|d_{j,k}|\left|\psi_{j,k}(x)\right|}_{\text{\textrm{II}}_x}
		+\underbrace{\sum_{j\ge J}\sum_{k\in \Z^d}|d_{j,k}|\left|\psi_{j,k}(y)\right|}_{\text{\textrm{II}}_y}.
	\end{align*}
	We first estimate the term $\textrm{II}_x$. We already proved that $\|x-k2^{-j}\|\le K 2^{-j}$ for all $x\in \supp(\psi_{j,k})$. Hence, for $x\in \supp(\psi_{j,k})$,
	$$
	\|k-2^j x_0\|=2^{j}\|k2^{-j}-x_0\|\le 2^j(\|k2^{-j}-x\|+\|x-x_0\|)\le K+2^j\|x-x_0\|.
	$$
	Because $\psi$ has compact support, there is $M>0$ depending only on $\supp(\psi)$ such that
	\begin{equation}\label{supp:psi}
		\sum_{k\in \Z^d} \chi_{\supp(\psi_{j,k})}(x)
		=\sum_{k\in \Z^d}
		\chi_{\supp(\psi)}(2^jx-k)
		\le M,\qquad \forall\; j\in \Z, x\in \R^d.
	\end{equation}
	Since $\partial^\alpha f\in C_{x_0}^{\mcs-m,\mcs'}(\R^d)$,
	\eqref{wavelet2microlocal} holds and using $|\psi_{j,x}(x)|\le 2^{dj/2}\|\psi\|_{L^\infty} \chi_{\supp(\psi_{j,k})}(x)$, we have
	\begin{align*}
		\textrm{II}_x
		&\le C_1 \sum_{j\ge J} \sum_{k\in \Z^d} 2^{-(d/2+\mcs-m)j}(1+\|k-2^jx_0\|)^{-\mcs'}|\psi_{j,k}(x)|\\
		&\le C_1 \sum_{j\ge J} \sum_{k\in \Z^d} 2^{-(d/2+\mcs-m)j}(K+1+2^j \|x-x_0\|)^{-\mcs'}|\psi_{j,k}(x)|\\
		&\le C_1 M \|\psi\|_{L^\infty} (\|x-x_0\|+2^{-J}(K+1))^{-\mcs'}
		\sum_{j\ge J} 2^{-(\mcs+\mcs'-m)j}\\
		&\le C_2 (\|x-x_0\|+2^{1-J})^{-\mcs'} 2^{-(\mcs+\mcs'-m)J},
	\end{align*}
	where
	$$
	C_2:=C_1 M \|\psi\|_{L^\infty}
	\max(1, (K/2+1/2)^{-\mcs'}) (1-2^{-(\mcs+\mcs'-m)})^{-1}<\infty
	$$
	and we used the assumption $\mcs'\le 0$ and $m:=\lfloor \mcs+\mcs'\rfloor <\mcs+\mcs'$ by $\mcs+\mcs'\not \in \N_0$.
	Now we deduce from the above inequality and \eqref{con2micro} that
	\begin{equation}\label{IIx}
		\textrm{II}_x
		\le C_2
		\|y-x\|^{-(\mcs+\mcs'-m)}( \|x-x_0\|+\|y-x\|)^{-\mcs'}.
	\end{equation}
	For the term  $\textrm{II}_y$,
	\eqref{IIx} obviously holds by exchanging $x$ with $y$.
	Noting $\|y-x_0\|\le \|y-x\|+\|x-x_0\|$ and using \eqref{IIx}  by exchanging $x$ with $y$, we conclude that
	$$
	\textrm{II}_y
	\le C_2 2^{-\mcs'} \| y-x\|^{-(\mcs+\mcs'-m)}( \|x-x_0\|+\|y-x\|)^{-\mcs'}.
	$$
	To estimate the term \textrm{I}, using the mean value theorem, we conclude that
	$$
	|\psi_{j,k}(x)-\psi_{j,k}(y)|\leq \|y-x\| \sup_{c\in [0,1]}\|\nabla(\psi_{j,k})(cx+(1-c)y)\|.
	$$
	Since $\sup_{z\in \R^d} \|\nabla(\psi_{j,k})(z)\| \le 2^{(d/2+1)j} \|\nabla \psi\|_{L^\infty}<\infty$, we conclude that
	$$
	|\psi_{j,k}(x)-\psi_{j,k}(y)|\le \|y-x\| 2^{(d/2+1)j}  \|\nabla \psi\|_{L^\infty}(\chi_{\supp(\psi_{j,k})}(x)
	+\chi_{\supp(\psi_{j,k})}(y)).
	$$
	Using \eqref{wavelet2microlocal} and \eqref{con2micro}, we conclude from the above inequality that
	$I\le I_x+I_y$ with
	{\small
		$$
		I_x:=C_1 \|\nabla \psi\|_{L^\infty} \|y-x\|
		\sum_{j<J}\sum_{k\in \Z^d}
		2^{-(d/2+\mcs-m)j} (1+\|k-2^jx\|)^{-\mcs'} 2^{(d/2+1)j} \chi_{\supp(\psi_{j,k})}(x)
		$$}
	and $I_y$ is defined by exchanging $y$ with $x$. Using \eqref{supp:psi} and the same argument for estimating the term $\textrm{II}_x$, we see that
	\begin{align*}
		I_x &\le C_1 M \|\nabla \psi\|_{L^\infty} \|y-x\| \sum_{j<J} 2^{-(d/2+\mcs-m)j}(K+1+2^j \|x-x_0\|)^{-\mcs'} 2^{(d/2+1)j}\\
		&\le C_1 M \|\nabla \psi\|_{L^\infty} (K+1)^{-\mcs'}
		\|y-x\| \sum_{j<J} 2 ^{(m+1-\mcs)j}=C_3 \|y-x\| 2^{(m+1-\mcs)J},
	\end{align*}
	where we used $m+1-\mcs>0$ by $m=\lfloor \mcs\rfloor>\mcs-1$, and
	$$
	C_3:=C_1 M \|\nabla \psi\|_{L^\infty} (K+1)^{-\mcs'} (2^{m+1-\mcs}-1)^{-1}<\infty.
	$$
	Using \eqref{con2micro} and $m+1-\mcs>0$ and noting that $\mcs'\le 0$,
	we conclude from $2^J<\|y-x\|^{-1}$ that
	$$
	I_x\le C_3\|y-x\|^{\mcs-m}
	\le C_3 \|y-x\|^{\mcs+\mcs'-m}(\|y-x\|+\|x-x_0\|)^{-\mcs'}.
	$$
	For the term $I_y$, the above inequality obviously holds by exchanging $y$ with $x$.
	Noting $\|y-x_0\|\le \|y-x\|+\|x-x_0\|$ and using the above inequality by exchanging $y$ with $x$, we conclude that
	$$
	\textrm{I}_y
	\le C_3 2^{-\mcs'} \| y-x\|^{-(\mcs+\mcs'-m)}( \|x-x_0\|+\|y-x\|)^{-\mcs'}.
	$$
	Putting all estimates together, we conclude that \eqref{kspaceform0} holds and hence,
	$f\in K^{\mcs,\mcs'}_{x_0}(\R^d)$.
	Using a smooth wavelet $\psi$
	with high vanishing moments,
	the claim for $m<\lfloor \mcs\rfloor$ can be proved by following the same argument as given in \cite[Section~6]{seuret2003time1}.
	
	To prove item (ii), by item (i), we have $K^{\mcs,\mcs'}_{x_0}(\R^d)\subseteq C^{\mcs,\mcs'}_{x_0}(\R^d)\subseteq C^\mcs_{x_0}(\R^d)$. If $\mcs'=-\mcs$,
	then $C^{\mcs}_{x_0}(\R^d) \subseteq K_{x_0}^{\mcs,-\mcs}(\R^d)$  follows from \eqref{kspaceform0} and \eqref{pointholder}.
\end{proof}

\begin{proof}[Proof of {Theorem~\ref{thm:2micro}}] It is enough to show the result for the case $m:=\lfloor \mcs+\mcs'\rfloor=\lfloor \mcs\rfloor$ while for the case $m<\lfloor \mcs\rfloor$, the estimate is similar by using a smooth shearlet with at least $\lfloor \mcs\rfloor-m+1$ vanishing moments, as in the proof of Lemma~\ref{lem:timecharcz2microlocal}.
	
	(i)
	Let $f\in C_{x_0}^{\mcs,\mcs'}(\R^2)$.
	Define $\wh{\tilde{\psi}_1}(\xi):=\frac{\wh{\psi_1}(\xi)}{(-i2\pi \xi)^m}$. Since $\wh{\psi_1}\in \mathcal{D}(\R)$ and $\supp(\wh{\psi_1})\subseteq [-2,-\frac{1}{2}]\cup [\frac{1}{2},2]$, it is straightforward to conclude that $\wh{\tilde{\psi}_1}$ is well defined and
	$\wh{\tilde{\psi}_1}\in \mathcal{D}(\R)$
	with $\supp(\wh{\tilde{\psi}_1})\subseteq [-2,-\frac{1}{2}]\cup [\frac{1}{2},2]$.
	Define $\wh{\tilde{\psi}}(\xi_1,\xi_2):=\wh{\tilde{\psi}_1}(\xi_1)\wh{\psi_2}(\xi_2/\xi_1)$ as in \eqref{psidef}.
	Choose a particular choice $\alpha:=(m,0)\in \N_0^2$ with $|\alpha|=m$.
	By the Fourier transform of $\psi_{ast}$ in \eqref{hatpsiast}, for $\xi=(\xi_1,\xi_2)$, we have
	$$
	a^m \wh{\tilde{\psi}_{ast}}(\xi)=
	a^{3/4} e^{-i2\pi t\cdot \xi}
	(-i2\pi \xi_1)^{-m}
	\wh{\psi_1}(a\xi_1)
	 \wh{\psi_2}\left(\tfrac{1}{\sqrt{a}}\left(\tfrac{\xi_2}{\xi_1}-s\right)\right)
	=(-i2\pi \xi)^{-\alpha}\wh{\psi_{ast}}(\xi).
	$$
	Define $f_\alpha:=\partial^\alpha f$.
	Since $\wh{f_\alpha}(\xi)=(i2\pi \xi)^\alpha \wh{f}(\xi)$, for $\alpha=(m,0)$, we deduce from the above identity that
	$$
	\la \wh{f}, \wh{\psi_{ast}}\ra=
	\la f, \psi_{ast}\ra=
	\la (i2\pi \xi)^\alpha \wh{f}(\xi), (-i2\pi \xi)^{-\alpha} \wh{\psi_{ast}}(\xi)\ra=a^m \la \wh{f_\alpha}, \wh{\tilde{\psi}_{ast}}\ra.
	$$
	Now we deduce that
	$$
	\mathcal{SH}_\psi f(a,s,t)=\la \wh{f}, \wh{\psi_{ast}}\ra=
	a^m \la \wh{f_\alpha}, \wh{\tilde{\psi}_{ast}}\ra=
	a^m \mathcal{SH}_{\tilde{\psi}} f_\alpha (a,s,t).
	$$
	Similarly, taking $\beta=(0,m)$, we have
	$\mathcal{SH}^{(v)}_\psi f(a,s,t)=
	a^m \mathcal{SH}^{(v)}_{\tilde{\psi}} f_\beta(a,s,t)$.
	To prove \eqref{necessary2micro},
	it suffices to prove that
	{\small	
		\begin{equation} \label{SHfalpha} \max\left(|\mathcal{SH}_{\tilde{\psi}}f_\alpha (a,s,t)|, |\mathcal{SH}^{(v)}_{\tilde{\psi}}f_\alpha(a,s,t)|\right)\leq C a^{\frac{3}{4}+\frac{\mcs-m}{2}}\left(1+\left\|\tfrac{t-x_0}{\sqrt{a}}\right\|^{-\mcs'}\right),
		\end{equation}
	}
	for all $a\in(0,1)$, $s\in[-2,2]$, $t\in \mathbb R^2$ and for all $\alpha\in \N_0^2$ with $|\alpha|=m$.
	
	By $\supp(\wh{\tilde{\psi}_1})\subseteq [-2,-\frac{1}{2}]\cup [\frac{1}{2},2]$,
	we note that $\wh{\tilde{\psi}}(0)=0$ and hence $\int_{\R^2} \tilde{\psi}(x) dx=0$. Then
	{\small	
		$$
		|\mathcal{SH}_{\tilde{\psi}} f_\alpha(a,s,t)|
		=|\la f_\alpha, \tilde{\psi}_{ast}\ra|
		=|\la f_\alpha(x)-f_\alpha(t), \tilde{\psi}_{ast}\ra|
		\le \int_{\R^2}|f_\alpha(x)-f_\alpha(t)|\, |\tilde{\psi}_{ast}(x)|dx,
		$$
	}
	Since $f\in C^{\mcs,\mcs'}_{x_0}(\R^2)$ and $\mcs+\mcs'\not\in \N_0$,
	by item (i) of Lemma~\ref{lem:timecharcz2microlocal}, $f\in K^{\mcs,\mcs'}_{x_0}(\R^2)$ and \eqref{kspaceform0} holds with $|\alpha|=m$ due to our assumption. Since $\tilde{\psi}_{ast}=a^{-3/4}\tilde{\psi}(M_{as}^{-1}(\cdot-t))$, using substitution $z=M_{as}^{-1}(x-t)$, we have $x=t+M_{as} z$. By \eqref{kspaceform0} with $|\alpha|=m$, we have
	\begin{align*} |\mathcal{SH}_{\tilde{\psi}}f_\alpha(a,s,t)|
		\leq C\ a^{\frac{3}{4}} \int_{\R^2}
		\|M_{as} z\|^{\mcs+\mcs'-m}
		\left(\|M_{as} z\|+\|t-x_0+M_{as} z\|\right)^{-\mcs'}
		|\tilde{\psi}(z)|dz.
	\end{align*}
	Using the operator norms in \cite{kutyniok2009resolution}, for $a>0$ and $s\in [-2,2]$, we have
	\begin{equation}\label{Masnorm}
		\|M_{as} z\|\le a^{1/2} \|S_s\| \|z\|
		\le a^{1/2} \left(1+\tfrac{s^2}{2}+\left( s^2+\tfrac{s^4}{4}\right)^{1/2}\right)^{1/2}  \|z\|\le C_0 a^{1/2} \|z\|,
	\end{equation}
	where $C_0:=\sup_{s\in [-2,2]} (1+\frac{s^2}{2}+( s^2+\frac{s^4}{4})^{1/2})^{1/2}=\sqrt{3+\sqrt{8}}$.
	Then we deduce that
	\begin{small}
		\begin{align*}		 &|\mathcal{SH}_{\psi}f_\alpha(a,s,t)|\\
			&\qquad \leq C C_0^{\mcs+\mcs'-m}\ a^{\frac{3}{4}+\frac{\mcs+\mcs'-m}{2}}\ \int_{\R^2} \|z\|^{\mcs+\mcs'-m}
			(2C_0a^{\frac{1}{2}}\|z\|
			+\|t-x_0\|)^{-\mcs'} |\tilde{\psi}(z)|dz\\
			&\qquad \leq C_1\ a^{\frac{3}{4}+\frac{\mcs-m}{2}} \left(\int_{\R^2}\|z\|^{\mcs-m} |\tilde{\psi}(z)|dz+\left\|\frac{t-x_0}
			{\sqrt{a}}\right\|^{-\mcs'}	 \int_{\R^2}\|z\|^{\mcs+\mcs'-m}|\tilde{\psi}(z)|dz\right),
		\end{align*}
	\end{small}
	where $C_1:=C C_0^{\mcs+\mcs'-m} 2^{-\mcs'}\max(1,2^{-\mcs'} C_0^{-\mcs'})$.
	Since $\tilde{\psi}$ is a Schwartz function, we have
$$
\int_{\R^2} \|z\|^{\mcs-m} |\tilde{\psi}(z)| dz<\infty\quad \mbox{and}\quad	 \int_{\R^2}\|z\|^{\mcs+\mcs'-m}|\tilde{\psi}(z)|dz<\infty.
$$
	This proves the first part of \eqref{SHfalpha}.
	The estimation for the vertical shearlet transform $\mathcal{SH}_{\tilde{\psi}}^{(v)} f_\alpha(a,s,t)$ is similar. This proves \eqref{SHfalpha}. Consequently, \eqref{necessary2micro} must hold.
	
	\noindent (ii) Recall
	that we can decompose $f=P_{\lfD} f+P_{\mathcal{C}} f +P_{\mathcal{C}^{v}}f$ in \eqref{decomp}. Since $P_{\lfD}f$ is the inverse Fourier transform of the compactly supported distribution $\wh{f}\chi_{\lfD}$, $P_{\lfD} f$ must be analytic and it is easy to see that $P_{\lfD} f\in C^{\mcs,\mcs'}_{x_0}(\R^2)$.
	Therefore, we only need to prove the result for $P_{\mathcal{C}}f$ while the analysis for  $P_{\mathcal{C}^{v}}f$ is similar.
	Note that
	$$
	\mathcal{SH}_\psi (P_{\mathcal{C}} f)(a,s,t)=\langle P_{\mathcal{C}} f, \psi_{ast}\rangle
	=\langle \wh{P_{\mathcal{C}} f
	},\wh{\psi_{ast}}\rangle=
	\langle \wh{f} \chi_{\mathcal{C}}, \wh{\psi_{ast}}\rangle
	$$
	and $\|\wh{\psi_{ast}}\chi_{\mathcal{C}}\|_{L^2}\ne 0$ only if $a\in (0,1]$ and $s\in [-2,2]$.
	Thus, $\mathcal{SH}_\psi (P_{\mathcal{C}} f)(a,s,t)=\mathcal{SH}_\psi f(a,s,t)$ for all $a\in (0,1]$ and $s\in [-2,2]$.
	Using the reproducing formula in \eqref{reprodformula}
	with $f$ being replaced by $P_{\mathcal{C}} f$, we have
	$$
	P_{\mathcal{C}} f(x)=
	\int_{a=0}^{a=1}\int_{s=-2}^{s=2}
	\int_{t\in \R^2} \mathcal{SH}_\psi f(a,s,t)\psi_{ast}(x) a^{-3} dtdsda.
	$$
	Define the gradient operator $\nabla:=(\partial_1,\partial_2)$ and let $\otimes^m \nabla$ be its $m$th (right) Kronecker product. Using the identity in \cite[Proposition~2.1]{han03} or \cite[Lemma~7.2.1]{hanbook} and defining $f_m:=(\otimes^m \nabla) (P_{\mathcal{C}} f)$, we deduce from the above identity that
	$$
	 f_m(x)=\int_{a=0}^{a=1}\int_{s=-2}^{s=2}
	\int_{t\in \R^2} \mathcal{SH}_\psi f(a,s,t) \tilde{\psi}_{ast}(x) (\otimes^m M_{as}^{-1})
	a^{-3} dtdsda,
	$$
	where $\tilde{\psi}:=(\otimes^m \nabla) \psi$. Note that $f_m$ is a row vector and each entry of $f_m$ is $\partial^\alpha (P_{\mathcal{C}} f)$ for some $\alpha\in \N_0^2$ with $|\alpha|=m$.
	Hence, we can decompose $f_m$
	into a sum of Littlewood-Paley as follows:
	$$
	 f_m(x)=\sum_{j=-\infty}^{-1}\Delta_j(x)
	$$
	with
	$$
	 \Delta_j(x)=\int_{2^j}^{2^{j+1}}\int_{-2}^{2}\int_{\mathbb R^2}\mathcal{SH}_{\psi}f (a,s,t)\tilde{\psi} _{ast}(x)(\otimes^m M_{as}^{-1})
	a^{-3} dt ds da.
	$$
	By the definition of $M_{as}$ in \eqref{mas},
	there exists a positive constant $C_2$ such that the operator norm $\|\otimes^m M_{as}^{-1}\|\le C_2 a^{-m}$ for all $s\in [-2,2]$ and $a\in (0,1)$.
	By our assumption in \eqref{sufficient2micro1} and the definition of $\tilde{\psi}_{ast}$ in \eqref{psiast}, we have
	{\small	\begin{align*}
			&|\Delta_j(x)|\\
			&\le CC_2 \int_{2^j}^{2^{j+1}}
			\int_{-2}^2\int_{\R^2}
			 a^{\frac{5}{4}+\mcs-\frac{3}{4}}
			\left(1+\left\| \frac{t-x_0}{\sqrt{a}}\right\|^{-2\mcs'}\right)
			 |\tilde{\psi}(M_{as}^{-1}(x-t))| a^{-3-m} dt ds da\\
			&=CC_2 \int_{2^j}^{2^{j+1}}
			\int_{-2}^2\int_{\R^2}
			a^{\mcs-m-1}
			\left(1+\left\| \frac{x-x_0-M_{as}z}{\sqrt{a}}\right\|^{-2\mcs'}\right)
			|\tilde{\psi}(z)| dz ds da\\
			&\le CC_2  \int_{2^j}^{2^{j+1}}
			\int_{-2}^2\int_{\R^2}
			a^{\mcs-m-1}
			\left(1+a^{\mcs'} 2^{-2\mcs'} (\|M_{as}z\|^{-2\mcs'}+\|x-x_0\|^{-2\mcs'})
			\right) |\tilde{\psi}(z)| dz ds da,
		\end{align*}
	}
	where we used $\mcs'\le 0$ in the last inequality. Note that
	\eqref{Masnorm} holds for all $a>0$ and $s\in [-2,2]$.
	Observing that $\tilde{\psi}$ is a Schwartz function and $2^{-j q}\int_{2^j}^{2^{j+1}}
	a^{q-1} da=(2^q-1)/q$ for all $q\in \R$ (with $\lim_{q\to 0} \frac{2^q-1}{q}=\ln 2$),
	we deduce from the above inequalities that
	 \begin{equation}\label{littlewoodpaley}
		|\Delta_j(x)| \leq \tilde{C} 2^{j (\mcs-m)}+\tilde{C} \ 2^{j (\mcs+\mcs'-m)}\left\|x-x_0\right\|^{-2\mcs'},
	\end{equation}
	where $\tilde{C}:=CC_2 \max(C_3,C_4)<\infty$, and by $\mcs+\mcs'\ge m$, $\mcs'\le 0$ and $\mcs\ge \lfloor \mcs\rfloor=m$,
	\begin{align*}
		C_3&:=
		(4\|\psi\|_{L^1}
		+C_0^{-2\mcs'} 2^{2-2\mcs'} \| |z|^{-2\mcs'} \tilde{\psi}(z)\|_{L^1})
		\tfrac{2^{\mcs-m}-1}{\mcs-m},\\
		C_4&:=
		2^{2-2\mcs'} \|\tilde{\psi}\|_{L^1} \tfrac{2^{\mcs+\mcs'-m}-1}{\mcs+\mcs'-m}.
	\end{align*}
	The above estimate in \eqref{littlewoodpaley} guarantees fast decay of $|\Delta_j(x)|$ as $j\to -\infty$.
	
	In order to prove $f\in K_{x_0}^{\mcs,\mcs'}(\R^2)$,
	we now estimate $|\Delta_j(x)-\Delta_j(y)|$.
	By our assumption in \eqref{sufficient2micro1} and the definition of $\tilde{\psi}_{ast}$ in \eqref{psiast}, we have
	{\small
		\begin{align*}
			&|\Delta_j(x)-\Delta_j(y)|\\
			&\le C \int_{2^j}^{2^{j+1}}
			\int_{-2}^2\int_{\R^2}		 a^{\mcs-m-\frac{5}{2}}
			 \left(1+\left\|\tfrac{t-x_0}{\sqrt{a}}\right\|^{-2\mcs'}
			\right) |\tilde{\psi}(M_{as}^{-1}(x-t))-\tilde{\psi}(M_{as}^{-1}(y-t))| dt ds da\\
			&=C \int_{2^j}^{2^{j+1}}
			\int_{-2}^2\int_{\R^2}
			a^{\mcs-m-1}
			\left(1+\left\| \tfrac{x-x_0-M_{as}z}{\sqrt{a}}\right\|^{-2\mcs'}\right)	 |\tilde{\psi}(z)-\tilde{\psi}(z+M_{as}^{-1}(y-x)| dz ds da.
		\end{align*}
	}
	Note that
	\begin{align*}
		 |\tilde{\psi}(z)-\tilde{\psi}(z+M_{as}^{-1}(y-x))|
		&=\left|\int_{p=0}^{p=1} \nabla \tilde{\psi}(z+M_{as}^{-1}(y-x) p)\cdot M_{as}^{-1}(y-x) dp\right|\\
		&\le \|M_{as}^{-1}(y-x)\| \int_{p=0}^{p=1} \|\nabla \tilde{\psi}(z+M_{as}^{-1}(y-x)p)\| dp.
	\end{align*}
	Hence, we deduce that
	{\small
		\begin{align*}
			&|\Delta_j(x)-\Delta_j(y)|\\
			&\le C \int_{2^j}^{2^{j+1}}
			\int_{-2}^2
			\int_{p=0}^{p=1}
			\int_{\R^2}
			a^{\mcs-m-1}
			\left(1+\left\| \tfrac{x-x_0-M_{as}z}{\sqrt{a}}\right\|^{-2\mcs'}\right)
			\|M_{as}^{-1}(y-x)\|
			\\ &\hspace{2 in}\times
			\|\nabla \tilde{\psi}(z+M_{as}^{-1}(y-x)p)\| dz dp ds da\\
			&\le C \int_{2^j}^{2^{j+1}}
			\int_{-2}^2
			\int_{p=0}^{p=1}
			\int_{\R^2}
			a^{\mcs-m-1}
			\left(1+\left\| \tfrac{x-x_0+(y-x)p-M_{as}z}{\sqrt{a}}\right\|^{-2\mcs'}\right)
			\|M_{as}^{-1}(y-x)\|
			\\
			&\hspace{2 in}\times
			\|\nabla \tilde{\psi}(z)\| dz dp ds da\\
			&\le C \int_{2^j}^{2^{j+1}}
			\int_{-2}^2
			\int_{\R^2}
			a^{\mcs-m-1}
			\left(1+a^{\mcs'} 3^{-2\mcs'} (
			\|M_{as} z\|^{-2\mcs'}+\|x-x_0\|^{-2\mcs'}+
			\|y-x\|^{-2\mcs'})\right)
			\\ &\hspace{2 in}\times
			\|M_{as}^{-1}(y-x)\| \|\nabla \tilde{\psi}(z)\| dz ds da.
		\end{align*}
	}
	Employing the operator norm $\|M_{as}^{-1}\|$ and similar to \eqref{Masnorm}, we have
	$$
	\|M_{as}^{-1}(y-x)\|\leq \|M_{as}^{-1}\| \|y-x\|
	\leq C_0a^{-1}\|y-x\|.
	$$
	Hence, we further deduce that
	{\small
		\begin{align*}
			&|\Delta_j(x)-\Delta_j(y)|\\
			&\le CC_0 \|y-x\|\left\lbrace\int_{2^j}^{2^{j+1}}
			a^{\mcs-m-2}
			\int_{-2}^2 \int_{\R^2}\left(1+(3C_0)^{-2\mcs'}\|z\|^{-2\mcs'}\right)
			\|\nabla \tilde{\psi}(z)\|dz ds da \right.\\
			&\left. \qquad
			 +3^{-2\mcs'}\left(\left\|x-x_0\right
			 \|^{-2\mcs'}+\left\|y-x_0\right\|^{-2\mcs'}
			\right)\int_{2^j}^{2^{j+1}}
			a^{\mcs+\mcs'-m-2}
			\int_{-2}^2
			\int_{\R^2}
			\|\nabla \tilde{\psi}(z)\|dz ds da \right\rbrace.
		\end{align*}
	}
	Thus we conclude that for all $j\le -1$ and $x,y\in \R^2$,
	{\small
		 \begin{equation}\label{difflittlewoodpaley}
			|\Delta_j(x)-\Delta_j(y)| \leq \mathring{C} \left\|y-x\right\| \left\lbrace2^{j (\mcs-m-1)}+2^{j (\mcs+\mcs'-m-1)}
			 (\|x-x_0\|^{-2\mcs'}+\|y-x_0\|^{-2\mcs'})
			\right\rbrace,
		\end{equation}
	}
	where $\mathring{C}:=CC_0\max(C_5,C_6)<\infty$ with
	\begin{align*}
		C_5 &:=
		2^2\left(\|\nabla \tilde{\psi}\|_{L^1}+(3C_0)^{-2\mcs'}\||z|^{-2\mcs'}\nabla \tilde{\psi}(z)\|_{L^1}\right)\frac{2^{\mcs-m-1}-1}{\mcs-m-1},\\
		C_6 &:=2^23^{-2\mcs'}\|\nabla \tilde{\psi}\|_{L^1}\frac{2^{\mcs+\mcs'-m-1}-1}{\mcs+\mcs'-m-1}.
	\end{align*}
	Since $K_{x_0}^{\mcs,\mcs'}(\R^2)= C_{x_0}^{\mcs,\mcs'}(\R^2)$ by Lemma~\ref{lem:timecharcz2microlocal},
	we want to show $f\in K_{x_0}^{\mcs,\mcs'}(\R^2)$.
	We consider $x\ne y\in \R^2$ with $\|x-x_0\|<1/4$ and $\|y-x_0\|<1/4$.
	Then there exists a unique positive integer $J$ such that \eqref{con2micro} holds. Hence, by $f_m(x)=\sum_{j=-\infty}^{-1}\Delta_j(x)$, we can write
	\begin{align*}
		|f_m(x)-f_m(y)|\leq
		 \underbrace{\sum_{j=1}^{J-1}\left|\Delta_{-j}(x)
			 -\Delta_{-j}(y)\right|}_{\textrm{I}_0} +\underbrace{\sum_{j=J}^{\infty}|\Delta_{-j}(x)|}_{\textrm{I}_x}		 +\underbrace{\sum_{j=J}^{\infty}|\Delta_{-j}(y)|}_{\textrm{I}_y}.
	\end{align*}
	We first estimate $\textrm{I}_0$.
	By $m:=\lfloor \mcs+\mcs'\rfloor=\lfloor \mcs\rfloor$ and $\mcs'<0$, we see that $m+1-\mcs-\mcs'>m+1-\mcs=\lfloor \mcs\rfloor+1-\mcs>0$. Consequently, by
	\eqref{difflittlewoodpaley}, we deduce that
	\begin{align*}
		\textrm{I}_0&\le \mathring{C} \|y-x\|
		\left\{
		\sum_{j=1}^{J-1} 2^{j (m+1-\mcs)}+
		 (\|x-x_0\|^{-2\mcs'}+\|y-x_0\|^{-2\mcs'})
		\sum_{j=1}^{J-1} 2^{j(m+1-\mcs-\mcs')} \right\}\\
		&\le C_7
		\|y-x\| 2^{J(m+1-\mcs)} \left\{ 1+
		2^{-J\mcs'}
		 (\|x-x_0\|^{-2\mcs'}+\|y-x_0\|^{-2\mcs'})
		\right\},
	\end{align*}
	where $C_7:=\mathring{C} \max((2^{m+1-\mcs}-1)^{-1}, (2^{m+1-\mcs-\mcs'}-1)^{-1})<\infty$.
	Using \eqref{con2micro} and $\mcs'<0$, we further deduce from the above inequality that
	$2^J<\|y-x\|^{-1}$ and
	$$
	\textrm{I}_0\le C_7  \|y-x\|^{\mcs+\mcs'-m}\left(
	\|y-x\|^{1-\mcs'}+
	 \|x-x_0\|^{-2\mcs'}+\|y-x_0\|^{-2\mcs'}\right).
	$$
	Since $\max(\|x-x_0\|,\|y-x_0\|)<1/4$ and $\mcs'<0$, we conclude that
	$$
	\|y-x\|^{1-\mcs'}+
	 \|x-x_0\|^{-2\mcs'}+\|y-x_0\|^{-2\mcs'}
	\le C_8 (\|y-x\|+\|x-x_0\|)^{-\mcs'},
	$$
	where $C_8=1+2^{-2\mcs'}$,
	and hence
	$$
	\textrm{I}_0\le C_7 C_8  \|y-x\|^{\mcs+\mcs'-m}(
	\|y-x\|+
	\|x-x_0\|)^{-\mcs'}.
	$$
	Notice that $m:=\lfloor \mcs+\mcs'\rfloor\le \mcs+\mcs'<\mcs$ by $\mcs'<0$. Hence, $m-\mcs<0$.
	Using \eqref{littlewoodpaley},
	we now estimate $\textrm{I}_x$:
	\begin{align*}
		\textrm{I}_x
		&\le
		\tilde{C} \sum_{j=J}^\infty 2^{j(m-\mcs)}+\tilde{C} \|x-x_0\|^{-2\mcs'}
		\sum_{j=J}^\infty 2^{m-\mcs-\mcs'}\\
		&\le C_9 \left\{2^{J(m-\mcs)}+
		2^{J(m-\mcs-\mcs')} \|x-x_0\|^{-2\mcs'}\right\},
	\end{align*}
	where $C_9:=\tilde{C}\max((1-2^{m-\mcs})^{-1}, (1-2^{m-\mcs-\mcs'})^{-1})<\infty$.
	The above inequality for $I_y$ must hold by exchanging $y$ with $x$. Hence,
	$$
	\textrm{I}_x+\textrm{I}_y\le 2 C_9 \left\{2^{J(m-\mcs)}+
	2^{J(m-\mcs-\mcs')} (\|x-x_0\|^{-2\mcs'}+\|y-x_0\|^{-2\mcs'})
	\right\}.
	$$
	Using \eqref{con2micro}, we
	have $2^{-J}\le \|y-x\|$ and
	we conclude from the above inequality that
	\begin{align*}
		\textrm{I}_x+\textrm{I}_y&\le
		2 C_9 \left\{\|y-x\|^{\mcs-m}+
		\|y-x\|^{\mcs+\mcs'-m} (\|x-x_0\|^{-2\mcs'}+\|y-x_0\|^{-2\mcs'})
		\right\}\\
		&=2C_9 \|y-x\|^{\mcs+\mcs'-m}
		 \left(\|y-x\|^{-\mcs'}+\|x-x_0\|^{-2\mcs'}+\|y-x_0\|^{-2\mcs'}\right)\\
		&\le C_{10} \|y-x\|^{\mcs+\mcs'-m}
		(\|y-x\|+\|x-x_0\|)^{-\mcs'},
	\end{align*}
	where $C_{10}=2C_9(1+2^{-2\mcs'})$.
	Putting all together, this proves \eqref{kspaceform0}. Hence, $f\in K^{\mcs,\mcs'}_{x_0}(\R^2)$. By Lemma~\ref{lem:timecharcz2microlocal}, we conclude that $f\in C^{\mcs,\mcs'}_{x_0}(\R^2)$.
	Since $\mcs+\mcs'>0$ and $\mcs'<0$, we have $\mcs>0$ and consequently, $C_{x_0}^{\mcs,\mcs'}(\R^2)\subseteq C_{x_0}^\mcs(\R^2)$. So, $f\in C^{\mcs}_{x_0}(\R^2)$.
\end{proof}

We now prove Theorem~\ref{thm:2microsobolevwavefront}.

\begin{proof}[Proof of Theorem~\ref{thm:2microsobolevwavefront}]
	
	Since $f\in C^{\mcs,\mcs'}_{x_0}(\R^2)$, by Theorem~\ref{thm:2micro}, we have
	\begin{align*}
		I:= &\int_{s=s_0-r_0}^{s=s_0+r_0} \int_{t \in  B_{r_0}(x_0)}\int_{a=0}^{a=1}\left|\mathcal{SH}_{\psi}f(a,s,t)\right|^2a^{-2m-3} da dt\ ds \\
		\leq & C^2 \int_{s=s_0-r_0}^{s=s_0+r_0}\int_{t\in B_{r_0}(x_0)}\int_{a=0}^{a=1} a^{\mcs+\lfloor \mcs\rfloor+\frac{3}{2}}\left(1+\left\|\frac{t-x_0}{\sqrt{a}}\right\|^{-\mcs'}\right)^2 a^{-2m-3}dads dt \\
		\leq &
		2C^2 \int_{s=s_0-r_0}^{s=s_0+r_0}\int_{t\in B_{r_0}(x_0)}\int_{a=0}^{a=1} a^{\mcs+\lfloor \mcs\rfloor-2m-\frac{3}{2}} \left(1+\left\|\frac{t-x_0}{\sqrt{a}}\right\|^{-2\mcs'}\right) dads dt\\
		=&4\pi C^2 r_0^3 \int_0^1 a^{\mcs+\lfloor \mcs\rfloor-2m-\frac{3}{2}} da+
		\frac{4C^2 r_0^{3-2\mcs'}}{2(1-\mcs')}
		\int_0^1 a^{\mcs+\mcs'+\lfloor \mcs\rfloor-2m-\frac{3}{2}}da.
	\end{align*}
	Due to $\mcs+\mcs'+\lfloor \mcs\rfloor-2m-\frac{3}{2}>-1$ and $\mcs'<0$, the above last two integrals must be finite. Hence, $I<\infty$. By Theorem~\ref{thm:swf:integral}, we conclude that $(x_0,\xi_0)\notin \wf_m(f)$ for all $\xi_0\in \R^2\bs\{0\}$.
\end{proof}

We finish this paper by demonstrating anisotropic singularity detection of H\"older continuous functions by the continuous shearlet transform.
For $\mcs>0$, we define a function below:
\begin{equation}\label{holderexample}
	\mathcal{B}_{\mcs}(x)=\begin{cases}
		(1-|x|^2)^{\mcs}, &\text{if $|x|\le 1$,} \\
		0 &\text{if $|x|>1$}.
	\end{cases}
\end{equation}

\begin{figure}[tbhp]
	\centering
	 \subfigure[]{\label{fig:a}\includegraphics[width=5 cm,height=4 cm]{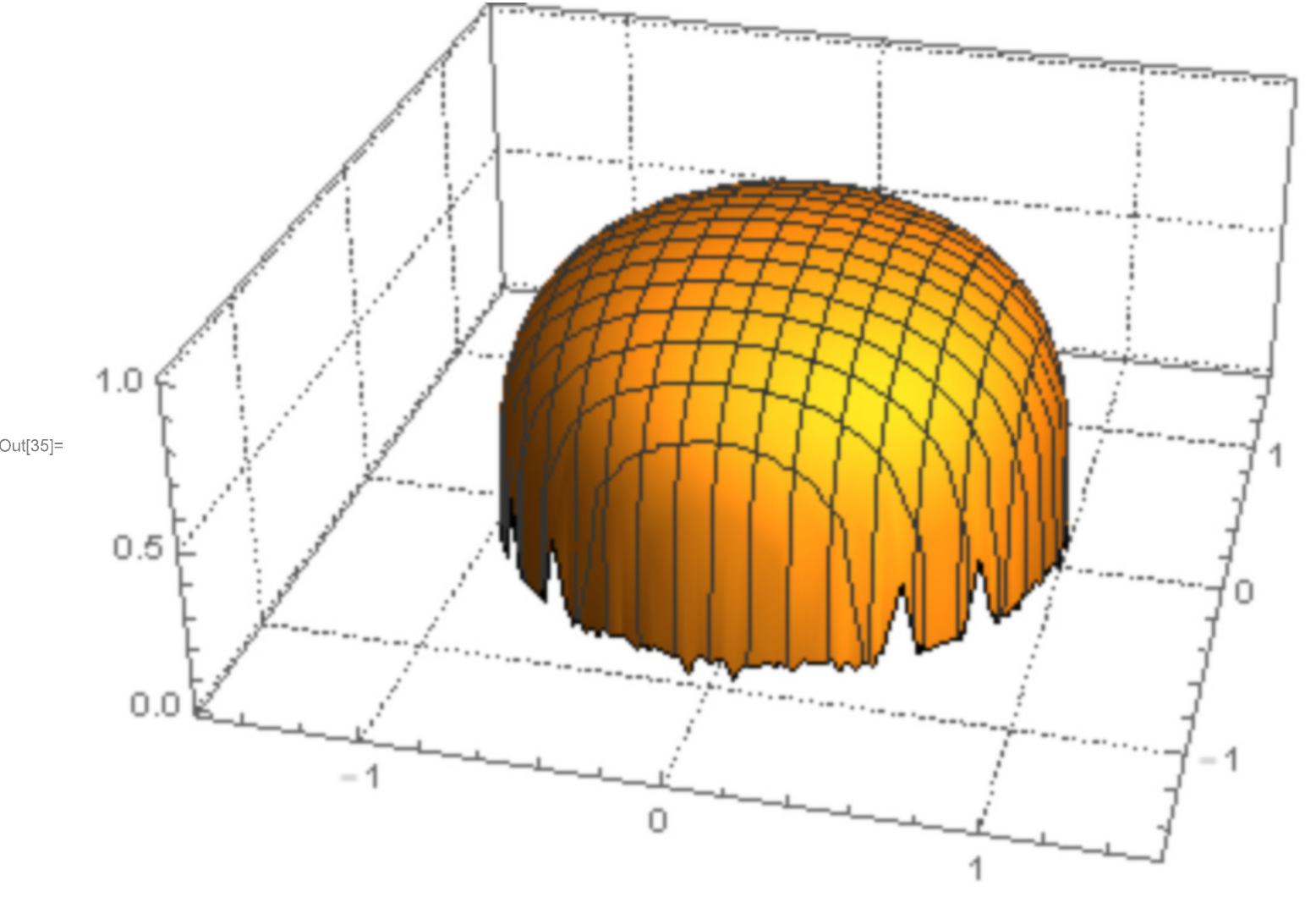}}
	 \subfigure[]{\label{fig:b}\includegraphics[width=5.3cm,height=4.3cm]{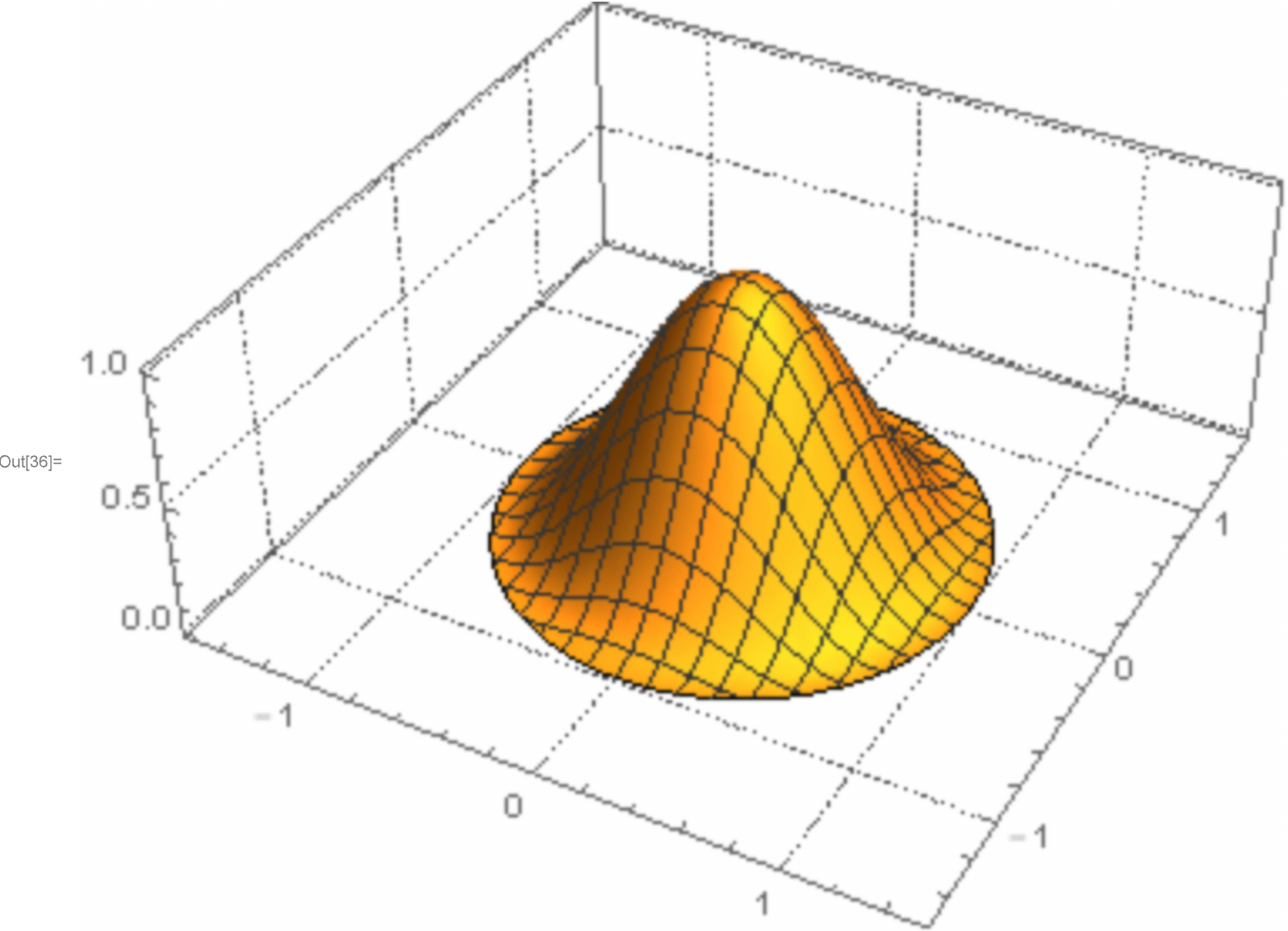}}
	\caption{Plots of $\mathcal{B}_{\mcs}(x)$ for (a) $\mcs=\frac{1}{4}$  and (b) $\mcs=\frac{5}{2}$.}
	\label{fig:testholder}
\end{figure}

Let $S^1:=\{ (x_1,x_2) \in \R^2 \setsp x_1^2+x_2^2=1\}$ be the unit circle in $\R^2$.
The function $\mathcal{B}_\mcs$ is known to be H\"older continuous with exponent $\mcs$ at every point in $S^1$. For illustration purpose, the plots of $\mathcal{B}_{\mcs}(x)$ with $\mcs=\frac{1}{4}$ and $\mcs=\frac{5}{2}$ are displayed in Figure~\ref{fig:testholder}.

We now determine the Sobolev wavefront set of the function $\mathcal{B}_\mcs$.
The following result generalizes
the result in \cite{kutyniok2009resolution} with $\mcs=0$.

\begin{prop} \label{holdersobolevapp}
	Let $\mathcal{B}_\mcs$ be the function defined in \eqref{holderexample} with $\mcs> 0$.
	Then
	\begin{itemize}
		\item[(i)] For $(x_1,x_2)\in S^1$ and $s\ne \frac{x_2}{x_1}$ with $x_1\neq 0$,  $\mathcal{SH}_{\psi}\mathcal{B}_{\mcs}(a,s,t)$ decays rapidly as $a\to 0^+$.
		
		\item[(ii)] For $(x_1,x_2)\in S^1$ and $s=\frac{x_2}{x_1}$ with $x_1\neq 0$,
\begin{equation}\label{est:upper} \mathcal{SH}_{\psi}\mathcal{B}_{\mcs}(a,s,t)= \mathcal{O}(a^{\mcs+\frac{3}{4}}),\qquad \mathcal{SH}^{(v)}_{\psi}\mathcal{B}_{\mcs}(a,s,t) = \mathcal{O}(a^{\mcs+\frac{3}{4}})\quad \mbox{as}\quad a\to 0^+,
\end{equation}
where the above big $\mathcal{O}$ notation is defined in \eqref{cst:decay}.

		\item[(iii)] For $(x_1,x_2)\not \in S^1$,
		 $\mathcal{SH}_{\psi}\mathcal{B}_{\mcs}(a,s,t)$ and $\mathcal{SH}^{(v)}_{\psi}\mathcal{B}_{\mcs}(a,s,t)$ decay rapidly as $a\to 0^+$.
		
		\item[(iv)]  For $(x_1, x_2)\in S^1$ and $m<\mcs-\frac{1}{4}$,
		$((x_1, x_2), \xi_0)\not\in \wf_m(\mathcal{B}_{\mcs})$ for all $\xi_0\in \R^2\bs\{0\}$.
	\end{itemize}
\end{prop}

\begin{proof}
	Note that $\mathcal{B}_\mcs$ is a bivariate radial function by $\mathcal{B}_\mcs(x)=(1-|x|^2)^\mcs\chi_{[0,1]}(|x|)$ for all $x\in \R^2$. Through the polar coordinates, it is well known that its Fourier transform is given by
	$$
	\widehat{\mathcal{B}_{\mcs}}(\xi)
	=2\pi \int_0^\infty J_0(|\xi|t)\mathcal{B}_\mcs(t) t dt
	 =\frac{\Gamma(1+\mcs)}{\pi^{\mcs}}|\xi|^{-\mcs-1}J_{\mcs+1}(2\pi |\xi|),
	$$
	where the Bessel function $J_m (\xi)$ of order $m>-\frac{1}{2}$ is defined to be
	$$
	J_m(2\pi r)=\frac{r^m\pi^{m-\frac{1}{2}}}{\Gamma(m+\frac{1}{2})}\int_{-1}^{1}(1-t^2)^{m-\frac{1}{2}}e^{2\pi i r t}dt.
	$$
	Note that  $\widehat{\mathcal{B}_{\mcs}}(\xi)=\mathcal{O}(|\xi|^{-\mcs-\frac{3}{2}})$ as $|\xi|\to \infty$, because
	$J_m(2\pi r)\sim r^{-\frac{1}{2}}\left(e^{2\pi i r}+e^{-2\pi i r}\right)$ as $r\to \infty$.
	
	\noindent (i)  Using the polar coordinates, we have
	\begin{small}
		\begin{align*}
			 \mathcal{SH}_{\psi}&\mathcal{B}_{\mcs}(a,s,t) =a^{\frac{3}{4}}\int_{\mathbb{R}^2}\wh{\psi_1}(a\xi_1)\wh{\psi_2}\left(a^{-\frac{1}{2}}
			 \left(\frac{\xi_2}{\xi_1}-s\right)\right)e^{2\pi i \xi\cdot t}\widehat{\mathcal{B}_{\mcs}}(\xi)d\xi_1d\xi_2\\
			 &=a^{-\frac{5}{4}}\int_{0}^{\infty}\int_{0}^{2\pi}
			\wh{\psi_1}(\rho \ \cos\ \theta)\wh{\psi_2}\left(a^{-\frac{1}{2}}\left(\tan \theta-s\right)\right)e^{2\pi i \rho (t_1 \cos\theta + t_2\sin\theta)}\widehat{\mathcal{B}_{\mcs}}\left(\frac{\rho}{a}\right)\rho d\theta d\rho \\
			 &=a^{-\frac{5}{4}}\int_{0}^{\infty}I(a,s,t,\rho)\widehat{\mathcal{B}_{\mcs}}
			\left(\frac{\rho}{a}\right)\ \rho \ d\rho,
		\end{align*}
	\end{small}
	where noting that $\supp(\wh{\psi_2})\subseteq [-1,1]$, we define
	\begin{small}
		\begin{eqnarray*}
			I(a,s,t,\rho):= \int_{|\tan\theta-s|<\sqrt{a}}\wh{\psi_1}(\rho \ \cos\ \theta)\wh{\psi_2}\left(a^{-\frac{1}{2}}\left(\tan \theta-s\right)\right)e^{2\pi i \rho (t_1 \cos\theta + t_2\sin\theta)}d\theta.
		\end{eqnarray*}
	\end{small}
	By \cite[Lemma~4.9 and Proposition~4.7]{kutyniok2009resolution},  we know that $I(a,s,t,\rho)$ is an oscillatory integral of the  first kind that decays rapidly as $a\to 0^+$ for each $\rho>0$ when $s \neq x_2/x_1$. That is, for each $N\in \mathbb{N}$, there is a positive constant $C_N$ such that $ \sup_{\rho\in (0,\infty)}\left|I(a,s,t,\rho)\right|\leq C_Na^{\frac{N}{2}}$.
	Hence the continuous shearlet transform $\mathcal{SH}_{\psi}\mathcal{B}_{\mcs}(a,s,t)$ decays rapidly as $a\to 0^+$ when $s \neq x_2/x_1$.

	\noindent (ii)
Due to the rotation symmetry of $S^1$,
	it suffices to consider $x_1=1, x_2=0$ and hence $s=0$.
	Using the change of variable formula, one gets
	\begin{small}
		$$
		 \mathcal{SH}_{\psi}\mathcal{B}_{\mcs}(a,0,(1,0))=
		 a^{-\frac{3}{4}}\int_{0}^{\infty}\eta_a(\rho)\widehat{\mathcal{B}_{\mcs}}
		 \left(\frac{\rho}{a}\right)e^{2\pi i\frac{\rho}{a}}\ \rho \ d\rho
		$$
	\end{small}
	where
	\begin{eqnarray*}
		 \eta_a(\rho)&=&\int_{-(1+a)^{-\frac{1}{2}}}^{(1+a)^{-\frac{1}{2}}}\wh{\psi_1}(\rho\sqrt{1-au^2})\wh{\psi_2}\left(\frac{u}{\sqrt{1-au^2}}\right)e^{2\pi i \frac{\rho}{a}(\sqrt{1-au^2}-1)}\frac{du}{\sqrt{1-au^2}}.
	\end{eqnarray*}
Note that  $\wh{\psi_1}$ is supported inside $[-2,-\frac{1}{2}]\cup[\frac{1}{2},2]$ and $\wh{\psi_2}$ is supported inside $[-1,1]$.
Now it is clear that the function $\eta_a(\rho)$ has compact support w.r.t. $\rho$.
	Hence,
$$
\lim_{a\to0^+}\eta_a(\rho)=\eta_0(\rho):=\int_{-1}^{1}\wh{\psi_1}(\rho)\wh{\psi_2}(u)e^{-\pi i \rho u^2}du=\wh{\psi_1}(\rho)
\int_{-1}^{1}\wh{\psi_2}(u)e^{-\pi i \rho u^2}du.
$$
The convergence of the integrand also holds for all derivatives w.r.t. $u$. Therefore, we can write
	\begin{eqnarray*}
		 \mathcal{SH}_{\psi}\mathcal{B}_{\mcs}(a,0,(1,0))&\sim & a^{-\frac{3}{4}}\left(\int_{0}^{\infty}\left(\frac{a}{\rho}\right)^{\mcs+\frac{3}{2}}
\eta_a(\rho)e^{4\pi i \frac{\rho}{a}}\ \rho\  d\rho+\int_0^\infty\left(\frac{a}{\rho}\right)^{\mcs+\frac{3}{2}}\eta_a(\rho)\ \rho\  d\rho\right)\\
		&=& a^{\mcs+\frac{3}{4}}\left(\wh{F_{a,\mcs}}\left(-\frac{2}{a}\right)
+\int_0^\infty F_{a,\mcs}(\rho)d\rho\right),
	\end{eqnarray*}
	where $F_{a,\mcs}(\rho)=\eta_a(\rho)\rho^{-\mcs-\frac{1}{2}}$. By the uniform convergent properties of $\eta_a(\rho)$, $F_{a,\mcs}(\rho)$ and all its derivatives, the function $\wh{F_{a,\mcs}}\left(-\frac{2}{a}\right)$ decays rapidly as $a\to 0^+$. The integral $\int_0^\infty F_{a,\mcs}(\rho)d\rho$ converges to $\int_0^\infty\eta_0(\rho)\rho^{-\mcs-\frac{1}{2}}d\rho$ as $a\to 0^+$.
Since both $\wh{\eta_0}$ and $\wh{\psi_1}$  vanish on $[0,\infty)\bs [\frac{1}{2},2]$, we have
\begin{align*}
\left|\int_{0}^{\infty}\eta_0(\rho)\rho^{-\tau-\frac{1}{2}}d\rho\right|
=&\left|\int_{0}^{\infty}\rho^{-\tau-1/2}\widehat{\psi}_1(\rho)\left(\int_{-1}^1\widehat{\psi}_2(u)e^{-i\pi \rho u^2}du\right)d\rho\right|\\
\leq &\left|\int_{0}^{\infty}\rho^{-\tau-1/2}|\widehat{\psi}_1(\rho)|\int_{-1}^1|\widehat{\psi}_2(u)|
du d\rho\right|\\
\leq &\sqrt{2}\ \|\psi_2\|_{L^2}\int_{1/2}^{2}\rho^{-\tau-1/2}|\widehat{\psi}_1(\rho)|d\rho
\leq  \sqrt{3}\ 2^{\mcs+1/2}\|\psi_1\|_{L^2}\|\psi_2\|_{L^2}<\infty.
\end{align*}
This proves the first identity in \eqref{est:upper} in item (ii). The second identity in \eqref{est:upper} in item (ii) can be proved similarly.

	\noindent (iii) follows from the localization property of shearlet (cf. \cite[Proposition~3.4]{kutyniok2009resolution}).
	
	\noindent (iv) The result follows easily from Theorem~\ref{thm:swf:integral}.
\end{proof}

Item (ii) in Proposition~\ref{holdersobolevapp} only provides us an upper bound of   shearlet coefficients.
Following the argument of the lower bound estimates on shearlet coefficients in \cite{guo2009characterization} for detecting singularities, by appropriately choosing $\psi_1$ and $\psi_2$, we now show that there exists a positive constant $c>0$ such that
$|\mathcal{SH}_{\psi}\mathcal{B}_{\mcs}(a,s,t)| \ge c a^{\mcs+\frac{3}{4}}$ as $a\to 0^+$.
To prove this claim, by the proof of \eqref{est:upper} in item (ii),
it suffices to prove that there exists $c>0$ such that $\left|\int_0^\infty \eta_0(\rho) \rho^{-\mcs-\frac{1}{2}}d\rho \right|\ge c$. In addition to our assumption in (C1) and (C2) in Subsection~\ref{shearlet}, we now assume that $\wh{\psi_1}(\xi)>0$ for all $\xi\in (-2,-\frac{1}{2})\cup (\frac{1}{2},2)$ and
$\wh{\psi_2}$ is supported inside $[-\frac{1}{2},\frac{1}{2}]$ such that $\wh{\psi_2}(\xi)\ge 0$ for all $\xi\in [-\frac{1}{2},\frac{1}{2}]$. Then
\[
\int_{-1}^1 \wh{\psi_2}(u)e^{-\pi i \rho u^2} du=
 \int_{-\frac{1}{2}}^{\frac{1}{2}} \wh{\psi_2}(u) \cos(\pi \rho u^2)du-i \int_{-\frac{1}{2}}^{\frac{1}{2}} \wh{\psi_2}(u) \sin(\pi \rho u^2)du=\wh{\psi_3}(\rho)-i\wh{\psi_4}(\rho),
\]
where
$\wh{\psi_3}(\rho):=\int_{-1/2}^{1/2}\wh{\psi_2}(u) \cos(\pi \rho u^2)du$ and
$\wh{\psi_4}(\rho):=\int_{-1/2}^{1/2}\wh{\psi_2}(u) \sin(\pi \rho u^2)du$ for $\rho\ge 0$.
Therefore, since all functions $\wh{\psi_1},\ldots, \wh{\psi_4}$ are real-valued functions, we conclude that
\begin{align*}
\left|\int_{0}^{\infty}\eta_0(\rho)\rho^{-\tau-\frac{1}{2}}d\rho\right|
=&\left| \int_{1/2}^2 \rho^{-\tau-\frac{1}{2}} \wh{\psi_1}(\rho) \wh{\psi_3}(\rho)d \rho-i \int_{1/2}^2 \rho^{-\tau-\frac{1}{2}} \wh{\psi_1}(\rho) \wh{\psi_4}(\rho) d\rho\right|\\
\ge & \left|\int_{1/2}^2 \rho^{-\tau-\frac{1}{2}} \wh{\psi_1}(\rho) \wh{\psi_3}(\rho) d\rho\right|,
\end{align*}
For all $\rho\in [\frac{1}{2},2]$ and $|u|\le 1/2$, we have $\cos(\pi \rho u^2)\ge 0$ by $|\pi \rho u^2|<\pi/2$. Since we assumed that $\wh{\psi_2}(\xi)\ge 0$ for $\xi \in [-1/2,1/2]$, we must have $\wh{\psi_3}(\rho)\ge 0$ for all $\rho\in [\frac{1}{2},2]$. Note that $\wh{\psi_3}$ cannot be identically zero on $[\frac{1}{2},2]$.
Hence, we have
\[
\left|\int_{0}^{\infty}\eta_0(\rho)\rho^{-\tau-\frac{1}{2}}d\rho\right|
\ge \int_{1/2}^2 \rho^{-\tau-\frac{1}{2}} \wh{\psi_1}(\rho) \wh{\psi_3}(\rho) d\rho=:c.
\]
Because $\wh{\psi_1}(\rho)>0$ and $\wh{\psi_3}(\rho)\ge 0$
for all $\rho\in (\frac{1}{2},2)$, we must have $c>0$. This proves $|\mathcal{SH}_{\psi}\mathcal{B}_{\mcs}(a,s,t)| \ge c a^{\mcs+\frac{3}{4}}$ as $a\to 0^+$, which provides a lower bound for item (ii) in Proposition~\ref{holdersobolevapp}.

\bigskip
\noindent \textbf{Acknowledgment:} The authors would like to thank the reviewers for their valuable suggestions which improved the presentation of the paper.

\end{document}